\documentclass[10pt]{amsart}

\usepackage{amssymb, amsmath}
\usepackage{diagbox}

\usepackage{tikz, tikz-cd}
\usepackage{enumitem}
\usepackage{mathtools}
\usepackage{todonotes}

\usepackage{hyperref}
\hypersetup{
    colorlinks=true,
    citecolor=blue,
    linkcolor=blue,
    filecolor=magenta,      
    urlcolor=blue,
}

\theoremstyle{plain}
 \newtheorem{thm}{Theorem}[section]
 
 \newtheorem{cor}[thm]{Corollary}
\newtheorem{lem}[thm]{Lemma}
\newtheorem{prop}[thm]{Proposition}

\theoremstyle{definition}
\newtheorem{defn}[thm]{Definition}

\newtheorem{rmk}[thm]{Remark}
\newtheorem{fact}[thm]{Fact}
\newtheorem{convention}[thm]{Convention}
\newtheorem{summary}[thm]{Summary}

\theoremstyle{plain}
\newtheorem{theorem}[thm]{Theorem}

\newtheorem{lemma}[thm]{Lemma}
\newtheorem{corollary}[thm]{Corollary}
\newtheorem{proposition}[thm]{Proposition}
\theoremstyle{definition}

\newtheorem{claim}[thm]{Claim}

\newtheorem{construction}[thm]{Construction}

\newtheorem{example}[thm]{Example}

\newtheorem*{question*}{Questions}
\newtheorem{remark}[thm]{Remark}

\numberwithin{equation}{section}
\newcommand{\sB}{{\mathcal B}}
\newcommand{\sC}{{\mathcal C}}
\newcommand{\sD}{{\mathcal D}}
\newcommand{\sE}{{\mathcal E}}

\newcommand{\sH}{{\mathcal H}}

\newcommand{\sK}{{\mathcal K}}

\newcommand{\sO}{{\mathcal O}}

\newcommand{\sR}{{\mathcal R}}

\newcommand{\sX}{{\mathcal X}}

\newcommand{\C}{{\mathbb C}}

\newcommand{\G}{{\mathbb G}}

\newcommand{\N}{{\mathbb N}}

\newcommand{\Q}{{\mathbb Q}}

\newcommand{\NN}{\ensuremath{\mathbb{N}}}
\newcommand{\hol}{\ensuremath{\mathcal{O}}}

\newcommand\la{\lambda}

\newcommand\s{\sigma}

\newcommand\e{\epsilon}

\newcommand\be{\beta}

\newcommand\De{\Delta}
\newcommand\ga{\gamma}

\DeclareMathOperator{\Pic}{Pic}

\DeclareMathOperator{\Tors}{Tors}

\newcommand{\ra}{\ensuremath{\rightarrow}}

\newcommand{\CC}{\mathbb{C}}

\newcommand{\PP}{\mathbb{P}}
\newcommand{\QQ}{\mathbb{Q}}
\newcommand{\RR}{\mathbb{R}}
\newcommand{\ZZ}{\mathbb{Z}}

\newcommand{\Aut}{\mathrm{Aut}}
\newcommand{\bir}{\mathrm{bir}}
\newcommand{\Bir}{\mathrm{Bir}}

\newcommand{\Gal}{\mathrm{Gal}}

\newcommand{\GL}{\mathrm{GL}}
\newcommand{\I}{\mathrm{I}}
\newcommand{\II}{\mathrm{II}}
\newcommand{\III}{\mathrm{III}}
\newcommand{\IV}{\mathrm{IV}}
\newcommand{\id}{\mathrm{id}}
\newcommand{\im}{\mathrm{im}}

\newcommand{\irr}{\mathrm{irr}}

\newcommand{\Jac}{\mathrm{Jac}}
\newcommand{\Ker}{\mathrm{Ker}}

\newcommand{\MW}{\mathrm{MW}}

\newcommand{\Num}{\mathrm{Num}}

\newcommand{\rank}{\mathrm{rank\,}}
\newcommand{\red}{\mathrm{red}}
\newcommand{\Res}{\mathrm{Res}}
\newcommand{\SL}{\mathrm{SL}}

\newcommand{\tor}{\mathrm{tor}}
\newcommand{\Triv}{\mathrm{Triv}}

\newcommand{\tB}{\widetilde{B}}

\newcommand{\tP}{\widetilde{P}}
\newcommand{\tQ}{\widetilde{Q}}
\newcommand{\tR}{\widetilde{R}}
\newcommand{\tS}{\widetilde{S}}

\newcommand{\tX}{\widetilde{X}}

\begin{document}
\title[
 Cohomologically
trivial automorphisms of elliptic surfaces]{ 
 On the numerically and
cohomologically trivial automorphisms of  elliptic surfaces II:  $\chi(S)>0$ }
\author{Fabrizio Catanese}
\author{Wenfei Liu}
\author{Matthias Sch\"utt}
\date{\today}

\address{Lehrstuhl Mathematik VIII, Mathematisches Institut der Universit\"{a}t
Bayreuth,\linebreak NW II, Universit\"{a}tsstr. 30,
95447 Bayreuth, and  Korea Institute for Advanced Study, Hoegiro 87, Seoul, 
133-722, Korea}
\email{Fabrizio.Catanese@uni-bayreuth.de}

\address{School of Mathematical Sciences, Xiamen University, Siming South Road 422, Xiamen, Fujian Province, P.R. China}
\email{wliu@xmu.edu.cn}

\address{Institut f\"ur Algebraische Geometrie, Leibniz Universit\"at
  Hannover, \linebreak
  Welfengarten 1, 30167 Hannover, Germany\\ and
   Riemann Center for Geometry and Physics, Leibniz Universit\"at
  Hannover, Appelstrasse 2, 30167 Hannover, Germany}
\email{schuett@math.uni-hannover.de}

%\thanks{The first author acknowledges support of the ERC 2013 Advanced Research Grant - 340258 - TADMICAMT; part of this work was performed at KIAS Seoul. The second author acknowledges support of    . }



\begin{abstract}
In this second part we  study  first  the group $\Aut_{\QQ}(S)$ of  numerically trivial automorphisms of  an algebraic  properly elliptic surface $S$, that is, of a minimal algebraic surface with Kodaira dimension $\kappa(S)=1$, in the case $\chi(S) \geq 1$. Our first surprising result is that, against what has been believed for over 40 years, there exist nontrivial such  groups  for $p_g(S) >0$. Indeed, we show even  that 
 $\Aut_{\QQ}(S)$   is always a 2-generated finite abelian group,
but there is no absolute upper bound for its cardinality. 

At any rate, we give explicit  and  essentially  optimal upper
bounds for $|\Aut_{\QQ}(S)|$ in terms of the numerical invariants of $S$, as $\chi(S)$,  or the irregularity $q(S)$,  or  the bigenus $P_2(S)$.
Moreover, we  reach an  almost complete description of the possible  groups $\Aut_{\QQ}(S)$  
 and
 we give  effective  criteria for such surfaces to have  trivial $\Aut_{\QQ}(S)$.

Our second surprising results concern  the  quite elusive group $\Aut_{\ZZ}(S)$ of cohomologically trivial   automorphisms; we are able to give  the explicit upper bounds for
$|\Aut_{\ZZ}(S)|$  in special cases: $9$   when $p_g(S) =0$, and  we achieve the sharp upper bound 
 $3$ when  $S$ (i.e.,  the pluricanonical elliptic fibration) is isotrivial. 
 Also in the non isotrivial case we produce  subtle examples where $\Aut_{\ZZ}(S)$
 is a group of order $2$ or $3$.
\end{abstract}
\maketitle

\setcounter{tocdepth}{1}

\tableofcontents
\section{Introduction}

Let $X$ be a  compact connected  complex manifold. By \cite{bm1, bm2}  the automorphism group $\Aut(X)$  is a finite dimensional complex Lie Group, whose  connected  component of the identity is denoted  $\Aut^0(X)$.

We continue our investigation of two special subgroups: 

\begin{enumerate}
\item[(i)]
 $\Aut_\QQ(X)$, called the group of {\bf  numerically trivial}
automorphisms,  is the subgroup of the automorphisms acting as the identity on $H^*(X, \QQ)$, and

\item[(ii)]   $\Aut_\ZZ(X)$, the group of {\bf cohomologically trivial } automorphisms,  consists of those automorphisms 
acting as the identity on $H^*(X, \ZZ)$.
\end{enumerate}

Of course they coincide, $\Aut_\ZZ(X) = \Aut_\QQ(X)$, if $H^*(X, \ZZ)$ has no torsion.

Recall moreover, as already mentioned in Part I \cite{CFGLS24}, that in the case where $X$ is a compact K\"ahler manifold,
 $\Aut_\QQ(X) / \Aut^0(X)$ is a finite group (\cite{Lie78}, Fujiki \cite{Fuj78}).

Hence the major problems to be considered are:   to determine when   this quotient group or $\Aut_\ZZ(X) / \Aut^0(X)$  may be  nontrivial,  to give 
explicit upper bounds for their cardinality in terms of the invariants of $X$, and to try to establish sharpness of
these upper bounds  by exhibiting appropriate examples.

Since for curves (complex dimension $n=1$)  $\Aut_\QQ(X) =  \Aut^0(X)$, the first important case to consider
is the case $n=2$ (hence $X$, being a surface, shall be denoted by $S$), which we already treated  in \cite{CatLiu21} and in Part I \cite{CFGLS24}.

This paper is dedicated to the remaining case of minimal  surfaces of Kodaira dimension $1$ with $\chi(S) >0$:  
they  are automatically    K\"ahler, and for them   $\Aut^0(S)$ is trivial;  we shall deal here with the case of algebraic surfaces, leaving
the case of non algebraic surfaces for  a future paper.

The main purpose of this article is to revive the interest for this  very subtle  topic of research,
which, like the Arab Phoenix, has a rebirth from its own ashes. Indeed,  for 
over  40 years it was believed  that  these groups are trivial for $p_g(S) >0$ (Theorem 4.5 of \cite{Pe80}, confirmed by  Cor. 0.3 of \cite{Cai09}).

Recent results of Dolgachev--Martin \cite{DM22} for the  case where $S$ is  a rational surface of any characteristic
(which we shall use)  go also in the direction of investigating numerically  trivial automorphisms
(these are automatically  cohomologically trivial if the surfaces are rational).

Our results, which will be more amply illustrated later on in Subsection \ref{ss:results}, are rather complete for the case of the group
$\Aut_\QQ(S)$ of numerically trivial automorphisms.

For the intriguing case of the  group
$\Aut_\ZZ(S)$ of cohomologically trivial automorphisms we establish again, with very hard work, the existence
(as in part I) of nontrivial groups of order $2,3$: but we cannot establish an absolute upper bound for $|\Aut_\ZZ(S)|$ 
except under special assumptions,  for instance if  $p_g(S)=0$, or if $S$ is isotrivially fibred.

Hence the main open question remaining is  whether there is an absolute upper
 bound for $|\Aut_{\ZZ}(S)|$ in the non  isotrivial case with  $p_g(S) > 0$.

To explain the difficulty of the question, the problem is that, even if one shows  that an automorphism has  trivial action on $H_1(S, \ZZ) = H^3 (S, \ZZ)$,
there remains to handle the case of $H^2 (S, \ZZ)$, which has a non canonical splitting as a free part plus
a torsion subgroup, canonically  isomorphic to the torsion of $H_1(S, \ZZ)$. Hence the difficulty is to find a system
of explicit generators for $H^2 (S, \ZZ)$  in terms of  which we can then describe the action of automorphisms;  this is quite hard in general, and almost impossible  in the case where
there is nontrivial torsion and  $p_g(S) > 0$.

%Let us pass now to describe the main ideas and tools used in this paper.
\subsection{ Main ideas and tools: two allied points of view}

Throughout the paper, $f\colon S\to B$ denotes the pluricanonical fibration of 
a properly elliptic  surface with $\chi(S) : = \chi(\mathcal O_S)>0$.

As the reader shall see, there are two allied points of view to take.

The most algebraic point of view is to look at the function field extension $\CC(S) \supset \CC(B)$
as providing a curve $\sC$ of genus $1$ (the fibre of $f$ over the generic point of $B$)
defined over the non algebraically closed field $ \sK : = \CC(B)$.
This point of view  is  essential  to have a  simple algebraic picture of $\Aut(S)$: since in our case the group $\Aut(S)$
%:
equals  the group $\Bir(S)$ of birational self maps of $S$,  and $\Bir(S) =\Aut (\CC(S))=Aut(\sC)$. 

On the other hand  the  study of the topological properties of automorphisms is based on  the  relative minimal model $f\colon S \ra B$, that is, on the action on    the fibration  and on  its singular fibres;  here there is a  crucial difference for checking when automorphisms are
numerically trivial, respectively cohomologically trivial.

Turning back to the algebraic approach, we consider the subgroup $\Aut_{\sK}(\sC)$
of  automorphisms which act as the identity
on the base field $ \sK$. By base extension,
these embed into $\Aut_{ \bar{\sK}}(\sE)$, where $\sE : = \sC \otimes \bar{\sK}$.\footnote{ Indeed, it suffices to take a finite 
 field extension $\sK' \supset \sK$ to get an elliptic curve $\sC\otimes{\sK'}$.}
This being  an elliptic  curve, $\sE$ acts by translations on itself via the isomorphism 
 $ \sE \cong \Pic^1(\sE) \cong \Pic^0(\sE)$ provided by the choice of a point $0 \in \sE$
 where $\Pic^d(\sE)$ is, as usual,  the
 set of linear equivalence classes of divisors of degree $d$.
Since the base field is of characteristic zero, we infer 
 \begin{equation}\label{autell}
  \Aut_{\bar{\sK}}( \sE) = \sE \rtimes \mu_r, \quad r \in \{2,4,6\},
  \end{equation}
  where $\mu_r$ is the group of $r$th roots of unity in $\CC$,
  and $r=2$ except for the special cases $r=4$ for the Gaussian = harmonic elliptic curve,
  and $r=6$ for the Fermat = equianharmonic elliptic curve.
%Moreover, $\sE$ acts by translations via the isomorphism 
% $ \sE \cong \Pic^1(\sE) \cong \Pic^0(\sE)$ provided by the choice of a point $0 \in \sE$
% where $\Pic^d(\sE)$ is, as usual,  the
% set of linear equivalence classes of divisors of degree $d$.
 
Going back to $\sC$, there is no longer 
 %(as we were taught by Kodaira \cite{kodaira} and Shafarevich \cite[Chapter VII]{steklov})
%\footnote{  at school, we were taught 
 %to read the classics, our plead is: this should  also  be done more in mathematical education. }
  an isomorphism 
 $\sC \cong \Pic^0(\sC)$ as curves over the field $\sK$ if there is no $\sK$-rational point on $\sC$, 
 but the isomorphism $\sC \cong \Pic^1(\sC)$ holds true
  (by Riemann--Roch and by descent).
   
 Then we denote as usual $\Jac (\sC) : = \Pic^0(\sC)$, and $\sC$ is a principal homogeneous space over $\Jac (\sC)$
 (\cite{steklov}, Chapter 2 of \cite{dolg-cime}); 
 in particular, the Mordell--Weil group, the group of $\sK$-rational points of $\Jac (\sC)$
 acts on $\sC$.
 We highlight that 
 the torsion subgroup of the Mordell--Weil group will be a major source for the construction of numerically trivial automorphisms
 (see especially Section \ref{s:constr}).
 
 The geometrical counterpart of $ \Jac (\sC) $  is the Jacobian surface $ J(S)$, 
 a suitable compactification of  the sheaf of groups $\sR^1 f_* \hol_S / \sR^1 f_* \ZZ$ 
 (whose stalks have a connected component of the identity which is either an elliptic curve, or $\CC^*$, or $\CC$:
 in the latter cases one says that the fibre is of multiplicative resp.\ additive type).
 
The Jacobian surface $ J(S)$ has the property
  that the singular fibres of $J(S)$ are  the same as the singular fibres $F$ of $S$ whenever $F$ is not multiple,
 while they are  isogenous  to $F_{\red}$ when $F$ is multiple. In this way, even if the global topological behaviour of 
 $J(S)$ may differ significantly from that of $S$, at least the topological nature of the  fibres is the same.

Now, the curve $\sC$ admits also automorphisms which do not act as the identity on the field $\sK$,
and we have an exact sequence
$$ 1 \ra \Aut_{\sK}(\sC) \ra\Aut (\sC) \ra\Aut(\sK) $$
where the left hand side group contains   the Mordell--Weil group as a subgroup of  index $r \leq 6$ by \eqref{autell}.

The geometrical counterpart is the exact sequence 
\begin{eqnarray}
\label{eq:es}
1\rightarrow \Aut_B(S) \rightarrow \Aut(S)\xrightarrow{r_B} \Aut(B)
\end{eqnarray}
and we denote by $\Aut(S)|_B$ the image of the right hand homomorphism.

\subsection{ Main strategy: a dichotomy of cases}

Based on \eqref{eq:es},
the strategy is of course to give upper bounds for $|\Aut(S)|$ via upper bounds for the cardinalities  
$|\Aut(S)|_B|$ of the image group (here an essential role is given by the analysis of the singular fibres)
and  $|\Aut_B(S)|$ of  the group of automorphisms acting trivially on the base curve $B$.
 Obviously this approach can also be applied to  compute $|\Aut_\QQ(S)|$ and $|\Aut_\ZZ(S)|$.

Here, the analysis bifurcates: if the fibration is isotrivial, that is, birational to a quotient $ (C \times E)/G$
(where the action of $G$ is no longer assumed to be free), then we use many of the tools developed in Part I,
but now enhanced  (as we shall see) by the presence of reducible fibres,  which have to be of additive type.

Whereas, if the fibration is not isotrivial (that is, the smooth fibres are not all isomorphic), then the Mordell--Weil group 
 plays a dominant role: since then, in view of the finiteness of $\Aut_{\QQ}(S)$, we have to study when elements of  the torsion subgroup 
$\MW(J(S))_\tor$   yield  numerically trivial automorphisms.

One of the main advantages is that there is a huge literature on these groups (see \cite{SS19}), especially 
in the case where $J(S)$ is a rational surface (this is  the case for $p_g(S)=0$) there are explicit lists by 
Miranda--Persson \cite{MP86}, \cite{Mir89},
Oguiso--Shioda \cite{OS90} and others
describing these groups.

And when $J(S)$ is not rational,  the method of viewing   it as a pull back from an elliptic modular surface   paves the way
 to exhibiting large groups of numerically trivial automorphisms (in Section \ref{s:constr}).

Some existence results and arguments  also apply to Enriques surfaces, see Theorem \ref{thm:bc} and Remark \ref{rem:Enriques}:
the case of Enriques surfaces indeed suggests the choices of candidates for an integral basis of the second cohomology
(as we shall exploit in Section \ref{s:2})
and motivates our basic construction of numerically trivial automorphisms,  Construction \ref{constr}, which is an instrumental special case of constructing $S$ as a principal homogeneous space over a jacobian elliptic  surface.

\subsection{ Main results for numerically trivial automorphisms}
\label{ss:results}

Let $f\colon S\to B$ denote a properly elliptic   algebraic surface with $\chi(S)>0$. Our main theorems show that the group $\Aut_\QQ(S)$ of  numerically trivial automorphisms can be arbitrarily large, but in a controlled way. 
We do this by 
controlling the induced action of $\Aut_\QQ(S)$ on the base $B$ of the fibration 
and then by relating $\Aut_\QQ(S)$ to the torsion part $\MW(J(S))_\tor$ of the Mordell--Weil group of the relative Jacobian $J(S)$. 

The behaviour of $\Aut_\QQ(S)$ differs a lot according to whether $p_g(S)$ is positive or not. In the case where $p_g(S)>0$, or, equivalently, where $J(S)$ is not rational, $\Aut_\QQ(S)$ preserves each fibre and induces translations on smooth fibres, thus comes from  a subgroup of  $\MW(J(S))_\tor$, as is illustrated in Proposition~\ref{lem: AutQ vs MW}. The upshot is then
that, via our main construction of examples (Construction~\ref{const: log transform}), 
we give a flexible sufficient condition for a subgroup  $G\subset\MW(J(S))_\tor$ to be realized as a subgroup of $\Aut_\QQ(S)$.

\begin{theorem}[Groups for positive geometric genus]
\label{thm:p_g>0}
\label{thm} 
Suppose that $p_g(S)>0$.  Then the following hold:

(i) $\Aut_\QQ(S)$ is isomorphic to a subgroup of $\MW(J(S))_\tor$, and, as such, it is a finite 2-generated abelian group
which can be written in the form $ G = \ZZ/d\ZZ \oplus \ZZ/ da\ZZ$.

(ii) Conversely, for any  group $ G = \ZZ/d\ZZ \oplus \ZZ/ da\ZZ$, there is a properly elliptic surface $S$ such that $\Aut_\QQ(S)\supseteq G$. Equality $\Aut_\QQ(S)=G$ can be ensured if $a$ is square free.
 
(iii)
$\Aut_\QQ(S)$ is trivial if 
 there is a  fibre of additive type or if
 all multiple fibres of
$S\to B$ 
have smooth support; in particular, this holds true if  there are no multiple fibres, e.g.\ if
the fibration admits a section, that is,  $S \cong J(S)$.

(iv) 
$\Aut_\QQ(S)$ is trivial if the fibration is isotrivial.

\end{theorem}

By bounding the Mordell--Weil group of the elliptic modular surfaces $X(N)$ and $X_1(N)$ associated to the congruence modular groups $\Gamma(N)$ and $\Gamma_1(N)$, we obtain an upper  bound for  $ | \Aut_\QQ(S)|$ in the case $p_g(S)>0$. There are even some arithmetic relations between the prime exponents of $\Aut_\QQ(S)$ and the Euler characteristic $\chi(S)$.

\begin{thm}[Bounds for $\Aut_\QQ$ in positive geometric genus]\label{thm: bound pg>0}
\label{thm2}
Suppose  $p_g(S)>0$. Then the following hold:
\begin{enumerate}
\item[(i)] In terms of the irregularity $q(S)$, there is the global bound 
\begin{eqnarray}
\label{eq:Aut_Q<=}
|\Aut_\Q(S)|\leq 12\pi^2(q(S)+2).
\end{eqnarray}
\item[(ii)]
If $p\mid|\Aut_\Q(S)|$ for some prime $p>5$, then 
$$p_g\geq  \frac{p^2-1}{12}-\frac{p-1}2 \;\;\; \text{
and } \;\;\; \frac{p^2-1}{24}\mid\chi(S).
$$
For $p\in\{2,3,5\}$, there are properly elliptic surfaces $S$  with any given $\chi>1$ or $p_g>0$ such that $p\mid|\Aut_\Q(S)|$.
\item[(iii)]
If $(\ZZ/p\ZZ)^2\subset\Aut_\Q(S)$  for some prime $p\geq 3$,
then 
$$p_g\geq \frac 1{12}(p-3)(p^2-1) \;\;\; \text{ and } \;\;\;
 \frac{p(p^2-1)}{24}\mid\chi(S).
$$
\end{enumerate}

%\[
%e(X) = \frac{p^2-1}2, \;\; \chi(\sO_X) = \frac{p^2-1}{24}.
%\]
%More precisely,
%\begin{eqnarray*}
%q(X) &  = & g(B) \;\; = \;\;  1 + \frac{p^2-1}{24}-\frac{p-1}2\\
%p_g(X) & = &  \frac{p^2-1}{12}-\frac{p-1}2\\
%\rho(X) & = & h^{1,1}(X) \;\;  = \;\;  2 + \frac{(p-1)^2}2
%\end{eqnarray*}
\end{thm}

Details about the bounds in 
Theorem~\ref{thm: bound pg>0}
can be found in  Subsections  \ref{s:pf1-v}, \ref{s:pf1-iii}.
We emphasize that the bound in \eqref{eq:Aut_Q<=}
also holds if $p_g(S)=0$, with exception of   a few cases, see Remark \ref{rem:bounds}.
This can also be deduced from  (and improved by) our 
main theorem in the case $p_g=0$.
% see also Theorem \ref{rs} for part (ii).

\begin{theorem}[Bounds for $\Aut_\QQ$ for geometric genus zero]
\label{thm:p_g=0}
Suppose that $p_g(S)=0$. 
\begin{enumerate}
\item[(i)]
If $f$ is not isotrivial, then $|\Aut_\Q(S)|\leq 9$,  $\Aut_\Q(S)$ is abelian and 2-generated, and equality is only attained if the Jacobian $J(S)$ 
is the extremal rational elliptic surface $X_{3333}$ in the list of \cite{MP86}, and $\Aut_\Q(S) \cong (\ZZ/3\ZZ)^2$.

\item[(ii)]
 If $f$ is isotrivial, then, letting $s\geq 2$ be the number of multiple fibres,  
$$|\Aut_\Q(S)|_B|\leq s \leq P_2(S)+1
\;\;\;
\text{ and } \;\;\; |\Aut_\QQ(S)|\leq 
4
\cdot |\Aut_\QQ(S)|_B|,
$$
 but we have the overall  estimate $|\Aut_\QQ(S)| \leq   3s$.

Conversely, for any $s\in  3\NN$, there is  an elliptic  surface  
$S$ with $P_2(S)=s-1, |\Aut_\Q(S)|_B|= s$
and $|\Aut_\QQ(S)| =  3s$,
and for any $s\in\NN$, there is  an elliptic  surface  
$S$ with $P_2(S)=s-1$,  $|\Aut_\QQ(S)| = 2s$.

\end{enumerate}
\end{theorem}

\begin{question*}
(i) Can one  prove the bound $|\Aut_\QQ(S)| \leq 2s$ if $3\nmid s$?

(ii) Is it true that,  if $|\Aut_\QQ(S)| =  3s$, 
then $J(S)$ is the extremal rational elliptic surface  $X_{22}$ in the list of \cite{MP86}
(see Table \eqref{table:special})?

\end{question*}

We should point out that, for ease of presentation, the above theorem is largely simplified, in the sense
that for many isotrivial fibrations the bounds from {\it (i)} hold, cf.\ Proposition \ref{prop: pg=0 bd J(S)} (2).

\medskip

\subsection{ Main results for cohomologically trivial automorphisms}

Cohomologically trivial automorphisms are much rarer 
than numerically trivial ones.
In contrast to Theorem~\ref{thm:p_g=0}, we have a uniform bound for $\Aut_\ZZ(S)$ in the case $p_g=0.$ 
The actual existence of properly elliptic surfaces with nontrivial $\Aut_\ZZ(S)$ is also surprising in view of the 
previous statements in the literature (cf.~\cite[Theorem~4.5]{Pe80}).

\begin{theorem}[Bounds for $\Aut_\ZZ$ for genus zero]\label{thm:Aut_Z}
\label{thm:ct}
Suppose that $p_g(S)=0$. Then the following hold for $\Aut_\ZZ(S)$:
\begin{enumerate}
\item[(i)]  $\Aut_\ZZ(S)$ preserves every fibre of $f$.

\item[(ii)]
There  is a uniform bound
$$|\Aut_\ZZ(S)|\leq 9.$$  And, if equality holds, the group must be $(\ZZ/3\ZZ)^2$.
\item[(iii)]
There are non-isotrivial properly elliptic surfaces with 
$|\Aut_\ZZ(S)|=2$ and $3$.
More precisely, there is a 1-dimensional family with $|\Aut_\ZZ(S)|=3$,
and there are families of arbitrary dimension with $|\Aut_\ZZ(S)|=2$.
\end{enumerate}
\end{theorem}

This can be further improved in the isotrivial case
 (or, more generally, if there is an additive fibre, see Corollary \ref{cor:non-ss}.)
In fact, the isotrivial case allows to cover uniformly all $p_g$ (and also the case $\chi(S)=0$
by inspection of Part I).
%A more precise version of the following result is given in Theorem \ref{r<6}.

\begin{theorem}[The isotrivial case]
\label{thm5}
\label{thm:iso}
Assume that $f\colon S \to B$ is isotrivial. Then 
$$|\Aut_\ZZ(S)|\leq 3, \;\;\; {\rm and \ the \ bound \  is \ sharp}.$$
There are 
arbitrary dimensional families of isotrivial properly elliptic surfaces $S$
with $p_g(S)=0$ and with   $|\Aut_\ZZ(S)|=2$, respectively with   $|\Aut_\ZZ(S)|=3$.
\end{theorem}

A posteriori, we can  infer the following from our results in this paper combined with those from \cite{CFGLS24}:

\begin{theorem}
\label{thm:abelian}
For any algebraic properly elliptic  surface $S$, $\Aut_\ZZ(S)$ is abelian;
 and $\Aut_\QQ(S)$ is abelian  and 2-generated if $\chi(S) >0$.
\end{theorem}

\subsection{Organization of the paper}

Section \ref{s:prel} reviews preliminaries about automorphism groups, especially numerically trivial automorphisms,
and elliptic fibrations, such as singular fibres and Mordell--Weil groups.

%Section \ref{s:iso} focusses on the isotrivial case
%where we can make detailed use of the arguments of \cite{CFGLS24},
%but also take into account the reducible fibres
%(by methods which will also be used in later parts of the paper);
%this gives  a complete proof of Theorem \ref{thm5}, 
% and partially of Theorem \ref{thm:p_g=0}.
%

Starting in Section \ref{sec: Jacobian MW}, we treat elliptic fibrations in full generality,
but with a certain focus on jacobian fibrations, i.e.,  those admitting a section.
In particular, we study the interplay of $\Aut(S)$ and $\MW(J(S))$
for an elliptic surface $S$ and its Jacobian surface $J(S)$.

Section \ref{s:MW} bounds the size of the torsion subgroup of $\MW(X)$
for a jacobian elliptic surface $X$
in terms of the genus of the base curve.

This is used in Section \ref{s:upper_bounds} to prove Theorem \ref{thm2} (i)
(and eventually (ii), (iii)) as well as most of Theorem \ref{thm:ct}.

Section \ref{s:constr} introduces our main technique to construct
explicit properly elliptic surfaces with given groups of numerically trivial automorphisms.
In particular, this paves the way to the proof of Theorem \ref{thm} in Section \ref{s:pf1}.

Section \ref{s:iso} focusses on the group of numerically trivial automorphisms in the isotrivial case:
here we can make detailed use of the arguments of \cite{CFGLS24},
but also take into account the reducible fibres
(by methods which will also be used in the concluding sections of the paper);
this allows us to  complete the proof of Theorem \ref{thm:p_g=0}.

We turn to cohomologically trivial automorphisms in Sections \ref{s:2}, \ref{s:3}
by explicitly constructing the non isotrivial elliptic  surfaces required to complete the proof of Theorem \ref{thm:ct}.

Finally the isotrivial case of Theorem
 \ref{thm5} is handled extensively in Section \ref{s:iso2}
 before the paper concludes with the proof of Theorem \ref{thm:abelian} in Section \ref{s:ab}.

\begin{convention}
Throughout the paper, with the exception of subsections 2.1-2.3, we work  with algebraic surfaces  $S$ over the complex numbers,
although our constructions work over arbitrary fields
as long as the characteristic does not divide the group order.

Root lattices are assumed to be negative definite,
and $\sim$ is meant to indicate numerical equivalence of divisors.

\end{convention}

\section{Notation and preliminaries}
\label{s:prel}

\subsection{Special subgroups of the automorphism group}

 Let $X$ be a compact complex manifold, and $R$ a commutative ring, which may be $\ZZ$, or   $\QQ$. 

We  denote by $\Aut(X)$ the Lie group of   biholomorphic automorphisms of $X$, and consider its action on the cohomology ring $H^*(X, R)$.

The main concern of this paper is the study of  the kernel of such an action:
\[
\Aut_R(X):=\left\{\sigma\in\Aut(X) \mid \sigma^*\xi = \xi \text{ for any $\xi\in H^*(X, R)$}\right\}.
\]
The elements of $\Aut_{\ZZ}(X)$ is said to be  \emph{cohomologically trivial}, while those of $\Aut_{\QQ}(X)$ are said to be \emph{numerically trivial}. It is clear that
\[
\Aut_\ZZ(X)\subset\Aut_{\QQ}(X),
\]
and these coincide if $H^*(X, \ZZ)$ has no torsion.

If $X$ is a compact K\"ahler manifold, then $\Aut_{\QQ}(X)$ has finitely many components by \cite{Lie78} or \cite{Fuj78},
hence  $\Aut_\QQ(X)$ is finite unless there are vector fields on $X$.

For  an algebraic  properly elliptic surface $S$, the existence of vector fields is equivalent to  being pseudo-elliptic, which means that  $S = (C \times E)/G$ is isogenous to a product, and with $G$ acting by translations on the elliptic
curve $E$: in particular $\chi (S)=0$,
and we have  the first of the two cases  studied in Part I,  \cite{CFGLS24}.

In this paper, we will thus focus on the case $\chi (S)>0$.

Occasionally, we use the \emph{group of 
%{\red $\sO$}  \
  automorphisms trivial on holomorphic-cohomology}:
\[
\Aut_\sO(X):=\left\{\sigma\in\Aut(X) \mid \sigma^*\xi = \xi \text{ for any $\xi\in H^*(X, \sO_X)$}\right\};
\]
key examples, also for our considerations, are provided by translations by sections on elliptic surfaces.
Via the Hodge decomposition, one has the inclusion $\Aut_\QQ(X)\subset \Aut_\sO(X)$.
  
Suppose that there is a fibration $f\colon X\rightarrow B$, that is, a surjective morphism between complex spaces with connected fibres. Then we may consider the \emph{group of fibration preserving automorphisms}
\[
\Aut_f(X):=\left\{\sigma\in\Aut(X) \mid \text{$\exists\,\sigma_B\in \Aut(B)$ such that $\sigma_B\circ f = f\circ \sigma$ }\right\}.
\]
There is a natural homomorphism $r_B\colon \Aut_f(X)\rightarrow \Aut(B)$ such that, for $\sigma\in \Aut_f(X)$, its image $\sigma_B:=r_B(\sigma)$ satisfies $\sigma_B\circ f = f\circ \sigma$. The kernel 
$$\Aut_B(X):=\ker(r_B)
$$ 
is the \emph{group of fibre preserving automorphisms.}

Now let $S$ be a smooth projective surface, and $f\colon S\rightarrow B$ a relatively minimal elliptic fibration over a smooth projective curve. Suppose that the Kodaira dimension of $S$ is  nonzero.
%\wenfei{The case of negative $\kappa(S)$ is used in Proposition~\ref{prop: pg=0 bd J(S)}.}

 Then  the fibration $f$ is induced by $|mK_S|$ for some 
  nonzero integer $m$
 and hence preserved by the whole $\Aut(S)$, that is, $\Aut(S)=\Aut_f(S)$. Therefore, we have an exact sequence
\[
1\rightarrow \Aut_B(S) \rightarrow \Aut(S)\xrightarrow{r_B} \Aut(B)
\]
%For $R=\ZZ$ or $\QQ$, restricting \eqref{eq: Aut(S)} to $\Aut_R(S)$, we obtain an exact sequence\[1\rightarrow \Aut_B(S) \cap \Aut_R(S)\rightarrow \Aut_R(S) \rightarrow \Aut(B)\]
For ease of notation, denote 
\[
\Aut_{B, R}(S):= \Aut_B(S) \cap \Aut_R(S),\quad\Aut_R(S)|_B=\im( \Aut_R(S) \xrightarrow{r_B} \Aut(B)).
\]
Then we have an exact sequence
\begin{equation}\label{eq: exact sequence Aut_R}
1\rightarrow \Aut_{B, R}(S) \rightarrow \Aut_R(S) \rightarrow \Aut_R(S)|_B \rightarrow 1
\end{equation}
In order to give an upper bound for $|\Aut_R(S)|$, it suffices to bound $|\Aut_{B, R}(S)|$  and $|\Aut_R(S)|_B|$ from above.

Similarly, we introduce the notation:
\begin{equation}
\Aut_{B, \sO}(S):=\Aut_B(S)\cap \Aut_\sO(S).
\end{equation}
%By the Hodge decomposition, one sees immediately the inclusion $\Aut_\QQ(S)\subset \Aut_\sO(S)$.

\subsection{Notation and basic formulae}

Establishing some further notation,
we let 
\begin{eqnarray} 
\label{eq:f}
f\colon S \ra B
\end{eqnarray}
be a relatively minimal elliptic surface $S$,
and we assume that $\chi(S) >0$, so not all the fibres can be multiple of a smooth elliptic curve
(this is the case when $\chi(S)=0$, as treated in Part I). In this case we have that the genus of $B$  equals the irregularity 
$q(S)= h^1 (\hol_S)$.

We denote by $s$ the number of multiple fibres, and by $ m_1, \dots , m_s$ their multiplicities:
for $i=1, \dots, s$, we let $F_i$ be the corresponding multiple fibre, which can be written as $ F_i = m_i F'_i$
 where $F'_i$ cannot be written as a multiple of an effective divisor.

We have the Kodaira canonical bundle formula

\begin{equation}\label{eq: canonical bundle formula}
K_S = f^*(K_B+L) + \sum_{1\leq i\leq s} (m_i-1) F'_i
\end{equation}
where the $m_iF'_i$ are as above  the multiple fibres of $f$, and $\deg L =\chi(S)>0$. 

If $S$ has Kodaira dimension $1$, we have 
\begin{equation}
\deg(K_B+L) +\sum_{1\leq i\leq 
s
}
\left(1-\frac{1}{m_i}\right) = 2g(B)-2+\chi(S) +\sum_{1\leq i\leq 
s
}\left(1-\frac{1}{m_i}\right)>0.
\end{equation}
This formula implies that, in the case where $B $ is rational ($q(S)=0$ in our case), and $\chi(S) =1$,
then there are at least two multiple fibres.

We denote further by $s_{\red}$ the number of reducible fibres which are not multiple,
and by $s_{\irr}$ the number of irreducible and not multiple singular fibres.

Hence the total number of singular fibres equals $s' : =  s + s_{\red} + s_{\irr}$.

By the Noether formula, and since $K^2_S=0$, we have 
\begin{equation}\label{eq: e(S)}
12\chi(S) = 
e(S) = \sum_{i=0}^{ s'} e(f^* b_i)
\end{equation}
where $f^* b_i$ runs through all the (finitely many) singular fibres.

\subsection{Invariance of the fibres by a numerically trivial automorphism}

Let $\Psi \in \Aut_{\QQ}(S)$ be a numerically trivial automorphism. Then (see Part I, \cite{CFGLS24}, Section 2):

\begin{enumerate} 
\item
If $C$ is an irreducible curve with $C^2 <0$, then $\Psi (C)=C$;
\item
if there is a fibration $ f\colon S \ra B$ over a curve, then $\Psi $ preserves it, inducing an action on $B$;
\item
any fibre $F$ such that $F_{red}$ is reducible is left invariant ($\Psi (F)=F$);
\item
if a fibre is multiple of a smooth elliptic curve $ F= m F'$, and $\Psi \in \Aut_{\ZZ}(S)$,
then $\Psi (F)=F$ if $q(S) = {\rm genus} (B)$: this holds unless all the fibres are smooth elliptic curves, i.e.\ unless $\chi(S)=0$;
\item
$\Psi$ acts as the identity on $B$ if $B$ has genus $\geq 2$, or if $B$ has genus $1$ and there is an invariant fibre,
or if $B$ has genus $0$ and there are at least three  invariant fibres;
\item
if
$B$ has genus $0$ and there are two   invariant fibres, then $\Psi$ belongs to a cyclic group of automorphisms of $\PP^1$,
fixing two points.
\end{enumerate}

A first consequence is:

\begin{lemma}
\label{lem2.1}
Let $f\colon S \ra B$ be a relatively minimal  elliptic fibration, admitting a singular fibre which is not the multiple of a smooth elliptic curve 
(that is, $e(S), \chi(S) >0$): then
$$ H_1(S, \ZZ) = H_1(B, \ZZ) \oplus (\textstyle\bigoplus_{i=1}^s (\ZZ/ m_i\ZZ)\ga_i) / \ZZ (\sum_{j=1}^s \ga_j),$$
and $\Psi \in\Aut_{\ZZ}(S)$ acts trivially on $ H_1(B, \ZZ) $.

\end{lemma} 

\begin{proof}
This is an immediate consequence of Lemma 2.5 of Part I 
and of (4), (5) above.
\end{proof}

The next table collects the possible types for $F'$ in Kodaira's notation,
together with the corresponding Dynkin types obtained 
from the dual graph by omitting a single simple fibre component (as indicated in Figures \ref{Fig:In}, \ref{Fig:add}):
$$
\begin{array}{c||c|c|c|c|c|c|c|c}
\text{Kodaira type} & \I_n \,  (n>0) &
\I_n^* \, (n\geq 0) & \II & \III & \IV & \IV^* & \III^* & \II^*\\
\hline
\text{Dynkin type} & A_{n-1} & D_{n+4} & A_0 & A_1 & A_2 & E_6 & E_7 & E_8
\end{array}
$$

Here types $\I_1$ and $\II$ refer to the nodal resp.\ cuspidal cubic,
the only irreducible singular fibre types $F'$.
The reducible fibres relevant to this paper are displayed below in Figures \ref{Fig:In} and \ref{Fig:add},
with simple fibre components (i.e.\ of multiplicity one) printed thin and
labels indicating the multiplicities of the other components (printed thick).

\begin{figure}[ht!]
\setlength{\unitlength}{.30in}
\begin{picture}(10,2.7)(0,-0.3)

\thinlines

%\I_2
%\qbezier(0.6,0)(0,1)(0.6,2)
%\qbezier(0.4,0)(1,1)(0.4,2)
\put(0.9,1.8){\makebox(0,0)[l]{$\I_2$}}

\put(0.5,1){\circle{1}}
\put(.5,0){\line(0,1){2}}

%\I_3
\put(3,0.2){\line(1,0){2}}
\put(3.1,0){\line(1,2){1}}
\put(3.9,2){\line(1,-2){1}}
\put(4.4,1.8){\makebox(0,0)[l]{$\I_3$}}

%\I_n
\put(7,0.8){\line(1,2){0.6}}
\put(7,1.2){\line(1,-2){0.3}}
\put(9,1.2){\line(-1,-2){0.3}}
\put(8.4,2){\line(1,-2){0.6}}
\put(7.3,1.8){\line(1,0){1.4}}

%\put(7.77,1.9){\makebox(0,0)[l]{$\hdots$}}
\put(7.77,0.35){\makebox(0,0)[l]{$\hdots$}}
\put(6.9,-.1){\makebox(0,0)[l]{($n$ comp's)}}
\put(8.9,1.8){\makebox(0,0)[l]{$\I_n (n>3)$}}

\end{picture}
\caption{Reducible multiplicative fibres (type $\I_n\; (n>1)$)}
\label{Fig:In}
\end{figure}

\begin{figure}[ht!]
\setlength{\unitlength}{.35in}
\begin{picture}(10,5.5)(-1,-0.3)

\thinlines

%\III
\put(.5,4){\circle{1}}
\put(0,3){\line(0,1){2}}

\put(0.7,4.8){\makebox(0,0)[l]{$\III$}}

\thinlines
%\IV
\put(2.8,4){\line(1,0){1.6}}
\put(3.1,3){\line(1,2){1}}
\put(3.1,5){\line(1,-2){1}}
\put(4.3,4.8){\makebox(0,0)[l]{$\IV$}}

\thinlines
%\IK_0^*

\put(6.4,4.7){\line(4,1){1}}
\put(6.4,4.4){\line(4,1){1}}
\put(6.4,3.3){\line(4,-1){1}}
\put(6.4,3.6){\line(4,-1){1}}

\thicklines
\put(6.6,5){\line(0,-1){2}}
\put(7.7,4.8){\makebox(0,0)[l]{$\I_0^*$}}
\put(6.45,5){\tiny 2}

\thinlines
%\IV*
\put(.6,1.75){\line(4,-1){.8}}
\put(.6,1.15){\line(4,-1){.8}}
\put(.6,0.55){\line(4,-1){.8}}

\thicklines
\put(.2,2){\line(0,-1){2}}
\put(0,1.5){\line(4,1){1}}
\put(0,0.9){\line(4,1){1}}
\put(0,0.3){\line(4,1){1}}

\put(1.35,1.9){\makebox(0,0)[l]{$\IV^*$}}
\put(0,1.9){\makebox(0,0)[l]{{\tiny 3}}}
\put(-.1,1.3){\tiny 2}
\put(-.1,0.7){\tiny 2}
\put(-.1,.1){\tiny 2}

\thinlines

%\III*
\put(3.4,1.6){\line(1,0){0.6}}
\put(3.4,0.4){\line(1,0){0.6}}

\thicklines
\put(3.2,2){\line(0,-1){2}}
\put(3.05,2.1){\makebox(0,0)[l]{{\tiny 4}}}
\put(3.62,2.1){\makebox(0,0)[l]{{\tiny 2}}}
\put(3.65,-.1){\tiny 2}

\put(3,1.8){\line(1,0){0.8}}
\put(2.9,1.65){\makebox(0,0)[l]{{\tiny 3}}}

\put(2.9,.8){\tiny 2}

\put(3,1){\line(1,0){0.7}}
\put(3,0.2){\line(1,0){0.8}}
\put(2.9,0.35){\makebox(0,0)[l]{{\tiny 3}}}

\put(3.6,2){\line(0,-1){0.6}}
\put(3.6,0){\line(0,1){0.6}}

\put(4.3,1.8){\makebox(0,0)[l]{$\III^*$}}

\thinlines
%\II*
\put(7.15,0.4){\line(1,0){0.45}}

\thicklines
\put(6.6,2){\line(0,-1){2}}
\put(6.63,2.1){\makebox(0,0)[l]{{\tiny 6}}}
\put(7.03,2.1){\makebox(0,0)[l]{{\tiny 2}}}

\put(6.4,1.8){\line(1,0){0.8}}
\put(6.3,1.65){\makebox(0,0)[l]{{\tiny 4}}}
\put(6.3,.95){\makebox(0,0)[l]{{\tiny 3}}}

\put(6.4,1.1){\line(1,0){0.7}}
\put(6.4,0.2){\line(1,0){0.75}}
\put(6.3,0.35){\makebox(0,0)[l]{{\tiny 5}}}

\put(7,2){\line(0,-1){0.6}}

\put(7,0){\line(0,1){0.8}}
\put(7.05,-0.05){\makebox(0,0)[l]{{\tiny 4}}}
\put(7.3,0.15){\makebox(0,0)[l]{{\tiny 2}}}

%\put(3.4,1.6){\line(1,0){0.6}}
\put(6.8,0.6){\line(1,0){0.7}}
\put(7.55,0.65){\makebox(0,0)[l]{{\tiny 3}}}

\put(7.3,0.7){\line(0,-1){0.45}}

\put(7.5,1.8){\makebox(0,0)[l]{$\II^*$}}

%
%\thinlines
%%\IK_1^*
%\put(6.4,1.7){\line(1,0){1}}
%\put(6.2,1.4){\line(1,0){1}}
%\put(6.4,0.3){\line(1,0){1}}
%\put(6.2,0.6){\line(1,0){1}}
%
%\thicklines
%\put(6.2,0.9){\line(1,2){0.5}}
%\put(6.2,1.1){\line(1,-2){0.5}}
%
%\put(6.75,0.05){\makebox(0,0)[l]{{\tiny 2}}}
%\put(6.75,1.95){\makebox(0,0)[l]{{\tiny 2}}}
%
%\put(7.6,1.8){\makebox(0,0)[l]{$\IK_1^*$}}
%
%
%\thinlines
%%\IK_n^*
%\put(9.4,1.8){\line(1,0){1}}
%\put(9.2,1.5){\line(1,0){1}}
%\put(9.4,0.2){\line(1,0){1}}
%\put(9.2,0.5){\line(1,0){1}}
%
%\thicklines
%\put(9.2,1.2){\line(2,3){0.5}}
%\put(9.2,0.8){\line(2,-3){0.5}}
%
%\put(9.75,0.05){\makebox(0,0)[l]{{\tiny 2}}}
%\put(9.75,1.95){\makebox(0,0)[l]{{\tiny 2}}}
%
%
%\put(10.6,1.8){\makebox(0,0)[l]{$\IK_n^*\;(n>1)$}}
%\put(9.32,1.1){\makebox(0,0)[l]{$\vdots$}}
%\put(9.5,1){\makebox(0,0)[l]{\scriptsize{$(n+1)$ double comp's}}}
%
%

\end{picture}
\caption{Reducible additive singular fibres}
\label{Fig:add}
\end{figure}

 Recall that the multiple fibres can only be of type $_m\I_k$, with $ m \geq 2$, and  $k \geq 0$.
The singular fibres  of type $\I_k \, (k>0)$ are called semistable (or multiplicative), and all the others are said to be
of additive type;
 this refers to the connected component of the identity of the algebraic group of which  
 $(F')^\#$ (the smooth locus of the reduced singular fibre
$F'$) is a principal homogeneous space
as this equals the multiplicative group $\G_m \cong \CC^*$ resp.\ the additive group $\G_a\cong\CC$.

The fibration is said to be semistable if all the fibres are semistable.
 
Taking multiplicities into account, 
the only singular fibres above with $F_{red}$ irreducible are  
$F = \phantom{}_m\I_0 \, (m \geq 2)$, $F= \phantom{}_m\I_1 \, (m\geq 1)$ or $F = \II$. Summing up,

\begin{cor}\label{base-action}
The action of $\Psi \in\Aut_{\QQ}(S)$ on the base $B$ is trivial if $B$ has genus $\geq 2$,
and also  if $B$ has genus $1$, unless all singular fibres are of type $_m\I_0$, $_m\I_1$ or $\II$.

Assume that $\chi(S)>0$.
If $\Psi \in\Aut_{\ZZ}(S)$ does not act trivially on $B$, then either $B$ has genus $1$ and all singular fibres
are of type $\I_1$ or $\II$, or $B$ has genus $0$ and there are at most two singular fibres not of type  $\I_1$ or $\II$.
%(nodal, resp. cuspidal cubic).
%
\end{cor}

\begin{remark}
In the next subsection, we decidedly turn to elliptic surfaces which are algebraic.
In a sequel to this article, we shall consider the non-algebraic case, the main subtlety being that the Jacobian
of an elliptic surface with multiple fibres has been defined by Kodaira and Shafarevich only for algebraic elliptic surfaces. 
\end{remark}

\subsection{Singular fibres and Mordell--Weil group}
\label{ss:MW}

If $\chi(S)>0$, as assumed throughout this paper,
then necessarily there is a singular fibre %$F$ %=mF' \, (m\geq 1)$ 
with non-smooth support
by inspection of the Euler number formula \eqref{eq: e(S)}.

Note that the rank of the Dynkin diagram (assumed to be negative-definite by convention) equals one less than the number of fibre components
(which we denote by $m_v$ if $F=f^*v$, slightly abusing notation).

If the fibration \eqref{eq:f} admits a section, we call the elliptic surface $S$ jacobian
(note that this excludes multiple fibres).
Then the sections correspond bijectively to the rational points on the generic fibre $\mathcal C$;
after choosing a zero section denoted by $O$ or $O_S$, the sections thus form an abelian group, 
called the Mordell--Weil group and denoted by $\MW(S)$,
which is finitely generated if $\chi(S)>0$.
We refer to \cite{SS19} for more information on the theory of Mordell--Weil groups.

The Picard number of $S$ can be  expressed by the Shioda--Tate formula:
\begin{eqnarray}
\label{eq:ST}
\rho(S) = 2 + \rank\MW(S) + \sum_{v\in B} (m_v-1).
\end{eqnarray}
Here the last summands correspond exactly to the Dynkin diagrams of the reducible fibres;
usually one omits the fibre component intersecting $O$,
often called identity component and denoted by $\Theta_0$.
This paves the way towards the representation of the trivial lattice $\Triv(S)$
generated by fibre components and sections as an orthogonal sum
\begin{eqnarray}
\label{eq:Triv}
\Triv(S) = \langle O, F\rangle \oplus \bigoplus_{v\in B} (\text{Dynkin type associated to } F_v).
\end{eqnarray}
Note that the leftmost summand is unimodular of parity determined by $O^2=-\chi(S)$.

The choice of zero section $O$ fixes a group structure on the smooth fibres which carries 
over to the smooth locus $F^\#$ of any singular fibre.
We note the non-canonical isomorphism
\begin{eqnarray}
\label{eq:FFF}
F^\# \cong 
\begin{cases}
\mathbb G_m \times A_F, & \text{ if $F$ is multiplicative (type $\I_n$),}\\
\mathbb G_a \times A_F, & \text{ if $F$ is additive (all other types)}
\end{cases}
\end{eqnarray}
Here $A_F$ denotes the discriminant group of the Dynkin type $D$ corresponding to the fibre,
i.e.\ $A_F=D^\vee / D$.
Note that $A_F$ is a finite abelian group which corresponds bijectively to the simple fibre components (as indicated in Figures \ref{Fig:In}, \ref{Fig:add});
in fact, $A_F$ is cyclic unless $F$ has type $\I_n^*$ for $n\in2\NN_0$ (whence $A_F \cong (\ZZ/2\ZZ)^2$).
In addition,  $A_F$ is also endowed with the structure of a quadratic form,
called discriminant form, taking values in $\QQ/2\ZZ$. 

Throughout this paper, it will be instrumental to consider torsion sections,
and the automorphisms of $S$ given by translations (see \eqref{eq: tMW}).
Note that torsion sections are always disjoint from each other
(generally true unless the characteristic of the ground field divides the order of the section, cf.\ \cite[Prop.\ 6.33 (v)]{SS19}).
In terms of lattices, torsion sections can be encoded in the intersection pattern with the singular fibres as
\begin{eqnarray}
\label{eq:pattern}
\MW(S)_\tor \cong \Triv(S)'/\Triv(S) \subset \bigoplus_{v\in B} A_{F_v}
\end{eqnarray}
where $\Triv(S)'=(\Triv(S)\otimes\QQ)\cap\Num(S)$ denotes the primitive closure (or saturation) of $\Triv(S)$ inside $\Num(S)$.
%In practice, this sees us searching for a dual vector of square $-2\chi(S)$ in order to find a torsion section.

The jacobian elliptic fibration is called extremal if the Picard number attains the maximum value $\rho(S)=h^{1,1}(S)$
while $\MW(S)$ is finite. Equivalently, $\rank \Triv(S) = h^{1,1}(S)$.
For classifications, see \cite{MP86} and \cite{SZ01}.

With respect to the geometric group structure, it follows that there is an injection
\begin{eqnarray}
\label{eq:MW->F}
\MW(S)_\tor \hookrightarrow F^\#
\end{eqnarray}
at any fibre.
In general, this implies that $\MW(S)$ can be generated by at most 2 elements.
At an additive fibre, \eqref{eq:FFF} forces \eqref{eq:MW->F} to specialize to
\begin{eqnarray}
\label{eq:FF}
\MW(S)_\tor \hookrightarrow A_F.
\end{eqnarray}
Note that this implies that $\MW(S)_\tor=\{O\}$ if there is a fibre of type $\II$ or $\II^*$,
and that  $\MW(S)_\tor$ is cyclic if there is an additive fibre of type different from $\I_n^*$ for $n\in2\NN_0$.

With a view towards numerically trivial automorphisms, we record the following essential property:

\begin{fact}
\label{fact:MW-add}
\label{fact}
At an additive fibre, 
any two distinct torsion sections meet different  components.
\end{fact}
Indeed, if two torsion sections $P$ and $Q$ were to meet the same  component of an additive fibre $F$,
then their difference $P-Q$ would meet the identity component $\Theta_0$.
But $\Theta_0\cap F^\#\cong \G_a$ admits no non-zero torsion,
so $P$ and $Q$ intersect on $F$, contradicting the property that torsion section are disjoint (or \eqref{eq:FF}).

Turning back to an arbitrary elliptic surface $S$ with $\chi (S) >0$, 
the relative Jacobian over $B$ associates to $S$ a jacobian elliptic surface
which we denote by $J(S)$, or to keep track of  the induced fibration,
$$J(f)\colon \;\; J(S)\rightarrow B.
$$
Intriguingly, $S$ and $J(S)$ share many invariants, such as Betti numbers, Euler number, geometric genus, Euler characteristic,
 types of singular fibres (up to multiplicity, i.e.\ if $S$ has a singular fibre $F=mF'$ with $m\geq 1$ and $F'$ indivisible,
 then $J(S)$ has a fibre of type $F'$),
 and also the Picard number (cf.\ \cite[\S 5.3]{CD}).
 It follows that \eqref{eq:ST} holds for $S$ with $\MW(S)$ replaced by $\MW(J(S))$.
 Another role of $\MW(J(S))$, this time for the automorphism group of $S$,
 will be explored in Section \ref{sec: Jacobian MW}.
 
\subsection{Local behavior of an automorphism around an $m\I_k$ fibre}

In studying a fibration preserving automorphism, it is important to control the action around a fixed fibre. 
This kind of  detailed local analysis has been used rather  effectively by   \cite{Cai09} in order  to give a global bound on the number of automorphisms 
acting trivially on the second cohomology group of an elliptic surface. 
However, there is a critical inaccuracy which the following remark records.

\begin{remark}
As our construction will show, 
the crucial \cite[Lemma~1.5 (iii)]{Cai09} is not correct (cf.\ Remark \ref{rem:Cai}).
The technical flaw in its proof seems to occur in the third paragraph on page 4232 of that paper: 
in Cai's notation, the automorphism $\bar\alpha$ of $\bar S_\Delta$, which was constructed by Kodaira, 
does not necessarily descend to $S_\Delta$.
\end{remark}

As the reader can see, in Theorem \ref{thm:p_g>0} (iii), we show that 
$\Aut_\QQ(S)$ is trivial if $p_g(S) > 0$ and all multiple fibres of
$S\to B$  have smooth support (in particular,  if  there are no multiple fibres, e.g.\ if
$S$ is a jacobian elliptic surface).

In a sense, this result rescues Cai's theorem up to some extent, showing that the only trouble occurs
from the multiple fibres of type $_m\I_k$. It is therefore interesting, also for future applications, to investigate which
results hold true in the local analysis for these fibres.

Hence  we provide in this subsection two nontrivial results in the flavour of \cite[Lemma~1.5]{Cai09}.

\begin{lem}\label{lem: mI0}
Let $f\colon S\rightarrow \Delta$ be a relatively minimal elliptic fibration over the unit disk such that the central fibre $F_0=f^*0$ is the only singular fibre. Let $\sigma\in \Aut(S)$ be an automorphism of finite order preserving the fibration structure, so that  it maps a fibre to another fibre and induces an automorphism $\sigma_\Delta\in \Aut(\Delta)$. Suppose that the following holds:
\begin{enumerate}
\item $\sigma$ and $\sigma_\Delta$ have the same order.
\item $F_0$ is a multiple fibre of type $m\I_0$.
\end{enumerate}
Then $\sigma$ acts on $(F_0)_\red$ by translations. In particular, we have $e(F_0^\sigma)=e(F_0/\sigma)=0$.
\end{lem}

\begin{proof}
Let $r$ be the order of $\sigma$ and $\sigma_\Delta$. Then  we can form the following diagram
\[
\begin{tikzcd}
\tS\arrow[rd, "\tilde f"'] \arrow[r, "\nu"]& S\times_{\Delta, \phi} \Delta \arrow[r, "\Phi"] \arrow[d] & S\arrow[d, "f"] \arrow[r, "\Psi"]& \bar S:=S/\langle\sigma\rangle \arrow[d, "\bar f"]\\
&  \Delta\arrow[r, "\phi"] & \Delta\arrow[r,"\psi"] & \Delta \cong \Delta/\langle\sigma_{\Delta}\rangle
\end{tikzcd}
\]
where,  after choosing a coordinate $z$ in which the action of $\sigma_\De$ is linear, $\phi(u)=u^m$, $\psi(z) = z^r$,  and $\nu$ is the normalization. When restricted to $\Delta\setminus\{0\}$, the horizontal maps are covering maps:
\[
\begin{tikzcd}
\tS \setminus \tilde f^{-1}(0)\arrow[d, "\tilde f"'] \arrow[r, "\Phi"]& S\setminus f^{-1}(0)\arrow[d, "f"] \arrow[r, "\Psi"]& \bar S \setminus \bar f^*(0)\arrow[d, "\bar f"]\\
\Delta\setminus\{0\}\arrow[r, "\phi"] & \Delta\setminus\{0\}\arrow[r, "\psi"] &\Delta \setminus\{0\}
\end{tikzcd}
\]
Using the universal property of fibre products, one has an induced morphism 
\[
\mu\colon \tS \setminus \tilde f^{-1}(0) \rightarrow \bar S \setminus \bar f^*(0) \times_{\Delta, \psi\circ\phi} \Delta.
\]
Since $\deg\Phi = |\sigma| = |\sigma_\Delta| = \deg\phi$, one sees that $\deg\mu =1$ and hence $\mu$ is an isomorphism. 

Since $\pi_1(\Delta\setminus \{0\})\cong \ZZ$, any of its covering spaces is regular with cyclic Galois group. In our case, the covering maps $\psi$, $\phi$, and $\psi\circ\phi$ are Galois with Galois groups:
\[
\Gal(\psi) = \langle \sigma_\Delta\rangle\cong \ZZ/r\ZZ, \,\quad \Gal(\phi)\cong \ZZ/m\ZZ \cong r \ZZ/mr \ZZ< \Gal(\psi\circ\phi)\cong \ZZ/m r\ZZ
\]
It follows that the corresponding covering maps $\Psi$, $\Phi$, $\Psi\circ\Phi$ are also Galois with with Galois groups:
\[
\Gal(\Psi) = \langle\sigma\rangle\cong \ZZ/r\ZZ, \,\quad\Gal(\Phi)\cong \ZZ/m\ZZ < \Gal(\Psi\circ\Phi)\cong \ZZ/mr\ZZ.
\]
Consequently, the branched covering $\tS \rightarrow \bar S$ enjoys the same property.

Note that the central fibre $\tilde F_0$ of $\tilde f$ is a smooth elliptic curve. Since $F_0$ is of type $m\I_0$, the group $\Gal(\Phi)\cong \ZZ/m\ZZ$ acts on $\tilde F_0$ by translations. It is now evident that the group $\Gal(\Psi\circ\Phi)\cong \ZZ/mr\ZZ$ necessarily acts on $\tilde F_0$ by translations. It follows that $\Gal(\Psi) =\langle \sigma\rangle$ acts on the quotient $F_0$ by translations.
\end{proof}

\begin{remark}
Assume that we have a singular fibre of type $mI_k$ of an elliptic fibration, and with $ k \geq 1$. 
Then the automorphism group of the reduced fibre equals $(\CC^*)^k \rtimes D_k$,
where $D_k$ is the dihedral group (of order $2k$), sitting in an exact sequence 
$$ 1 \ra \ZZ/k\ZZ \ra D_k \ra \ZZ/2\ZZ \ra 1.$$

\end{remark}

\begin{lem}\label{lem: mIk}
Let $f\colon S\rightarrow \Delta$ be a relatively minimal elliptic fibration over the unit disk such that the central fibre $F_0=f^*0$ is the only singular fibre. Let $\sigma\in \Aut(S)$ be an automorphism of finite order preserving the fibration structure, so it maps a fibre to another fibre and induces an automorphism $\sigma_\Delta\in \Aut(\Delta)$. Suppose that the following holds:
\begin{enumerate}
\item $\sigma$ and $\sigma_\Delta$ have the same order.
\item $F_0$ is a multiple fibre of type $m\I_k$, with $k\geq 0$.
\end{enumerate}
Then $\sigma$ acts 
 in such a way that the Euler characteristic of the  fixed point set of $\s$   on $(F_0)_\red$  equals $k$ or $0$,
 and the first alternative holds if $\s$ preserves the components of $(F_0)_\red$.
 \end{lem}

\begin{proof}
Since the settings are almost identical, we can follow the proof of Lemma \ref{lem: mI0}, employing the same notation.
 Denote by $\tau$ the generator of $\Gal(\Psi\circ\Phi)\cong \ZZ/mr\ZZ$ ($\tau$ induces $\s$). 

Note that the central fibre $\tilde F_0$ of $\tilde f$ is a reduced fibre of type $\I_{mk}$. 

Since $F_0$ is of type $m\I_k$, the group $\Gal(\Phi)\cong \ZZ/m\ZZ$ acts on $\tilde F_0$ 
 mapping to $(\CC^*)^{mk} \rtimes D_{mk}$ in such a way that its projection in $  D_{mk}$
is $ k \ZZ / mk \ZZ  <  \ZZ/ mk \ZZ < D_{mk}$.

The group the group $\Gal(\Psi\circ\Phi)\cong \ZZ/mr\ZZ$, being a cyclic group,
maps to $ \ZZ/ mk\ZZ < D_{mk}$ (else, since $\ZZ/m\ZZ <   \ZZ/mr \ZZ$, then $\ZZ/m\ZZ $ and a reflection would not generate a cyclic group, contradicting that every subgroup of a cyclic group is  cyclic).

Hence we can write $\tau = \rho \tau'$, where $\tau'$ is a cyclical permutation; and since $\tau$ lies over the transformation $ u \mapsto \e u$, where $\e$ is a primitive $mr$-th root of unity, it follows
that $(\rho_{i+1} \rho_i^{-1} )=\e $. Since we have a cycle of curves, the previous equation implies $\e^{mk} =1$,
this means that $r | k$. 

 $\s$ acts on the set of components of the singular fibre $F_0$ via a rotation. If this rotation is nontrivial,
 then there are no fixed points on $F_0$.
 
 If instead $\s$ preserves all the components, we know that it lies in $(\CC^*)^k$,
 hence it fixes all the nodes. 
 
 If $\s$ fixes another point, then it would act as the identity  on some component of $(F_0)_{\red}$,
 and there the Euler number of the fixed point set equals $e (\PP^1)=2$.
 
It suffices now to show that $\s$ does not act as the identity on  two neighbouring components.

If in fact $\s$  acts on the $i$th  component via $\be_i$, then, since $\s (z) = \tau(z) = \e^m z$,
in view of the fact that  in local coordinates  at a node of $(F_0)_{\red}$ it acts via multiplication by 
$(\be_{i+1} , \be_i^{-1} )$ we reach a contradiction by assuming that $\be_{i+1}  = \be_i = 1$.

First because  we should  have $(\be_{i+1}  \be_i^{-1} )^m = \e^m$, absurd;
second because $\s$,  acting as the identity on the tangent space a fixed point, must be the identity.
\end{proof}

\begin{remark}
\label{rem:Cai}
Contrary to the statement in \cite[Lemma~1.5 (iii)]{Cai09},
the fixed point set may contain not only points, but also fibre components, 
as will be inherent in Construction \ref{constr}.
In fact, we shall exploit this extensively when exhibiting cohomologically trivial automorphisms in Sections \ref{s:2}, \ref{s:3},
cf.\ especially the argument in Subsection \ref{ss:pf-lem}.
\end{remark}

\section{Jacobian elliptic fibrations and their Mordell--Weil groups}\label{sec: Jacobian MW}

Let $f\colon S\rightarrow B$ be a relatively minimal elliptic surface 
with $\chi(S)>0$
and let $J(S)$ be its Jacobian. 

As mentioned in the introduction, 
 the function field extension $\CC(S) \supset \CC(B)$
provides  a curve $\sC$ of genus $1$ (the fibre of $f$ over the generic point of $B$)
defined over the non algebraically closed field $ \sK : = \CC(B)$.
The subgroup of $\Aut(S)$ of automorphisms preserving the fibration and inducing the identity on $B$ 
is isomorphic to the  subgroup $\Aut_{\sK}(\sC)$
of the automorphisms of $\sC$ which act as the identity
on the base field $ \sK$.

 These embed in $\Aut_{ \bar{\sK}}(\sC \otimes \bar{\sK})$, where $\sE : = \sC \otimes \bar{\sK}$ 
 is an elliptic  curve  over an algebraically closed field of characteristic zero 
 %\footnote{ Indeed, it suffices to take a finite 
 %field extension $\sK' \supset \sK$ to get an elliptic curve}, 
 hence 
 \begin{equation*}%\label{autell}
  \Aut_{\bar{\sK}}( \sE) = \sE \rtimes \mu_r, \quad r \in \{2,4,6\},
  \end{equation*}
 where  $\mu_r$ is the group of $r$th roots of unity in $\CC$, and $\sE$ acts by translations via the isomorphism 
 $ \sE \cong \Pic^1(\sE) \cong \Pic^0(\sE)$ provided by the choice of a point $0 \in \sE$ ($\Pic^d(\sE)$ being  as usual  the
 set of linear equivalence classes of divisors of degree $d$).

We denote as usual $\Jac (\sC) : = \Pic^0(\sC)$, and $\sC$ is a principal homogeneous space over $\Jac (\sC)$
 (\cite{steklov}, Chapter 2 of \cite{dolg-cime}), in particular, the Mordell--Weil group, the group of $\sK$-rational points of $\Jac (\sC)$
 acts on $\sC$. 
 
 In fact, the automorphisms in $\Aut_{\sK}(\sC)$ are the elements of $ \Aut_{\bar{\sK}}( \sE)$ which are left fixed by
 the action of the Galois group of the extension $\sK \subset \bar{\sK}$, hence the subgroup which has trivial image to
 $\mu_r$ consists of the Galois -invariant elements in $ \sE $, and these are the $\sK$ rational points of $ \Jac (\sC) $.

 The geometrical counterpart of $ \Jac (\sC) $  is the Jacobian surface $ J(S)$.
 
 By specializing to the fibre $F_b$ over a general point, we see that the condition of  having trivial image to $\mu_r$ means 
 that the restricted automorphism consists of a translation on $F_b$.
 
 Hence we have:

\begin{lem}\label{lem: im(t) vs trans}

Let $f\colon S\rightarrow B$ be a relatively minimal elliptic surface with $\chi(S)>0$. Then 
 there is an injective homomorphism 
\begin{equation}\label{eq: tMW}
t = t_S\colon \MW(J(f))\rightarrow \Aut_B(S),
\end{equation}
such that 
\[
\im(t) = \{\sigma\in \Aut_B(S) \mid \text{$\sigma|_F$ is a translation for a general fibre $F$}\}.
\]
\end{lem}

\begin{lem}[{\cite[Theorem~3.3]{DM22}}]\label{lem: Jacobian}
Let $f\colon S\rightarrow B$ be a relatively minimal elliptic surface with $\chi(S)>0$. Then there is a homomorphism 
\begin{equation}\label{eq: Jacobian}
\Phi\colon \Aut_f(S)\rightarrow \Aut_{J(f)}(J(S))
\end{equation}
such that $\Ker(\Phi) = t(\MW(J(f)))$ and, for any $\sigma\in \Aut_f(S)$, the following holds:
\begin{enumerate}[leftmargin=*]
\item Both $\sigma$ and $\Phi(\sigma)$ induce the same automorphism of $B$.
\item $\Phi(\sigma)$ preserves the zero section of $J(f)\colon J(S)\rightarrow B$.
\item If $\sigma\in \Aut_\QQ(S)$, then $\Phi(\sigma)\in\Aut_\QQ(J(S))$.
\end{enumerate}
\end{lem}

 The automorphism $\Phi$ in the previous Lemma is just   the natural action of $\Aut(\sC)$ on $Jac(\sC) = Pic^0(\sC)$.

\begin{lem}\label{lem: add fib}
Let $f\colon S\rightarrow B$ be a relatively minimal elliptic surface with $\chi(S)>0$. Suppose that $f$ has a fibre of additive type. Then a nontrivial numerically trivial automorphism of $S$ does not induce fibrewise translations, that is,
\[
\Aut_{\QQ}(S) \cap t(\MW(J(f))_\tor) = \{\id_S\}
\]
\end{lem}
\begin{proof}
Let $f^*b$ be a fibre of additive type. 
Since $\MW(J(S))_\tor$ injects into the group of reduced components $A_{f^*b}$ by \eqref{eq:FF},
it follows that an automorphism $t(P)\in t(\MW(J(S))_\tor)$ is either trivial 
(e.g. if there is a fibre of type $\II$ or $\II^*$)
or it permutes the simple components of $f^*b$ nontrivially,
cf.\ Fact \ref{fact:MW-add}. In the latter case,
 $t(P)$ is obviously not numerically trivial. Therefore, we obtain the equality of the lemma. 
\end{proof}

The next proposition establishes an important step by relating the image of the homomorphism $t$ from \eqref{eq: tMW}
to the group $\Aut_{B, \sO}(S)$ of automorphisms of $S$
which act trivially on the base $B$ and on $H^*(S,\sO_S)$:

\begin{prop}\label{lem: AutQ vs MW}
Let $f\colon S\rightarrow B$ be a relatively minimal elliptic surface with $\chi(S)>0$. If $p_g(S)>0$, then we have
\[
\Aut_{B, \sO}(S) =\im(t), \quad \Aut_{B, \QQ}(S) \subset t(\MW(J(f))_\tor).
\]
\end{prop}

\begin{proof}
By \cite[Lemma~2.9]{DL23}, $t(P)$ acts trivially on $H^0(S, K_S) = H^2(S, \sO_S)^{\vee}$, and hence also trivially on the transcendental part of $H^2(S, \QQ)$.\footnote{ One may also argue like this: The quotient map $S\rightarrow S/t_P|_S$ is an isogeny of elliptic surfaces, so $p_g(S)=p_g(S/t_P|_S)$ and hence $t_P|_S$ acts trivially on $H^0(S, K_S)$.}

 Since $\chi(S)>0$, the pullback $f^*\colon H^1(B, \sO_B)\rightarrow H^1(S, \sO_S)$ is an isomorphism. 
 Since $t(P)\in \Aut_B(S)$ induces the trivial action on $B$, one sees that $t(P)$ also acts trivially on $H^1(S, \sO_S)$. It follows that $t(P)\in \Aut_{B, \sO}(S)$ and hence  $\im(t)\subset \Aut_{B, \sO}(S)$.

For the other inclusion, take any $\sigma\in  \Aut_{B, \sO}(S)$. 
Assuming that $\sigma$ is not a translation on a general fibre $F$, 
it necessarily fixes some points on $F$ and is thus of finite order. 
Hence we can argue with the quotient surface $S/\langle\sigma\rangle$, 
which is a $\PP^1$-fibration over $B$ whence $p_g(S/\langle\sigma\rangle)$ vanishes.
On the other  hand, by the choice of $\sigma$, we have $p_g(S/\langle\sigma\rangle)= p_g(S)>0$.
This gives the required contradiction. 

Therefore, $\sigma$ induces a translation on the general fibres of $f$. By Lemma~\ref{lem: im(t) vs trans}, $\sigma$ comes from $\MW(J(S))$.

Finally, since $\Aut_{B, \QQ}(S)$ is a finite subgroup of $\Aut_{B, \sO}(S)$
 and $t$ is injective, $\Aut_{B, \QQ}(S)$ lies in the image of the torsion subgroup $\MW(J(f))_\tor$ of $\MW(J(f))$.
 \end{proof}

 Combining Lemma~\ref{lem: add fib} and Proposition \ref{lem: AutQ vs MW}, we obtain
 
 \begin{cor}\label{cor: pg>0 add fib 1}
 Let $f\colon S\rightarrow B$ be a relatively minimal elliptic surface with $\chi(S)>0$. 
 If $p_g(S)>0$ and if $f$ admits a fibre of additive type, then $\Aut_{B,\QQ}(S)$ is trivial.
 \end{cor}

 We conclude this section by recording the following consequences for jacobian fibrations.
 
\begin{lem}\label{lem: MW finite}
Let $h\colon X\rightarrow B$ be a relatively minimal elliptic surface admitting  a section and with $\chi (X)>0$.
If $\Aut_{B,\QQ}(X)$ is nontrivial, then 

\begin{enumerate}
\item
$\MW(X)$ is a finite group, that is, $\MW(X)=\MW(X)_\tor$;
\item
$p_g(X)=0$, i.e.\ $X$ is rational.
\end{enumerate}

\end{lem}
\begin{proof}
(1) Note that the numerical classes of two distinct sections of $h$ are different, 
since $P^2=-\chi(X)<0$ for any section $P$.

It follows that $\Aut_{B, \QQ}(X)$ preserves each section of $h$. 
By the trivial action on the base, this actually implies that each section is fixed pointwise by $\Aut_{B, \QQ}(X)$. 
Since $\Aut_{B, \QQ}(X)$ is
a nontrivial  group, 
its fixed locus is a proper closed subset of $X$.
It contains all sections which thus form a finite set.
Taking into account the group structure which makes this set into the Mordell--Weil group $\MW(X)$, claim (1) follows.

(2)
If $p_g(X)>0$, then $\Aut_\QQ(X)\subseteq t(\MW(X)_\tor) = t(\MW(X))$ by Proposition \ref{lem: AutQ vs MW} and (1).
But then any torsion section $P$ meets some fibre in a non-identity component; 
this is a consequence of \eqref{eq:pattern}.
It follows that $t(P)$ acts as a non-trivial permutation on the (simple) fibre components,
hence $t(P)\not\in\Aut_\QQ(X)$.
\end{proof}

\section{Bounding the torsion part of the Mordell--Weil groups}
\label{s:MW}

In view of Proposition~\ref{lem: AutQ vs MW} and Lemma \ref{lem: MW finite}, it is crucial to bound $\MW(J(f))_\tor$, if one wants to bound $\Aut_{B, \QQ}(S)$ and $\Aut_{\QQ}(S)$. In this respect, the elliptic modular surfaces, introduced in \cite{Shi72}, play a key role, since $\MW(J(f))_\tor$ is essentially pulled back from the Mordell--Weil group of an elliptic modular surface. 

Note that any abelian group generated by two elements 
can be realized as a subgroup of the Mordell--Weil group of an elliptic modular surface by \cite[Example~5.4]{Shi72}. 
The aim of this subsection is to give a linear bound of $|\MW(X)_\tor|$ for a jacobian elliptic surface $h\colon X\rightarrow B$ in terms of $g(B)$; see Proposition~\ref{prop: bd MW}.

We recall, for each positive integer $N$,  the following congruence modular groups (\cite[Section~4.2]{Miy06}):
\begin{equation}
\begin{split}
&\Gamma(N) = \left\{\begin{pmatrix} a & b \\ c & d  \end{pmatrix} \in \SL_2(\ZZ)\, \Bigg|\, b\equiv c\equiv 0, a\equiv d\equiv 1 \mod N \right\}.\\
&\Gamma_1(N) = \left\{\begin{pmatrix} a & b \\ c & d  \end{pmatrix} \in \SL_2(\ZZ)\, \Bigg|\, c\equiv 0, a\equiv d\equiv 1 \mod N \right\}.
\end{split}
\end{equation}
Let $B(N)$ and $B_1(N)$ be the modular curves associated to $\Gamma(N)$ and $\Gamma_1(N)$ respectively, and let $g(N)$ and $g_1(N)$ be their respective genera. For $N\geq 5$, we have,  by \cite[(5.3)]{Shi72} and \cite[page 314, line 2]{Mar05},
\begin{equation}\label{eq: gN g1N}
\begin{split}
& g(N) =  \frac{1}{24}s(N)N^3 - \frac{1}{4}s(N) N^2 + 1, \\ &g_1(N) = \frac{1}{24} s(N)N^2- \frac{1}{4}u(N) + 1.
\end{split}
\end{equation}
where $s, u\colon \NN\rightarrow \ZZ$ are the multiplicative functions satisfying, for a prime number $p$ and an integer $n\geq 1$,
\begin{equation}\label{eq: s and u}
\begin{split}
& s(p^n) = 1-\frac{1}{p^2},\\
&u(p^n) = p^{n-2} (p-1)((n+1)p-n+1).
\end{split}
\end{equation}
The formula \eqref{eq: gN g1N} gives $g_1(N)=0$ for $5\leq N\leq 10$, and we may continue computing $g_1(N)$ for $11\leq N\leq 25$ as follows:
 %\begin{table}
   %\begin{center}
   \[
     \begin{tabular}{ | l | c | c| c | c | c|c| c| c| c| c| c| c| c| c|c|}
       \hline
$N$ &11& 12 & 13 &14 & 15 & 16 & 17 & 18 & 19 & 20 & 21 & 22 & 23 & 24 & 25\\\hline
$g_1(N)$ &1& 0 & 2 & 1 & 1 & 2 & 5 & 2 & 7 & 3 & 5& 6 & 12 & 5 & 12\\ \hline
     \end{tabular}
     \]
    % \caption{genera of modular curves $B_1(d)$}\label{tab: g1N}
 %  \end{center}
%\end{table}
Based on this, we go on to bound $g_1(N)$ from below in terms of $N$:
\begin{prop}\label{prop: lower bd g1N}
Let $N\geq 5$ be an integer. Then 
%\marginpar{Does \cite[Thm 7 a)]{Mar05} give a better bound? see Cor~\ref{cor: bound g_N}.}
\begin{equation}\label{eq: lower bd g1N}
g_1(N)> \frac{N^2}{12\pi^2} -2.
\end{equation}
\end{prop}
Later on, we will give in Corollary~\ref{cor: bound g_N} a better lower bound when  $N$ is \emph{sufficiently large}.

\begin{proof}
The inequality holds for $N\leq 25$ by looking at the above table. In the following, we may assume that $N\geq 26$. We rewrite \eqref{eq: gN g1N} as
\begin{equation}\label{eq: g1N'}
g_1(N) = \frac{1}{24}s(N)N^2\left( 1 - \frac{6u(N)}{s(N)N^2}\right) +1 
\end{equation}
By the definition of the multiplicative functions $s$ and $u$ in \eqref{eq: s and u}, we have
\[
s(N) = \prod_{p\mid N} \left(1-\frac{1}{p^2} \right),\quad u(N) = \prod_{p^n || N} \left((n+1)p^n - 2np^{n-1} + (n-1)p^{n-2}\right)
\]
where the product ranges over primes $p$ dividing $N$, and $p^n || N$ means that $p^n$ is the largest power of $p$ dividing $N$.

Since the infinite product $s(\infty):=\prod_{p \text{ prime}} (1-\frac{1}{p^2})$ converges to 
\[
\left(\sum_{n\in \NN} \frac{1}{n^2}\right)^{-1} = \frac{6}{\pi^2},
\]
we have
\begin{equation}\label{eq: s lower bd}
s(N)> s(N)\cdot \prod_{p{\nmid} N} \left(1-\frac{1}{p^2}\right) = s(\infty) = \frac{6}{\pi^2}.
\end{equation}
Plugging the lower bound \eqref{eq: s lower bd} of $s(N)$ into the expression \eqref{eq: g1N'}, we obtain
\[
g_1(N) \geq \frac{N^2}{4\pi^2}\left( 1 - \frac{6u(N)}{s(N)N^2}\right) +1.
\]
Thus, to show the inequality \eqref{eq: lower bd g1N} of the proposition, it suffices to show the following claim.
\end{proof}

\begin{claim}\label{claim: bd u/s}
For $N\geq 26$, we have
\begin{equation}\label{eq: bd u/s}
1 - \frac{6u(N)}{s(N)N^2} \geq \frac{1}{3},\quad \text{or equivalently,}\quad t(N):=\frac{u(N)}{s(N)N^2} \leq \frac{1}{9}.
\end{equation}
\end{claim}
\begin{proof}[Proof of the claim.]
Note that the function $t(N)$ is multiplicative. We need to study the values $t(p^n)$ at the prime powers $p^n$. By the definition of the functions $s(N)$ and $u(N)$ in \eqref{eq: s and u}, we have for a prime $p$ and an integer $n\geq 1$:
\begin{equation}\label{eq: t(pn)}
t(p^n) = \frac{p^{n-2}\left((n+1)p^2-2np+n-1\right)}{\left(1-\frac{1}{p^2}\right)p^{2n}} = \frac{(n+1)p-(n-1)}{p^n(p+1)}
\end{equation}
Replacing $p$ and $n$  in \eqref{eq: t(pn)} with two real variables $x, y\in \RR_{\geq 2}$, and analyzing the function $t(x^y)$, one sees that $t(p^n)$ is a decreasing function of $p$ and $n$. For the small $p$ and $n$, we compute directly the value $t(p^n)$ as follows:
 %\begin{table}
  % \begin{center}
  \[
     \begin{tabular}{ | l | c | c| c | c | c|c| c| c| c|}
       \hline
 \diagbox{$n$}{$p$}&$2$& 3 &5 &7 & 11 & 13 & 17 & 19 & 23 \\\hline
$1$ &2/3& $1/2$ & $1/3$ & $1/4$ &$1/6$ & $1/7$ & $1/9$ & $1/10$ &$1/12$ \\ \hline
$2$ &$5/12$& $2/9$ & $7/75$ & $5/98$ & $4/363$ & 19/1183  & \diagbox{}{}  & \diagbox{}{}  & \diagbox{}{} \\ \hline
$3$ &$1/4$& $5/54$ &  \diagbox{}{}  &  \diagbox{}{}  &  \diagbox{}{}  &  \diagbox{}{}  &  \diagbox{}{}  &  \diagbox{}{}  &  \diagbox{}{}  \\ \hline
$4$ &$7/48$& \diagbox{}{} & \diagbox{}{}  &  \diagbox{}{} &  \diagbox{}{}  &  \diagbox{}{}  &  \diagbox{}{}  &  \diagbox{}{}  & \diagbox{}{}  \\ \hline
$5$ &$1/12$&  \diagbox{}{}  &\diagbox{}{}  & \diagbox{}{}  &\diagbox{}{}  & \diagbox{}{}  & \diagbox{}{}  & \diagbox{}{}  &\diagbox{}{}  \\ \hline
     \end{tabular}
     \]
     %\caption{genera of modular curves $B_1(d)$}\label{tab: g1N}
  % \end{center}
%\end{table}
We stop the computation along a column as soon as $t(p^n)\leq \frac{1}{9}$, which is then sufficient for the proof of Claim~\ref{claim: bd u/s}. 

Now suppose that $N\geq 26$, and it has a prime decomposition
\[
N = p_1^{n_1}\cdot\dots\cdot p_r^{n_r}.
\]
We proceed according to the value of $r$. If $r\geq 3$, then, by the mononicity of $t(p^n)$ for $p$ and $n$, we have
\[
t(N) \geq t(2\cdot 3\cdot 5) =t(2)\cdot t(3)\cdot t(5) = \frac{2}{3}\cdot \frac{1}{2}\cdot \frac{1}{3} = \frac{1}{9}.
\]
If $r=2$, then, by the mononicity of $t(p^n)$ for $p$ and $n$, we may assume that $p_i^{n_i}<26$ for both $i=1,2$, and the value of $t(N)$ is at most one of the following:
\begin{align*}
&t(2^4\cdot 3),\, t(2^2\cdot 3^2),\, t(2^3\cdot 5), \, t(2\cdot 5^2),\,  t(2^2\cdot 7),\, t(2^2\cdot 11),\, t(2\cdot 13), \\
&t(3^2\cdot 5),\, t(3\cdot 5^2), \,t(3^2\cdot 7),\, t(3\cdot 11),\, t(5\cdot 7)
\end{align*}
all of which can be directly checked to be less than $\frac{1}{9}$. For example, 
\[
t(2^4\cdot 3) = t(2^4)\cdot t(3) = \frac{7}{48}\cdot \frac{1}{2} = \frac{7}{96}<\frac{1}{9}.
\]
If $r=1$, then $N=p^n$ for some prime $p$, and by the mononicity of $t(p^n)$ for $p$ and $n$ again, we have
\[
t(N) \leq \min\{t(2^5),\, t(3^3), \, t(7^2),\, t(29)\} = \min\left\{\frac{1}{12},\, \frac{5}{54},\, \frac{5}{98}, \, \frac{1}{15}\right\} = \frac{5}{98} <\frac{1}{9}.
\]
This finishes the proof of the claim.
\end{proof}
%The proof of the proposition is now complete.
%\end{proof}

As a side remark, we observe an asymptotic behavior of the ratio $g_1(N)/N^2$, which seems to be known (\cite[Theorem~7]{Mar05}),
 but not in the following form:
\begin{lem}
We have 
\begin{equation}
\varlimsup_{N\to \infty} \frac{g_1(N)}{N^2}=\frac{1}{24},\quad
\varliminf_{N\to \infty} \frac{g_1(N)}{N^2} = \frac{1}{4\pi^2}.
\end{equation}
\end{lem}
\begin{proof}
As is known (see for example \cite[Lemma~30]{Mar05}), we have 
\[
\lim_{
N
\to \infty} \frac{u(N)}{
{N^2}
} =0,
\]
and hence by \eqref{eq: gN g1N},  $\frac{g_1(N)}{N^2}$ and $\frac{s(N)}{24}$ share the same upper and lower limits.

To compute the upper limit, we note that $s(N)<1$ in any case and, taking 
$N$
 to be primes $p$, we have
\[
\lim_{p\to \infty} \frac{s(p)}{24} = \frac{1}{24}\lim_{p\to \infty} \left(1-\frac{1}{p^2}\right) =\frac{1}{24}.
\]

For the lower limit, we note that $s(N)>s(\infty) = \frac{6}{\pi^2}$ as in \eqref{eq: s lower bd}. On the other hand, taking the sequence of prime numbers $\{p_r\}_r$ and setting $N_r:=\prod_{1\leq i\leq r} p_i$, we obtain
\[
\lim_{r\to \infty} \frac{s(N_r)}{24} =\frac{1}{24} \lim_{r\to \infty} \prod_{1\leq i\leq r} \left(1-\frac{1}{p_i^2}\right) = \frac{1}{24}\cdot\frac{6}{\pi^2} = \frac{1}{4\pi^2}.
\]
\end{proof}

\begin{cor}\label{cor: bound g_N}
For any positive real number $\epsilon$, there is a positive integer $N_\epsilon$, depending on $\epsilon$, such that, for any $N\geq N_\epsilon$, we have
\[
g_1(N) \geq \frac{N^2}{4\pi^2+\epsilon}.
\]
\end{cor}

Now we may use Proposition~\ref{prop: lower bd g1N} to bound $|\MW(X)|$ for jacobian elliptic surfaces. 

\begin{prop}\label{prop: bd MW}
Let $h\colon X\rightarrow B$ be a jacobian elliptic surface with $\chi(\sO_X)>0$. Let $g=g(B)$ be the genus of the base curve. Then 
%\marginpar{One could try to improve this by using all $\Gamma_1(d)\cap\Gamma_0(d')$ for $d'\mid d$, but that's probably overly tedious}
\begin{equation}\label{eq: MW bd}
|\MW(X)_\tor|< 12\pi^2(g+2).
\end{equation}
Moreover, if the fibration is isotrivial, then there is the uniform bound $|\MW(X)_\tor|\leq 4$.
\end{prop}

\begin{proof}
First suppose that $h$ is not isotrivial. Let $P\in \MW(X)_\tor$ be a torsion section of maximal order, say $N$. Then $|\MW(X)_\tor|\leq N^2$,
since any finite subgroup of any elliptic curve is 2-generated (as in Theorem \ref{thm} (i)!).
Let $h_N\colon X_1(N)\rightarrow B_1(N)$ be the modular elliptic surface 
parametrizing elliptic curves with an $N$-torsion point $P_N$ for $N\geq 4$
(which thus  features as a section of $X_1(N)$). 
Then there is a surjective morphism $\pi\colon B\rightarrow B_1(N)$ such that $h\colon X\rightarrow B$ and $P$ is obtained by pulling back $h_N$ and $P_N$ via the base change $\pi$. Since $N^2\leq 12\pi^2(g_1(N)+2)$ by Proposition~\ref{prop: lower bd g1N} and $g=g(B)\geq g(B_1(N))=g_1(N)$, we obtain the claimed inequality \eqref{eq: MW bd}. 

If $h$ is isotrivial, then there are no $\I_n$ fibres with $n\geq 1$, and the order of $\MW(X)_\tor$ is at most $4$
by \eqref{eq:FF} (cf.~\cite[5.6.1]{SS19}), which is of course smaller than $12\pi^2(g+2)$.
\end{proof}

\begin{remark}
The bound $|\MW_\tor|\leq N^2$ in the proof of Proposition \ref{prop: bd MW} is not optimal;
for instance, if $|\MW_\tor|= N^2$, then the genus $g$ of the base curve $B$ should rather be bounded from below by $g(N)$.
We omit the  details for brevity.

A computation of the genera of the modular curves for all intermediate groups between $\ZZ/N\ZZ$ and 
$(\ZZ/N\ZZ)^2$
such that the obvious condition $g_1(N)\leq 4$ holds,
e.g.\ using Magma \cite{MAGMA}, yields the following sharp upper bounds for $|\MW_\tor|$ in terms of the genus $g$:
$$
\begin{array}{c||c|c|c|c|c}
g & 0 & 1 & 2 & 3 & 4\\
\hline
|\MW_\tor|\leq & 25 & 36 & 36 & 49 & 50
\end{array}
$$
\end{remark}

\section{General upper bounds for  $|\Aut_\QQ(S)|$ and $|\Aut_\ZZ(S)|$}
\label{s:upper_bounds}

%We are going to bound the cardinality $|\MW(X)|$ for a non-isotrivial elliptic surface $$h\colon X\rightarrow B $$ with a section, in terms of the genus of $B$. The key case is when $X=X_1(d)$ is an elliptic modular surface associated to the congruence modular group  \[\Gamma_1(d) = \left\{\begin{pmatrix} a_{11} & a_{12} \\ a_{21} & a_{22}  \end{pmatrix} \in \SL_2(\ZZ)\, \Bigg|\, a_{21}\equiv 0,\, a_{11}\equiv a_{22}\equiv 1 \mod d \right\}, \]so $B$ is the corresponding modular curve $B_1(d)$. 

%The idea of going to the Jacobian $J(f)\colon J(S)\rightarrow B$ is a natural one. The map $\varphi$ in the following lemma appears also in  \cite[(3.3)]{DM22}.

Let $f\colon S\rightarrow B$ be a minimal properly elliptic surface with $\chi(S)>0$. 

This inequality can occur in two ways:

either $q = p_g=0$ which means that   the relative Jacobian $J(S)$ is  rational,
or $p_g >0$.

Our aim in this section is to  give an
upper  bound for $|\Aut_\QQ(S)|$  in terms of the irregularity $q(S)$.  If the relative Jacobian $J(S)$ is not rational,  see Corollary~\ref{cor: AutB bd}. 

In the first case where  $J(S)$ is rational, the situation is quite different, and we treat this  case first. 

\subsection{The case where $J(S)$ is rational}

In treating the case of a rational Jacobian,
there will be 4 surfaces requiring special attention.
They are singled out by the properties that they are isotrivial and have only  two singular fibres,
so that the Mordell--Weil group is finite: they are the so called   extremal rational elliptic surfaces;
it was also shown in \cite[Table 3]{DM22} that these are the only jacobian rational elliptic surfaces $X$ over $\CC$ 
with infinite $\Aut_\QQ(X)\cong\CC^*$ (here, obviously, $\Aut_\QQ(X)= \Aut_\ZZ(X)$).
The following table lists the surfaces in the notation of  \cite{MP86} 
together with a polynomial $P \in\CC[t]$
such that a Weierstrass form is given by $y^2= x^3 + P$, and with further essential information.

\begin{table}[ht!]
$$
\begin{array}{c||c|c|c|c}
X & X_{22} & X_{33} & X_{44} & X_{11}(\lambda)  \;\; (\lambda\in\CC^\times)\\
\hline
P & t & tx & t^2 & tx^2 +\lambda t^2x \\
\text{singular fibres} & \II+\II^* & \III + \III^* & \IV + \IV^* & 2\times \I_0^*\\
\MW(X) & \{0\} & \ZZ/2\ZZ & \ZZ/3\ZZ & (\ZZ/2\ZZ)^2\\
\Aut_{B,\ZZ}(X) & \mu_6 & \mu_4 & \mu_3 & \mu_2\\
r & 6 & 4 & 3 &2
\end{array}
$$
\caption{Special isotrivial rational elliptic surfaces}
\label{table:special}
\end{table}

Here $r$ equals the size of $\Aut_{B,\ZZ}(X)\cong \mu_r$ and will be used in the sequel as a reference:
its geometric meaning is that an isotrivial elliptic surface with such a Jacobian is the quotient $(C \times E)/G$,
where $ G < T \rtimes \mu_r$, $T$ being a subgroup of $E$ (see Section 8.1).

\begin{prop}\label{prop: pg=0 bd J(S)}
Let $f\colon S\rightarrow B$ be a minimal properly elliptic surface such that $J(f)\colon J(S)\rightarrow B$ is a rational elliptic surface. Then the following holds.
\begin{enumerate}
\item 
 If $f$ does not admit any singular fibre of additive type, or, equivalently, if $J(f)$ is semistable, then
we have
\[
 |\Aut_\QQ(S)| = |\Aut_{B,\QQ}(S)| \leq 9,
\]
  $\Aut_\QQ(S)$ is abelian, 2-generated.  Moreover, the upper bound can be attained 
and, in the equality case,  $ \Aut_\QQ(S) \cong (\ZZ/3\ZZ)^2$ 
 and $J(S)$ must be the extremal rational elliptic surface $X_{3333}$ in the list of \cite{MP86}, see also  \cite[VIII.1.4]{Mir89}.

\item

If $f$ admits a singular fibre of additive type, then we have
\[
|\Aut_{B,\QQ}(S)|\leq 2, \quad |\Aut_{\QQ}(S)|_B|\leq 3,\quad |\Aut_{\QQ}(S)|\leq 4.
\]
unless $J(S)$ is one of the surfaces $X_{22}$, $X_{33}$, $X_{44}$, $X_{11}(\lambda)$
 in \cite{MP86},  \cite[VIII.1.4]{Mir89}.

\item 
If $J(S)$ is one of the isotrivial elliptic surfaces $X_{22}, X_{33}$, $X_{44}$, $X_{11}(\lambda)$, then 
 the following holds.
\begin{enumerate}
\item $\Aut_\QQ(S)|_B$ is finite cyclic, and its order divides $P_2(S)+1$, where   $P_2(S):= h^0(2 K_S)$ is  
the second plurigenus  of $S$;
\item 
$\Aut_{B, \QQ}(S)$ is a subgroup of $\Aut_{B,\QQ}(J(S))=\mu_r$ 
with $r=6,4,3,2$  from Table \ref{table:special} in the given order of the isotrivial surfaces.
Moreover, it is a proper subgroup if $r=6$.

\item
Conversely,
 for some of the surfaces $X=
 X_{22},$ 
$X_{33}$,
given $s \in \NN$, divisible by  $3$ or odd,  there is a minimal properly elliptic surface $S$ such that $J(S)=X$
and $|\Aut_\QQ(S)|_B| = P_2(S)+1 = s$, whereas for $s$ even we can achieve  $|\Aut_\QQ(S)|_B| = (P_2(S)+1)/2 = s/2$.

\end{enumerate}
%
%
%
%and hence $\Aut_\QQ(S)$ can have arbitrarily large order. More precisely, in these cases, $\Aut_\QQ(S)|_B$ i and can become arbitrarily large,  where $P_2(S):=h^0(S, 2K_S)$ is the $2$-genus of $S$.
%\marginpar{That the order actually becomes that large is only shown much later, and everything else also holds for $X_{22}$.}
\end{enumerate}
\end{prop}

%
%Observe that part (3) gives  a different (not completely self-contained) proof of  some assertions in Theorem \ref{rs}.
%Explicit equations of $X_{33}, X_{44}$, and $X_{11}(\lambda)$ can be found in Example \ref{ex}.

\begin{proof}[Proof of Proposition~\ref{prop: pg=0 bd J(S)}]
Consider the homomorphism 
$$\Phi\colon \Aut_f(S)\rightarrow \Aut_{J(f)}(J(S))
$$ 
of Lemma~\ref{lem: Jacobian}. Since $S$  has Kodaira dimension $1$, and $f$ is the pluricanonical fibration,
while, for  $J(S)$,  $J(f)$ is the plurianticanonical fibration in view of the canonical bundle formula \eqref{eq: canonical bundle formula},
and of the fact that  there are no multiple fibres, we have
\[
 \Aut_f(S) = \Aut(S),\quad  \Aut_{J(f)}(J(S)) = \Aut(J(S))
\]
  hence $\Phi$ is indeed a homomorphism between the full automorphism groups. Restricting $\Phi$ to $\Aut_\QQ(S)$, we obtain, by Lemma~\ref{lem: Jacobian}, the exact sequences
\begin{equation}\label{eq: AutQ(S) rational}
\begin{split}
& 1\rightarrow t(\MW(J(S)))\cap \Aut_\QQ(S) \rightarrow \Aut_\QQ(S) \rightarrow \Aut_\QQ(J(S)),\\
& 1\rightarrow t(\MW(J(S)))\cap \Aut_{B,\QQ}(S) \rightarrow \Aut_{B, \QQ}(S) \rightarrow \Aut_{B, \QQ}(J(S)), \\
&1\rightarrow \Aut_\QQ(S)|_B\rightarrow \Aut_\QQ(J(S))|_B.
\end{split}
\end{equation}

\medskip

(1) We have $|\MW(J(S))_\tor|\leq 9$ in any case (\cite[Corollary~8.21]{SS19}), 
and  equality holds exactly if $J(S)=X_{3333}$ with $\MW(J(S)) \cong (\ZZ/3)^2$.  If $J(f)$ is semi-stable, then it is not in \cite[Table 3]{DM22} and hence $\Aut_\QQ(J(f))$ is trivial. By the first exact sequence of \eqref{eq: AutQ(S) rational}, we obtain
\[
 \Aut_\QQ(S) \subset t(\MW(J(S))),
\]
and (1) follows  since the finiteness of $\Aut_\QQ(S)$ implies that $\Aut_\QQ(S)\subseteq 
t(\MW(J(S))_\tor)$.

This also proves that $\Aut_\QQ(S)$ is abelian and 2-generated.

(2) If $f$ admits an additive fibre, then by Lemma~\ref{lem: add fib}, the homomorphism $\Aut_{\QQ}(S)\rightarrow \Aut_{\QQ}(J(S))$ of Lemma~\ref{lem: Jacobian} is injective, 
and the same is true for the restrictions to $\Aut_{B,\QQ}$ and $\Aut_\QQ|_B$ by \eqref{eq: AutQ(S) rational}.

If $J(S)$ is not one of the surfaces $X_{22}$, $X_{33}$, $X_{44}$, $X_{11}(\lambda)$, then by \cite[Theorem~6.4]{DM22}, we have
\[
|\Aut_{B, \QQ}(J(S))|\leq 2, \quad |\Aut_\QQ(J(S))|_B|\leq 3,\quad |\Aut_\QQ(J(S))|\leq 4.
\]

Now (2) follows from these bounds since, by injectivity,  $|\Aut_{\QQ}(S)|\leq |\Aut_{\QQ}(J(S))|$ etc.

\medskip

(3) If $J(S)$ is one of the surfaces $X_{22}, X_{33}$, $X_{44}$, $X_{11}(\lambda)$, there is a pair of fibres 
  with singular support, i.e. with $F_{red}$ singular. 
  We may assume them to be $f^*0$ and $f^*\infty$; obviously they are preserved by $\Aut_\QQ(S)$. 
  It follows that $\Aut_\QQ(S)|_B$ fixes the two points $0$ and $\infty$ of $B=\PP^1$, 
  and hence is finite cyclic; in particular, all other orbits are free. 
Since $S$ is properly elliptic, the canonical bundle formula \eqref{eq: canonical bundle formula} implies
that there are $s\geq 2$ multiple fibres, say $f^*b_i = m_i F'_i$ with $1\leq i\leq s$. 
Then $\{b_i\mid 1\leq i\leq s\}$ consists of free orbits under $\Aut_\QQ(S)|_B$, 
and it follows that the order of $\Aut_\QQ(S)$ divides $s$.

On the other hand, spelling out the canonical bundle formula \eqref{eq: canonical bundle formula}, we have
\begin{eqnarray}
\label{eq:cbf}
K_S = f^*(K_B+L) + \sum_{1\leq i\leq s} (m_i-1) F'_i
\end{eqnarray}
where $L$ is a divisor on $B$ with $\deg L=\chi(S)=1$. 
Therefore, 
\begin{eqnarray*}
P_2(S) & = &
h^0\left(S, f^*(2K_B+2L) + \sum_{1\leq i\leq s} \underbrace{(2m_i-2) F'_i}_{=f^*b_i+(m_i-2)F'_i}\right)\\
&
= &
h^0\left(S, f^*\left(2K_B+2L+ \sum_{1\leq i\leq s} b_i\right)\right)
\\
& = & h^0\left(B, 2(K_B+L)+\sum_{i=1}^s b_i\right) = s-1.
\end{eqnarray*}
It follows that the order of $\Aut_\QQ(S)$ divides $s=P_2(S)+1$ 
 and (3a) is proved.

(3b) Since $\Aut_{B,\QQ}(S)\cap t(\MW(J(S)))$ is trivial by Lemma~\ref{lem: add fib}, the second exact sequence of \eqref{eq: AutQ(S) rational} shows that $\Aut_{B,\QQ}(S)\rightarrow \Aut_{B, \QQ}(J(S))$  is an injection.
But then, by  \cite{DM22} (or by direct computation), $ \Aut_{B, \QQ}(J(S)) =\mu_r$
 where $r=6,4,3,2$ depends on $J(S)=X_{22}, X_{33}, X_{44}, X_{11}(\lambda)$ respectively.
 
 The last assertion of Proposition \ref{prop: pg=0 bd J(S)} (3) (b) will be proven  in Lemma \ref{lem:r=6}.

 This shows (3b).

 For (3c), concerning 
the existence of surfaces with large  $\Aut_\QQ(S)|_B$,
see Proposition \ref{order3} for $s$ divisible by $3$,  
see  Lemma \ref{lem:2s} for $s$ odd, and finally, for $s$ even, see Theorem \ref{prop:X_33} or, again, Lemma \ref{lem:2s}.
\end{proof}

\subsection{Cohomologically trivial automorphisms}

Next we study the cohomologically trivial automorphism group $\Aut_\ZZ(S)$, mainly in the case where $J(S)$ is rational. 
Then the bounds from Proposition \ref{prop: pg=0 bd J(S)} (1) carry over directly,
so we will be mainly concerned with the isotrivial cases 
from Proposition \ref{prop: pg=0 bd J(S)} (2) and (3),
but we keep the arguments as general as possible.
As a quick reference, we record the following special case of Corollary \ref{base-action}.

\begin{lem}\label{lem: AutZ g=0}
Let $f\colon S\rightarrow \PP^1$ be a relatively minimal  elliptic surface with $\chi(S)>0$. Then  
\begin{enumerate}
\item $\Aut_\ZZ(S)$ preserves each multiple fibre of $f$.
\item $\Aut_\ZZ(S)|_{\PP^1}=\{\id_{\PP^1}\}$ unless $f$ has at most 2 multiple or reducible fibres  in total.
\end{enumerate}
\end{lem}

%\begin{proof}
%(1) this was already observed in section 2.3, and follows from 
%\cite[Lemma 2.1]{CFGLS24}.
%
%
%
%\medskip
%
%(2) Let $\sigma\in \Aut_\ZZ(S)$ and $\sigma|_B\in \Aut_\ZZ(S)|_B$ its induced automorphism of $B$, that is, $f\circ \sigma = \sigma|_B\circ f$. Let $f^*{b_i}$, $1\leq i\leq s$ be the collection of all the multiple fibres together with the reducible fibres. Since $\sigma$ preserves each $f^*{b_i}$, the automorphism $\sigma|_B$ fixes each $b_i$ for $1\leq i\leq s$. If $s\geq 3$, then necessarily $\sigma|_B=\id_B$
%since $B\cong\PP^1$.
%\end{proof}
%

\begin{corollary}
\label{cor: rational J(S) AutSB}
If $f\colon S\rightarrow B$ is a minimal properly elliptic surface 
 with $\chi(S)>0$ such that $J(S)$ is rational, 
 then $\Aut_\ZZ(S)|_B$ is trivial.
\end{corollary}

\begin{proof}

Suppose on the contrary that $\Aut_\ZZ(S)|_B$ is nontrivial. Then so is $\Aut_\ZZ(J(S))|_B$, and $J(S)$ lands in \cite[Table~3]{DM22}. Since $S$ has Kodaira dimension 1 while $J(S)$ is rational, there exist at least two multiple fibres of $f$ by the canonical bundle formula \eqref{eq: canonical bundle formula}. Then, by inspecting \cite[Table~3]{DM22}, the total number of reducible fibres and multiple fibres is at least 3.  
It follows from Lemma~\ref{lem: AutZ g=0} (2) that $\Aut_\ZZ(S)|_B$ is trivial. 
\end{proof}

\begin{rmk}
If $J(S)$ is not rational, then $\Aut_\QQ(S)|_B$ and hence also $\Aut_\ZZ(S)|_B$ will be shown to be trivial in Proposition~\ref{prop: trivial Aut_Q on B}.
\end{rmk}

\begin{corollary}
\label{cor: pg=0 Aut_Z}
Let $f\colon S\rightarrow B$ be a minimal properly elliptic surface such that $J(S)$ is a rational elliptic surface. Then %the following holds.
%\begin{enumerate}
%\item   
$|\Aut_\ZZ(S)|=|\Aut_{B,\ZZ}(S)| \leq  9$, and equality may hold only if $J(S)=X_{3333}$.
% \item   If $J(f)$ is not semistable (in particular, if $f$ is isotrivial), then $|\Aut_\ZZ(S)|\leq 4$, and equality may hold only if $J(S)=X_{33}$.
%\end{enumerate}
\end{corollary}

\begin{proof}
 %(1) 
 The first equality follows from  Corollary~\ref{cor: rational J(S) AutSB} 
combined with \eqref{eq:es} restricted to $\Aut_\ZZ(S)$, 
and, since $\Aut_{B,\ZZ}(S)\subseteq\Aut_{B,\QQ}(S)$,  the inequality 
follows from Proposition~\ref{prop: pg=0 bd J(S)}, and 
 the characterization of the  case
of potential equality   follows from Proposition~\ref{prop: pg=0 bd J(S)}(1).
%
%(2) If $J(f)$ is not semistable, then
%\[
%|\Aut_{B,\ZZ}(S)|\leq |\Aut_{B, \QQ}(S)|\leq 4
%\]
%where the second inequality and the characterization of the  case
% of potential equality is again by Proposition~\ref{prop: pg=0 bd J(S)} (2) and (3b). 
\end{proof}

\begin{remark}
We don't know whether the  bounds in Corollary \ref{cor: pg=0 Aut_Z} are sharp.

Sections \ref{s:2} and \ref{s:3} exhibit non-isotrivial properly elliptic surfaces
with $|\Aut_\ZZ(S)|=2$ and $3$,
 and same for isotrivial fibrations in Section \ref{s:iso2}.

\end{remark}

\subsection{The case where $J(S)$ is not rational}
\begin{prop}\label{prop: trivial Aut_Q on B}
Let $f\colon S\rightarrow B$ be a minimal properly elliptic surface with $\chi(S)>0$ such that $J(S)$ is not rational. Then the group $\Aut_\QQ(S)|_B$ is trivial.
\end{prop}

\begin{proof}
We can divide the discussion into the following three cases:
\begin{enumerate}
\item $g(B)\geq 1$;
\item $\chi(S)\geq 3$;
\item $g(B)=0$ and $\chi(S)=2$.
\end{enumerate} 
Indeed, $J(S)$ is rational if and only if $\chi(S)=1$ and $g(B)=0$.
Obviously, this is the only case with $\chi(S)>0$ left by the above three cases.

Take $\sigma\in\Aut_\QQ(S)$, and let $\sigma|_B$ denote its image in $\Aut_\QQ(S)|_B$. We need to show that $\sigma|_B=\id_B$ if one of the above conditions holds.

\medskip 

\noindent (1) First suppose that $g(B)\geq 1$. We have $f^*\colon H^1(B, \QQ)\hookrightarrow H^1(S, \QQ)$, and the inclusion is compatible with the actions of $\Aut_\QQ(S)$ and $\Aut_\QQ(S)|_B$ on the source and the target respectively. Since $\sigma$ induces the trivial action on $H^1(S, \QQ)$, $\sigma|_B$ acts trivially on $H^1(B, \QQ)$.  If $g(B)\geq 2$, it follows from Lefschetz' fixed point theorem 
%an instance of the Torelli theorem 
that the only numerically trivial automorphism is $\id_B$. 
If $g(B)=1$ then $B$ is an elliptic curve and $\sigma|_B$ acts as a translation on $B$. Since $e(S) =12\chi(S)>0$, by the topological Lefschetz fixed point theorem, $S^\sigma\neq \emptyset$. It follows that 
\[
\emptyset\neq f(S^\sigma)\subset B^{\sigma|_B}
\]
and $\sigma|_B$ must be the identity.

\medskip 

\noindent {(2)} Now assume that $\chi(S)\geq 3$. By the canonical bundle formula \eqref{eq: canonical bundle formula}, we have
\[
|K_S| = f^*|K_B+L| + \sum_{1\leq i\leq r} (m_i-1) F_i
\]
where $L$ is a divisor on $B$ with $\deg L=\chi(S)$.
Thus the canonical map $\varphi_{S}$ of $S$ factors through $B$ as follows:
\begin{equation*}%\label{eq: factor canonical map}
\varphi_{S}\colon S\xrightarrow{f} B\xrightarrow{\varphi_{B}} \PP^{p_g-1}
\end{equation*}
where $p_g = h^0(S, K_S)$ is the geometric genus of $S$ and $\varphi_B$ is the morphism associated to the linear system $|K_B+L|$. Since $\deg L = \chi(S)\geq 3$, the morphism $\varphi_{K_B+L}$ is an embedding. 

Then we have a commutative diagram
\begin{equation*}%\label{eq: factor canonical map}
\begin{tikzcd}
S\arrow[r, "f"] \arrow[d, "\sigma"]&  B\arrow[r, "\varphi_{B}"]  \arrow[d, "\sigma|_B"]& \PP^{p_g-1} \arrow[d, equal]\\
S\arrow[r, "f"] &  B\arrow[r, "\varphi_{B}"] & \PP^{p_g-1} 
\end{tikzcd}
\end{equation*}
where the last vertical equality is because $\sigma$ acts trivially on $H^0(S, K_S)$, which is a direct summand of $H^2(S, \CC)$ in the Hodge decomposition. It follows that $\sigma|_B = \id_B$, and  $\Aut_\QQ(S)|_B = \{\id_B\}$.

\medskip 

\noindent {(3)} Finally, assume that  $g(B)=0$ and $\chi(S)=2$. 
Then the relative Jacobian $J(S)$ over $B$ is a K3 surface. 
By the  proof of the Torelli theorem for K3 surfaces (\cite{PS72}, \cite{br})  one knows that $\Aut_\QQ(J(S))$ is trivial. 
Via the homomorphism \eqref{eq: Jacobian}, one sees that
\[
\Aut_\QQ(S)|_B = \Aut_\QQ(J(S))|_B = \{\id_B\}.
\]
This completes the proof of Proposition \ref{prop: trivial Aut_Q on B}.
\end{proof}

\begin{remark}
Note that $\Aut_\QQ(S)$ may be non-trivial,
even if $J(S)$ is a K3 surface.
This will follow from Construction \ref{const: log transform}
applied to suitable elliptic K3 surfaces
admitting a torsion section (as in the proof of Theorem \ref{thm:p_g>0} (ii)
at the end of Subsection \ref{s:pf1-ii}).
\end{remark}

\begin{cor}\label{cor: pg>0 add fib 2}
Let $f\colon S\rightarrow B$ be a minimal properly elliptic surface with $\chi(S)>0$. If $p_g(S)>0$ and if $f$ admits a fibre of additive type  or a section, then $\Aut_\QQ(S)$ is trivial.
\end{cor}
\begin{proof}
We have
\[
\Aut_{\QQ}(S) = \Aut_{B,\QQ}(S) = \{\id_S\}
\]
where the first equality is by Proposition~\ref{prop: trivial Aut_Q on B} 
and the second is by Corollary~\ref{cor: pg>0 add fib 1}  resp.\ Lemma \ref{lem: MW finite} (2).
\end{proof}

\begin{cor}\label{cor: p_g>0 isotrivial}
Let $f\colon S\rightarrow B$ be a minimal properly elliptic surface with $\chi(S)>0$. If $p_g(S)>0$ and if $f$ is isotrivial, then $\Aut_\QQ(S)$ is trivial.
\end{cor}
\begin{proof}
Since $f$ is isotrivial and $\chi(S)>0$, there exists a fibre of additive type of $f$. Now apply Corollary~\ref{cor: pg>0 add fib 2}.
\end{proof}

Now we turn to the kernel $\Aut_{B, \QQ}(S)$ of the natural homomorphism $\Aut_\QQ(S)\rightarrow\Aut(B)$. 
As the examples in Section~\ref{s:pf2} illustrate, there is no general bound for $\Aut_{B, \QQ}(S)$. 
Nevertheless, if $S$ is not isotrivial, we can bound $\Aut_{B, \QQ}(S)$ in terms of the genus $g(B)$ of the base curve $B$
(or of other invariants of $S$); 
in contrast, the isotrivial case showcases a uniform bound independent of $g(B)$,
 see %Theorem \ref{rs}, 
 Proposition~\ref{prop: pg=0 bd J(S)} and Corollary~\ref{cor: p_g>0 isotrivial}.

\begin{prop}\label{prop: bd AutB(S)}
Let $f\colon S\rightarrow B$ be a minimal properly elliptic surface with $\chi(S)>0$. 
 If $f$ is not isotrivial,
then 
\begin{equation}\label{eq: Aut(BQ) bd}
|\Aut_{B, \QQ}(S)|< 12\pi^2(q(S)+2).
\end{equation}
\end{prop}

\begin{proof}

\
If $J(S)$ is rational, then the order of $\Aut_{B,\QQ}(S)$ is at most $9$ by Proposition~\ref{prop: pg=0 bd J(S)}, 
and hence \eqref{eq: Aut(BQ) bd} holds in this case. 
If $J(S)$ is not rational, then $p_g(S)>0$, and we have $\Aut_{B,\QQ}(S)\subset t(\MW(J(S))_\tor)$ by Proposition~\ref{lem: AutQ vs MW}. 
It follows by Proposition~\ref{prop: bd MW} that
\[
|\Aut_{B, \QQ}(S)|\leq |\MW(J(S))_\tor| \leq 12\pi^2(g(B)+2)= 12\pi^2(q(S)+2)
\]
where, for the last equality, we used the fact that $g(B)=q(S)$, due to the assumption that $\chi(S)>0$.
\end{proof}

\begin{cor}\label{cor: AutB bd}
Let $f\colon S\rightarrow B$ be a minimal properly elliptic surface with $\chi(S)>0$ such that $J(S)$ is not rational. Then 
\begin{equation}\label{eq: AutB bd}
|\Aut_{\QQ}(S)|< 12\pi^2(q(S)+2).
\end{equation}
\end{cor}

\begin{proof}

If $f$ is isotrivial, then $\Aut_\QQ(S)$ is trivial by Corollary \ref{cor: p_g>0 isotrivial}.
Hence we may assume that $f$ is not isotrivial.
Then Propositions~\ref{prop: trivial Aut_Q on B} and \ref{prop: bd AutB(S)} show that
\[
|\Aut_{\QQ}(S)|= |\Aut_{B, \QQ}(S)|< 12\pi^2(q(S)+2),
\]
using the  exact sequence \eqref{eq:es} restricted to $\Aut_\QQ(S)$.

\end{proof}
%Then $\Aut_{B,\QQ}(S)$ admits a subgroup of index at most two
%(and actually of index one if $p_g>0$),
%acting by  translations on a general fibre.
%Let $P\in \MW(X)_\tor$ a torsion section of maximal order $N$. Then 
%\[
%|\MW(X)_\tor|\leq N^2.
%\]
%Thus it suffices to bound $N$ in terms of $g(B)$. If $X$ is a modular elliptic surface over $B$, then it is known that $N$ is uniformly bounded by a number depending only on $g(B)$ (REFERENCE?). In the general case, there is modular elliptic surface $h'\colon X'\rightarrow R_S$ with torsion section $P'$ of order $N$ and a surjective morphism $\pi\colon B\rightarrow R_S$ such that $h\colon X\rightarrow B$ and $P$ is obtained by pulling back $h'$ and $P'$ via the base change $\pi$. Since $N$ is uniformly bounded in terms of $g(R_S)$, so it is in terms of $g(B)\geq g(R_S)$.

\begin{remark}
\label{rem:bounds}
The bound \eqref{eq: AutB bd}
from Corollary \ref{cor: AutB bd} also holds if $J(S)$ is rational, but 
different from the four exceptional surfaces from Proposition \ref{prop: pg=0 bd J(S)} {\it (3)}.
\end{remark}

\section{The main construction}
\label{s:constr}

In this section, we explain a method to construct elliptic surfaces 
$$f\colon S\rightarrow B$$
with  nontrivial $\Aut_\QQ(S)$ by applying a suitable base change construction
to a jacobian elliptic surface 
$$h\colon X\rightarrow B$$ 
endowed with a (finite) subgroup $G\subset \MW(X)_\tor$. 
As a preparation, we recall how the group structure of the fibration $h$
 extends over the smooth locus $F^\#$ of any the singular fibre $F$.
In detail, if $F$ is multiplicative (or semi-stable), i.e. of Kodaira type $\I_n \; (n>0)$,
then there is a short exact sequence
\begin{equation}\label{eq: ses F}
1\rightarrow \CC^* \rightarrow F^\# \rightarrow  \ZZ/n\ZZ \rightarrow 1
\end{equation}
which expresses the non-canonical isomorphism \eqref{eq:FFF} in a canonical way.
Here $\CC^*\cong\mathbb G_m$ is identified with the identity component of $F^\#$ 
(intersecting the zero section $O$), 
and $1\in \CC^*$  is identified with the intersection point $F^\#\cap O$, 
while $\ZZ/n\ZZ\cong A_F$ denotes the cyclic group of components of $F^\#$, 
with $0$ corresponding to the identity component $\Theta_0$.
% meeting $O$ (often denoted by $\Theta_0$ and called identity component).

\begin{defn}\label{defn: split}
Let $h\colon X\rightarrow B$ be a jacobian elliptic fibration,
%with zero section $O$, 
and $G\subset \MW(X)_\tor$ a subgroup. 
Suppose $h^*b$ is a fibre of type $\I_n$ for some $n>0$. 
We say $G$ {\bf splits at} $f\colon=h^*b$ if, under the injection $\iota\colon G\rightarrow F^\#$ from \eqref{eq:MW->F},
we have
\[
\iota(G) \;
\cong
\; G_1 \times G_2,
\]
where $G_1 = \CC^*\cap \iota(G)$, and $G_2$ is the image of $\iota(G)$ in $\ZZ/n\ZZ$ under \eqref{eq: ses F}. 
\end{defn}

Note that $G$ splits at  $h^*b$ if and only if there are $P,\, P'\in G$ satisfying the following conditions:
\begin{enumerate} 
\item 
\label{item1}
$G=\langle P, \, P'\rangle$;
\item
\label{item2}
 $t_X(P)$ acts freely on  the set of components of the divisor $h^*b$
  \item
\label{item3}
 $P'$ and $O$ intersect $h^*b$ at the same component.
\end{enumerate}

For immediate use, 
we give sufficient conditions that a subgroup $G\subset \MW(X)$ splits at every multiplicative fibre.
Note that these are independent of the precise 
global structure
of $X$.

\begin{lem}
\label{lem:splitting}

Consider a finite group $G$ written   in the Frobenius-elementary divisors normal form,   
 $G\cong \ZZ/d\ZZ \times \ZZ/da\ZZ$.
 
 Then every injective homomorphism of $G$ to a group  as in \eqref{eq: ses F} splits if and only if 
 $a$ is square free.

In particular, let $h\colon X\rightarrow B$ be a jacobian elliptic fibration with zero section $O$, and $G\subset \MW(X)_\tor$ a subgroup:  if  $a$ is square free, then  $G$ splits at every multiplicative fibre.

A very particular case is the one with $G\cong(\ZZ/d\ZZ)^2$.
\end{lem}

\begin{proof}
The image of $G$ inside $\ZZ/n$ must be  isomorphic to $\ZZ/ da'$, where $ a = a' a''$.

Then the Kernel, $\subset \CC^*$, is cyclic, hence isomorphic to $\ZZ/ da''$.

By the  normal form theorem, $ G \cong \ZZ/ da' \oplus \ZZ/ da''$ if and only if $d = GCD (da' , da'')$,
this means that $a', a''$ are relatively prime. Since this must hold for each 
product decomposition $ a = a' a''$, this is equivalent to $a$  being square free.

\end{proof}

\begin{remark}
\label{rem:8211}
To illustrate that Lemma \ref{lem:splitting} is sharp,
we highlight the extremal rational elliptic surface $X_{8211}$ from \cite{MP86} (also described in \cite{Bea82}).
This has $\MW\cong\ZZ/4\ZZ$ not splitting at the fibre of type $\I_2$
as a generator $P$ meets $\I_2$ at the non-identity component.

Note, however, that the subgroup $\ZZ/2\ZZ\cong \langle 2P\rangle\subset\MW$ is splitting;
applied to a deformation of $X_{8211}$ preserving the 2-torsion section
while replacing the $\I_2$ fibre by two fibres of type $\I_1$, 
this will be used in Section \ref{s:2}.
\end{remark}

We can now introduce our key construction.
In practice, this results in an (algebraic!) logarithmic transformation which 
can be viewed as a generalization of Kond\=o's construction of Enriques surfaces 
starting from a rational elliptic surface with a $2$-torsion section (\cite{Kon86}, later generalized in \cite{HS11}).

Put differently, the construction realizes a homogenous space above $X$ explicitly
while retaining control about the action of $\Aut(X)$.

\begin{construction}\label{const: log transform}
\label{constr}
\begin{enumerate}[leftmargin=*]
\item Let $h\colon X\rightarrow B$ be a relatively minimal elliptic surface with 
$\chi(\sO_X)>0$, a  zero section $O$, and a fixed subgroup $G\subset \MW(X)_\tor$, 
 splitting at all semi-stable singular fibres  $h^*b_i,\,1\leq i\leq s_0$. 
Explicitly, at each singular fibre $h^*b$ of type $\I_{n}$ with $n\geq 1$, 
we have a splitting $G=\langle P_{b}\rangle \times \langle P_{b}'\rangle $, where $t_X(P_{b})$ induces a permutation of order $|\langle P_b\rangle|$ on  the components of $h^*b$, while $t_X(P_{b}')$ preserves each component of $h^*b$.

\item
Renumber the $b_i$ such that there is $s_1\leq s_0$ with $P_{b_i}\neq O$ if and only if $i\leq s_1$,
let $s\geq s_1$ and choose distinct points $b'_{s_1+1}, \hdots, b'_s\in B$  with smooth fibres $h^*b'_j$.
Set $\sB = \{b_1,\hdots,b_{s_1}, b'_{s_1+1},\hdots, b'_s\}$.
By the Riemann Existence Theorem, we can define a $G$-cover 
\[
\tB \stackrel{\bar\pi}\longrightarrow B
\]
with branch locus $\sB$ 
%containing all those points  $b_1,\hdots,b_{s_0}\in B$ underneath semi-stable fibres
%such that $P_{b_i}\neq O$, 
by choosing $P_{b_i}$
 as local monodromy around $b_i\in\sB_0 = \{b_1,\hdots,b_{s_1}\}$, 
complemented by suitable monodromies at 
%the smooth fibres over 
$b'_{s_1+1},\hdots,b'_s$.

\item
Consider the corresponding base change
\begin{equation}
\begin{tikzcd}
\tX \arrow[r]\arrow[rd, "\tilde h"'] & X\times_B \arrow[r]\arrow[d]\tB  & X\arrow[d, "h"] \\
& \tB\arrow[r, "\bar\pi"] & B 
\end{tikzcd}
\end{equation}
where $\tX\rightarrow X\times_B \tB$ is the minimal resolution and  $\tilde h\colon\tX\rightarrow\tB$ is the induced elliptic fibration. 
Note that  $X\times_B\tB$ is normal, with singularities in the fibres above the $b_i \, (i=1,\hdots s_1)$
depending on the order $d_i$ of $P_{b_i}$.
In effect, 
over each node of a fibre of type $\I_{n_i}$ over $b_i$, there is an $A_{d_i-1}$-singularity,
%in the notation of \eqref{eq:dd'}, 
so that $\tX$ acquires a fibre of type $\I_{n_id_i}$ over each pre-image of the  branch point $b_i$.
In addition, $\tilde h\colon \tX\rightarrow \tB$ is still relatively minimal.

\item
For $P\in G$, we denote by $\sigma_P\in \Aut(\tB)$ the corresponding deck transformation of $\bar\pi$. 
At the same time, $P$ induces an automorphism $t_X(P)\in \Aut_B(X)$ as in \eqref{eq: tMW}.
Hence the product group $G\times G$ acts on $X\times_B\tB$ in a natural way: for $(P,\, P') \in G\times G$ and $(x,\,\tilde b)\in X\times_B\tB$, 
\[
(P,\,P')(x,\,\tilde b) := (t_X(P)(x),\, \sigma_{P'}(\tilde b)).
\]
By the uniqueness of 
the minimal resolution, the action of $G\times G$ lifts to $\tX$, and $\tilde h$ is equivariant 
with respect to the $(G\times G)$-action on $\tX$ and the $G$-action on $\tB$. 
\end{enumerate}
\end{construction}

The construction culminates in the following result which will form the key
to constructing properly elliptic surfaces with non-trivial $\Aut_\QQ$ or $\Aut_\ZZ$.

\begin{theorem}
\label{thm:bc}
(1) The diagonal $\Delta_G\subset G\times G$ acts freely on $\tX$.

(2) 
The quotient $S = \tX / \Delta_G$ is a non-rational elliptic surface $S\stackrel f\longrightarrow B$ such that $J(S)=X$.
More precisely, $S$ is an Enriques surface if $p_g(X)=0, G=\ZZ/2\ZZ, s=2$, and properly elliptic otherwise.
%(or an Enriques surface if $p_g(X)=0, G=\ZZ/2\ZZ, s_0=s=2$).

\smallskip

(3) Moreover,
$S$ is endowed with a $G$-action which preserves each fibre of $f$ and  acts by translations on the smooth fibres. 
\end{theorem}

\begin{proof}
(1) To confirm the free action of $\Delta_G$, it suffices to analyse the stabilizers $G_{b_i}$
at the semi-stable fibres which are the preimages of the $b\in\sB_0$,
since on the smooth ramified fibres, $\Delta_G$ acts by translations, hence without fixed points.

Letting $g=(P,P)\in\Delta_G$ for $O\neq P\in G_{b}$, 
we infer from the discussion preceding the theorem
 that $g$ rotates the fibre components of an $\I_{n}$ fibre of $X\times_B \tB$ by $2\pi/d$
in the notation of Construction \ref{constr} (3)
(because $t_P$ does so on $X$).
It follows that the same holds for the $\I_{nd}$ fibre of $\tX$.
Hence $\Delta_G$ acts freely everywhere.

(2) Since $\Delta_G$ respects the elliptic fibration on $\tX$,
$S$ is elliptic. 
More precisely, it inherits fibres of multiplicity $d_i$ at each branch point $b_i\in\sB_0$, and similar at $\sB\setminus\sB_0$,
depending on the chosen monodromy.
It then follows from the canonical bundle formula \eqref{eq: canonical bundle formula} 
that $S$ is properly elliptic -- or Enriques in the given special case.
Due to the multiple fibres, the fibration $f$ admits no sections,
but $P$ maps to a multisection of index $|G|$.
This already shows that $J(S)=X$.

(3) The $G$-action then follows generally (regardless of the given construction) from the Jacobian since $G\subset\MW(J(S))$.
Explicitly we can exhibit it using the group isomorphism 
\[
G\cong (G\times G)/\Delta_G,\quad P\mapsto (P,\, O)\mod 
\Delta_G
\]
which endows $S$ with a $G$-action via the $(G\times G)$-action on $\tX$, 
which commutes with the $\Delta_G$-action, thus descends to $S$,
preserves each fibre of $\tX$ (and thus of $f$) and acts by translations on the smooth fibres. 
\end{proof}

\begin{rmk}
Elliptic surfaces with given Mordell--Weil groups can  be obtained from classifications 
of rational elliptic surfaces 
(see \cite{OS90}, 
 \cite[Table 8.2]{SS19})
 or elliptic K3 surfaces (see \cite{Shimada}).
 Of course, the Mordell--Weil group can also be read off from explicit equations, 
 or inferred from universal constructions (cf.\ \cite{Kubert})
 as we shall exploit in Section \ref{s:pf1-ii}.
% especially those for extremal rational elliptic surfaces (i.e.\ with finite 
%Mordell--Weil group), by Beauville \cite{Bea82} and Miranda--Persson \cite{MP86}.
\end{rmk}

In order to detect numerical trivial automorphisms of $S$, 
we first determine the generators of $\Num(S)_\QQ$.

\begin{lem}\label{lem: Num generators}
In the notation of Construction~\ref{const: log transform}, let $\varpi: \tX \to S$ denote the quotient map.
Suppose that $Q_1,\dots, Q_m\in \MW(X)$ generate the $\QQ$-vector space $\MW(X)_\QQ$. For $P\in \MW(X)$, let 
$\tP:=\pi^{-1}(P)\subset \tX$ be the inverse image of $P$ and $P_S := \varpi(\tP)\subset S$ the image of $\tP$. Then the $\QQ$-vector space $\Num(X)_\QQ$ is generated by the classes of $O_S$, $(Q_j)_S$, $1\leq j\leq m$, together with the classes of the fibre components of $f\colon S\rightarrow B$. In particular, if $\MW(X)$ is finite, then $\Num(S)_\QQ$ is generated by the class of $O_S$ together with classes of the fibre components.
\end{lem}
\begin{proof}
The pullbacks of $O_S,\, (Q_j)_S,\, 1\leq k\leq m$, and of the fibre components of $f\colon S\rightarrow B$ to $\tX$ generate the same $\QQ$-vector space as $\pi^* \Num(X)_\QQ$, which has the same dimension as $\Num(S)_\QQ$, hence the assertion follows.
\end{proof}

\begin{proposition}
\label{prop:G=triv}
Suppose that the group $G$ is nontrivial in Construction~\ref{const: log transform}. Then
$G$ acts trivially on the cohomology ring $H^*(S, \QQ)$ if and only if there are no additive reducible fibres.
In particular, this holds true if $|G|>4$.
\end{proposition}

\begin{remark}\label{rmk: free?}
Construction \ref{const: log transform} is quite a bit more flexible in the sense
that the base change only requires stabilizer subgroups $H\subset G_b$  at the semi-stable fibres 
to yield a free $G$-action 
 (cf.\ the pointer in Remark \ref{rem:8211}).
However, the action will be numerically trivial if and only if $H= G_b$ at each $b\in\sB_0$,
for otherwise  $t_S(P)$ does not preserve each component of the respective fibre for some $P\in G$
(as will be visible in the following proof).
%Moreover, the free action also allows for settings with other $N, N'$
%as long as a product structure as in \eqref{eq:N=N'} or \eqref{eq:dd'} holds true at every semi-stable fibre.
%\wenfei{The freeness of the $\Delta_G$-action on $\tX$ is due to (a) of Construction~\ref{const: log transform}, so it is born with it.}

\end{remark}

\begin{proof}[Proof of Proposition \ref{prop:G=triv}]
Recall that the presence of an additive fibre implies that $|\MW_\tor|\leq 4$ 
by \eqref{eq:FF}, hence the condition in the second statement.
Indeed, for a section $P\in\MW_\tor$,
$t_X(P)$ permutes the simple  components of any additive fibre non-trivially by Fact \ref{fact}.
Since there is no ramification at the additive fibres, the same holds on $\tX$ and on $S$.

Conversely, the construction, and the condition on the additive fibres, ensures that any $P\in G$ intersects the fibre $h^*b$ at the same component as $O$ for each 
$b\in B\setminus\sB_0$. Hence $t_S(P)$ preserves each component of the fibre $f^*b$ for $b\in B\setminus \sB_0$. 
For $b\in \sB_0$, the same holds by construction. 
By Lemma~\ref{lem: Num generators}, 
in order to show that $G$ acts trivially on $\Num(S)$, it suffices to check that, for any $P\in G$ and $Q\in \MW(X)$, one has
\[
t_S(P)(Q_S)= Q_S
\]
where $Q_S:=\varpi(\pi^{-1}(Q))$. But this also follows from the construction: Using the notation of Lemma~\ref{lem: Num generators}, we have
\[
Q_S=\varpi(\tQ) = \varpi\left(\bigcup_{R\in G} (\tQ+\tR)\right)
\]
where $\tQ+\tR$ denotes the addition in $\MW(\tX)$, and hence
\begin{multline*}
 t_S(P)(Q_S) =t_S(P)\circ\varpi(\tQ) =t_S(P) \circ \varpi\left(\bigcup_{R\in G} (\tQ+\tR)\right) \\  
 =\varpi\left(\bigcup_{R\in G} (\tQ+\tR+\tP)\right)= \varpi\left(\bigcup_{R\in G} (\tQ+\tR)\right) = Q_S.
\end{multline*}
Since $t_S(P)|_F$ is a translation for a smooth fibre $F$ of $f$, one sees that $t_S(P)$ acts trivially on the transcendental part of $H^2(S, \QQ)$. In conclusion, $G$ acts trivially on the whole $H^2(S, \QQ)$. 

Since $\chi(S)>0$, the pullback $f^*\colon H^1(B, \QQ)\rightarrow H^1(S, \QQ)$ is an isomorphism. Now that $G$ induces the trivial action on $B$, it acts trivially on $H^1(S, \QQ)$. The proof of the numerical triviality of $G$ is thus complete.
\end{proof}

\begin{remark}
Alternatively, in the last steps of the proof of Proposition \ref{prop:G=triv}, one could also argue directly with the quotient surface $S/G$
to deduce that $q(S/G)=q(S)$ and $p_g(S/G)=p_g(S)$, much like for an isogeny of a jacobian elliptic surface.
\end{remark}

\begin{remark}
We emphasize that Proposition \ref{prop:G=triv} fits well with Theorem \ref{thm} (iii)
as, by construction, the fibration $S\to B$ has multiple fibres of type $\phantom{}_m\I_n$ for suitable $m,n>1$.
\end{remark}

\section{Numerically trivial automorphisms -- Proof of Theorems \ref{thm:p_g>0}, \ref{thm2}}
\label{s:pf1}

We are now in the position to prove the main results concerning  the numerically trivial automorphism group $\Aut_\QQ(S)$ when $p_g(S)>0$.

\subsection{Proof of Theorem \ref{thm:p_g>0} (i)}

If $p_g(S)>0$, then $\Aut_\QQ(S)|_B$ is trivial by Proposition \ref{prop: trivial Aut_Q on B}.
Hence $\Aut_\QQ(S) = \Aut_{B,\QQ}(S)$ by the exact sequence \eqref{eq:es}, restricted to $\Aut_\QQ(S)$, and Proposition \ref{lem: AutQ vs MW} shows that
\begin{eqnarray}
\label{eq:Aut-MW}
\Aut_\QQ(S) = \Aut_{B,\QQ}(S) \hookrightarrow t(\MW(J(S))_\tor)
\end{eqnarray}
(as in the proofs of Proposition \ref{prop: bd AutB(S)}
and Corollary \ref{cor: AutB bd}).
Since the rightmost group in \eqref{eq:Aut-MW} is abelian with at most 2 generators, claim (i) follows.
\qed

\subsection{Proof of Theorem \ref{thm:p_g>0} (ii)}
\label{s:pf1-ii}

Let $M$ denote the maximal order of an element in $G$ or $M=4$ if $G\subset(\ZZ/2\ZZ)^2$.
Thanks to Theorem \ref{thm:bc},
proving the first statement basically amounts to 
applying Construction \ref{constr} to the universal elliptic curve $X=X(M)$
and $G= (\ZZ/M\ZZ)^2$
(similar to what we did in the proof of Proposition \ref{prop: bd MW}; note that $G$ is splitting by Lemma \ref{lem:splitting}).
By Proposition \ref{prop:G=triv}, we get
an abundance of properly elliptic surfaces $S$ with $\Aut_\QQ(S)\supseteq (\ZZ/M\ZZ)^2\supseteq G$.

For the second statement, we need to prove that actually equality persists in the above construction:
$\Aut_\QQ(S)= (\ZZ/M\ZZ)^2$.
More generally, by Lemma \ref{lem:splitting}, 
we can also allow $G$ exactly as in Theorem \ref{thm} (ii).
Then there is a universal elliptic curve $X(G)$ without additive fibres as soon as $|G|>4$; 
some explicit equations can be found in \cite{Kubert}.
By definition (or by inspection of the invariants from Section \ref{s:pf1-iii}
which, of course, go up with the group growing), 
these surfaces have $\MW(X(G))=G$,
so \eqref{eq:Aut-MW} implies that $\Aut_\QQ(S)=\MW(X(G))=G$
 as soon as $p_g(X(G))>0$;
for instance, this holds whenever $|G|>9$.
(Note that, $X(G)$ is extremal over $\CC$, hence $\MW(X(G))$ is finite,  by \cite[Thm 5.1]{Shi72}.)

Whenever $|G|>4$ and $p_g(X(G))=0$,
i.e.\ $X(G)$ exists and is rational,
we apply a quadratic base change unramified at the singular fibres
to obtain a semi-stable elliptic K3 surface $X'$ with same $\MW_\tor$.
Applying Construction \ref{constr} to $X'$, 
we can use \eqref{eq:Aut-MW} to deduce the claim.

Alternatively, also for $|G|\leq 4$, 
we can pick a semi-stable elliptic K3 surface with the given Mordell--Weil group
from \cite{Shimada} and proceed as above.
\qed

\subsection{Proof of Theorem \ref{thm:p_g>0} (iii)}

For a jacobian elliptic surface, or if there is an additive fibre, this follows from Corollary 
\ref{cor: pg>0 add fib 2}.
Hence we may assume that all fibres are semi-stable or multiple of a smooth elliptic curve.
Using Bertini's theorem, we find a smooth irreducible multisection $D\subset S$ 
such that (quite generally)
\begin{itemize}
\item
$D$ does not pass through the singular points of $F'$ for each fibre $F$;
\item
$D$
is transversal to all the components of the singular fibres;
\item
$D$
is either transversal to a smooth fibre $F$, or simply tangent at one point at most.
\end{itemize}

Using the assumption that no multiple fibre has singular support,
the above implies that the map  $D\to B$
is unramified at the singular fibres with singular support,
 and only ramified over points whose corresponding fibre  is a smooth  multiple fibre; it  follows that
the normalization $X$ of 
 $S\times_B D$ is a smooth elliptic surface 
which, by construction, admits a section.
Since the Jacobian commutes with base change (cf.\ \cite[\S 8]{Milne}, e.g.),
we infer that $X$ is at the same time the base change of $J(S)$, i.e.\ $X=\Jac(S)\times_B D$.

Let $\varphi\in\Aut_\QQ(S)$ be a non-trivial automorphism.
By \eqref{eq:Aut-MW}, there is $P\in\MW(J(S))_\tor$ such that $\varphi=t(P)$.
 We take  the natural  lift 
of $\varphi$ to $X$ given by the fibre product of  $\varphi$ with the identity of $D$
and denote it by $\psi$, it is immediate that 
$\psi$ is the same as translation by $\tilde P$, the section of $X\to D$ induced by $P$.
But $P$ and $\tilde P$ meet some reducible fibre (of type $\I_n, n>1$) in a non-identity component by \eqref{eq:pattern},
so translation permutes fibre components and thus cannot be numerically trivial on $J(S)$ or $X$.
Since the fibre is unramified, the same applies to $S$, giving the desired contradiction.
\qed

\subsection{Proof of Theorem \ref{thm:p_g>0} (iv)}

This follows from Corollary 
\ref{cor: p_g>0 isotrivial}.

\subsection{Proof of Theorem \ref{thm: bound pg>0} (i)}
\label{s:pf1-v}

The inclusion \eqref{eq:Aut-MW} lets us directly infer (i) from Proposition \ref{prop: bd MW}
since the base curves of the elliptic fibrations on $S$ and on $J(S)$ are the same.
\qed

\subsection{Proof of Theorem \ref{thm: bound pg>0} (ii) and (iii)}
\label{s:pf1-iii}

The bounds follow along the same lines as in the proof of Proposition \ref{prop: bd MW}.
Namely, with $G\subset\MW(J(S))$ of size $|G|>4$, 
$J(S)$ is a base change of $X(G)$,
and since $S$ and $J(S)$ share the same invariants as discussed in Subsection \ref{ss:MW},
we infer that
\[
q(X(G))\leq q(S), \;\;
p_g(X(G))\leq p_g(S),\;\;
\chi(\sO_{X(G)}) \mid \chi(S) \;\; \text{ and } \;\; e(X(G))\mid e(S).
\]
Here the last two divisibilities form an easy topological consequence of the Euler number formula \eqref{eq: e(S)},
since all singular fibres are semi-stable (cf.\ Proposition \ref{prop:G=triv}).
%and $e=12\chi$ by Noether's formula.
The invariants for the given groups $G$ can be read off from
\cite[\S 4.2]{Miy06} and
 \cite{Shi72}, for instance.

For $p=2,3,5$, we start again with a semi-stable rational elliptic surface $X_0$ with $\MW_\tor\cong \ZZ/p\ZZ$,
namely $\# 44, \# 63$ and $\# 67$ from \cite[Table 8.2]{SS19}.
Applying a cyclic base change of degree $d>1$ ramified at the reducible fibre(s) and, in case there is only one reducible fibre,
at a smooth fibre, we obtain an auxiliary elliptic surface $X$ with $\MW(X)\supset\ZZ/p\ZZ$ and with $\chi(\sO)=p_g(S)+1=d$.
Applying Construction \ref{const: log transform} to $X$ gives the claim.
\qed

\section{Numerically trivial automorphisms -- Proof of Theorem \ref{thm:p_g=0}}
\label{s:pf2}
\label{s:iso}

Most of Theorem \ref{thm:p_g=0} has already been shown.
Indeed, in the non-isotrivial case,
the inequality $|\Aut_\QQ(S)|\leq 9$ of
Theorem \ref{thm:p_g=0} (i) follows from Proposition \ref{prop: pg=0 bd J(S)} (1) and (2),
and same for the other conclusions.  

Conversely, 
by Lemma~\ref{lem:splitting}, we may apply Construction~\ref{const: log transform} to $X=X_{3333}$ and $G=\MW(X)\cong(\ZZ/3\ZZ)^2$, to obtain a properly elliptic surface $S$ with $\Aut_\QQ(S)=\MW(X) = (\ZZ/3\ZZ)^2$, and hence the equality can actually be  attained.

Therefore we focus on the isotrivial case, 
where
the first bound from Theorem \ref{thm:p_g=0} (ii) $|\Aut_\Q(S)|_B|\leq s \leq P_2(S)+1$ 
is proven in Proposition  \ref{prop: pg=0 bd J(S)} (3) (a) 
for  the four special surfaces.

To extend this estimate  for all isotrivial fibrations, we note that the upper bound $|\Aut_\QQ(S)|_B|\leq 3$
for the non-special surfaces from Proposition  \ref{prop: pg=0 bd J(S)} (2)
is an equality only for  non-isotrivial fibrations by inspection of \cite[Table 3]{DM22}.
Hence $|\Aut_\QQ(S)|_B|\leq 2\leq s$ for all non-special isotrivial fibrations as required.

For these, the second bound from Theorem \ref{thm:p_g=0} (ii)
follows by 
combining the bounds for  $|\Aut_\QQ(S)|_B|$ above and for $|\Aut_{B,\QQ}(S)|$ in Proposition \ref{prop: pg=0 bd J(S)} (2)
using  \eqref{eq:es}.
The special surfaces, particularly $X_{22}$ and $X_{33}$, however, require substantially more work.

\subsection{Galois covers for isotrivial fibrations}

In this section, we assume as usual that 
$S$ is a minimal algebraic properly elliptic surface with
$\chi(S) >0$ and 
 that the pluricanonical fibration $f\colon S \ra B$
is isotrivial.

This means that  all the smooth fibres are isomorphic to each other;
it also implies that there exists a $G$-Galois covering $ p : C \ra B= C/G$ such that the pull-back $ C \times_B S $ is birational to
a product 
$C\times E$, compatibly with the respective projections onto $C$ (see \cite{serrano}, Section 2, and also the arguments in 1.5.3.1 and 1.5.3.2 of \cite{time}).

\begin{remark}
$G$ acts on $C$ and on $E$, since $f$ is a holomorphic bundle with fibre $E$ on a Zariski open set $B^*$ of $B$,  and the covering  $ p : C \ra B= C/G$  of the base corresponds to the kernel
of the monodromy with values in $\Aut(E)$. 

$G$ acts faithfully on $C$ by construction, while we may assume that $G$ also acts faithfully on $E$, otherwise 
we can take $G ' : = \Ker ( G \ra\Aut(E))$ and replace $C $ by $C / G'$.

Observe that $C$ has genus $g  \geq 2$ since $S$ has Kodaira dimension $ 1$.

$G$ does not act freely on $C\times E$, otherwise we would have $\chi(S)=0$ and be  in the situation considered in Part I, as referenced under \cite{CFGLS24}.
\end{remark}

Since $G$ acts faithfully on the elliptic curve $E$, and not just via translations
(since  the action of $G$ on the product is not free), as in Part I we write
$$ G = T \rtimes \mu_r, \ r \in \{ 2,3,4,6\},$$
where $T$ is a finite group of translations of $E$ normalized by the group $\mu_r$ of $r$-th roots of unity.

We shall henceforth identify $G$ with the corresponding group of automorphisms of $E$, and in these terms we shall describe
 the monodromy homomorphism of the covering $ C \ra C/G = B$.

We write
$$ S \sim_{\bir} (C \times E) / \De_G;$$
 details of the resolution can be found in Subsection \ref{ss:geom}.
By the minimality of $S$ and since $ Kod(S) =1$, the group of birational automorphisms of $S$, respectively of  $C \times E$,
 equals the group of biregular automorphisms.
 Hence, 
as in 3.7 of Part I we find that 
\begin{equation}\label{eq: normalizer vs AutS}
\Aut(S) = N_{\De_G} / \De_G,
\end{equation}
 where $N_{\De_G} $ is the normalizer of the diagonal subgroup $\De_G < G \times G$ inside $\Aut(C \times E)$.

We continue to follow the notation of Part I: for $\Psi \in N_{\De_G} $ we write
$$ \Psi (x,z) = (\psi_1(x), \psi_2(z)) = (\psi_1(x), \la z + \phi(x)).$$

\begin{rmk}\label{rmk: isotrivial nt}
Assume now that $\Psi$ yields a numerically trivial automorphism
$\Psi_S\in\Aut_\QQ(S)$: then (as in section 4 of Part I), both fibrations,
onto $B$, respectively onto $E/G \cong \PP^1$ are preserved.
In particular, $\phi(x) =c$, a constant.

In particular, any automorphism $\Psi_S\in\Aut_{B,\QQ}(S)$ lifts to $\Psi=(\id_C, \psi_2)\in N_{\Delta_G}$
with $\psi_2\in\Aut(E)=E  \rtimes \mu_r$ for the respective $r\in\{2,4,6\}$.

\end{rmk}

Before focusing on  the group $\Aut_{\ZZ}(S)$ in Section \ref{s:iso2}, 
we take the necessary  steps
towards the classification of the possible groups $\Aut_{\QQ}(S)$  in the isotrivial case,
starting with some easy group theoretical considerations.

Assume that $\Psi$, represented by $(\psi_1(x), \psi_2(z))$, yields an element $\Psi_S\in\Aut_{\QQ}(S)$.
To study the impact of \eqref{eq: normalizer vs AutS}
and of Remark \ref{rmk: isotrivial nt}, we observe that if
$ g (z) = \e z + t, \e \in \mu_r, \ t \in T$, then conjugation by 
$\psi_2 (z) = \la z + c$
 sends $g$ to 
$$ g' : = \psi_2 \circ g \circ \psi_2^{-1} = c + \la ( g (\la^{-1} (z-c)))= \e z + (1-\e) c + \la t.$$

Therefore $\psi_2 (z)$ normalizes the action of $G$  on $E$ if and only if
\begin{equation}\label{num}
\la T = T , \ \ \; (1-\e) c \in T \;\; \forall \e \in \mu_r.
\end{equation}

It follows, perfoming the same calculations as in the proof of Lemma 4.3 of Part I,  that 
the second part of \eqref{num} is equivalent to the following conditions:

$$
\begin{array}{c||c|c|c|c}
r & 2 & 3 & 4 & 6\\
\hline
\text{condition} &
2 c \in T &

3 c \in T &
2 c \in T &
c \in T \; (\Rightarrow  \psi_2 \in G)
\end{array}
$$

%
%\begin{itemize}
%\item
%for $r=2$, to  the condition  $ 2 c \in T$,
%\item  
%for a harmonic (Gauss) elliptic curve $E$, $r=4$,  to  $ 2 c \in T$,
%\item
%for an equianharmonic (Fermat) elliptic curve, $r=3$,  to  $3 c \in T$,
%\item
%for an equianharmonic (Fermat) elliptic curve, $r=6 $,  to  $c \in T$, hence $ \psi_2 \in G$.
%\end{itemize}

In fact, $\la $ has multiplicative order dividing respectively $2, 4, 6$, according to the case of a general elliptic curve,
a harmonic one, an equianharmonic one: and then the condition $\la T = T$ holds automatically for $r \geq 3$
since by assumption $\mu_r (T)= T$.

%\begin{remark}
%Therefore, there exists a subgroup  $H\subseteq\Aut_{\QQ}(S)$ of index at most $6$ such that 
%$\psi_2$ is represented by a translation for all $\Psi\in H$, and a subgroup of 
 %respective index at most $24, 16, 54, 6$
%for which $\psi_2$ is a translation in $T$.
%\end{remark})

\subsection{Geometry of $S$ and conditions for numerical triviality of $\Psi_S$}
\label{ss:geom}

To prepare for the study of $\Aut_\QQ(S)$, we focus on the geometric features of $S$, notably on the two induced fibrations.
 Since $G$ acts faithfully on $C,E, $  $Y : = (C \times E)/G$ has isolated
singularities, which are cyclic quotient singularities with group cyclic of order $r''$ dividing $r$, as follows:

$$
\begin{array}{c||c|c|c|c}
r'' & 2 & 3 & 4 & 6\\
\hline
\text{singularities}\phantom{\hat I} & A_1 & A_2, \frac{1}{3}(1,1) &
A_3, \frac{1}{4}(1,1) &
A_5, \frac{1}{6}(1,1) 
\end{array}
$$

 Here the fibre over the image point of $x \in C$ is the quotient $E/\tau$,
 where  $\tau$ generates  the stabilizer  of  $x \in C$. Denoting by $r'$   the order of the image of
$\tau$ inside $\mu_r$, we have that $ r''$ divides $r'$.

%For $r'=2$, we get $A_1$ singularities,  for $r'=3$,  a quotient singularity of type $A_2$ or
%one of type $\frac{1}{3}(1,1)$, for $r'=4$,  $A_3$ or $\frac{1}{4}(1,1)$, for  $r'=6$, $A_5$ or $\frac{1}{6}(1,1)$.

Take now  the minimal resolution $Z$ of the singular quotient $Y : = (C \times E)/G$: then $Z \ra B$ is an elliptic fibration,
but not necessarily relatively minimal, and one has to successively contract the exceptional $(-1)$ curves to reach $S$,
 the relative minimal model.
 
 $$
\begin{array}{ccccccc}
C \times E &&&& Z &&\\
\downarrow & \searrow && \swarrow & & \searrow &\\
C && Y  = (C\times E)/G 
&&\downarrow && S\\
& \searrow & \downarrow &&&&\downarrow\\
&& B = C/G & = & B & = & B
\end{array}
$$

More precisely, the non multiple singular fibres of the elliptic fibration $S\to B$ depend on $r'$ as follows
(compatibly with  Table 5.2 in page 105 of \cite{SS19}):
$$
\begin{array}{c||c|c|c|c}
r' & 2 & 3 & 4 & 6\\
\hline
\text{Kodaira type} &
\I_0^*
& \IV, \IV^* &
\III, \III^* &
\II, \II^* 
\end{array}
$$
Here the starred fibres (the non-reduced ones) arise from resolving $A_{r'-1}$ singularities
exclusively, thus not requiring any contractions (cf.\ Figure \ref{Fig:add}),
while the unstarred fibres (the reduced ones) 
result from a fibre with 4 components on $Z$ being successively contracted to the given configuration.
In detail, the strict transform $\tilde F$ of the fibre $F$ of $Y\to B$ always is a $(-1)$-curve
(with multiplicity $r'$)
while the self-intersection numbers of the  three exceptional curves $C_1, C_2, C_3$ are depicted below:

\begin{figure}[ht!]
\setlength{\unitlength}{.3in}
\begin{picture}(16,4)(-6.5,-1)
\thicklines

%\multiput(4,1)(2,0){2}{\circle*{.1}}

\put(2,2.7){\framebox{$r'=4$}}

\put(0,0){\line(1,0){3}}
\put(0,1){\line(1,0){3}}
\put(0,2){\line(1,0){3}}
\put(1,-1){\line(0,1){4}}

\put(1.15,-1){$\tilde F$}
\put(2.8,-.5){$C_1$}
\put(2.8,.5){$C_2$}
\put(2.8,1.5){$C_3$}

\put(.3,2.7){$-1$}
\put(-.7,-.3){$-4$}
\put(-.7,.7){$-2$}
\put(-.7,1.7){$-4$}

\put(-4,2.7){\framebox{$r'=3$}}

\put(-6,0){\line(1,0){3}}
\put(-6,1){\line(1,0){3}}
\put(-6,2){\line(1,0){3}}
\put(-5,-1){\line(0,1){4}}

\put(-4.85,-1){$\tilde F$}
\put(-3.2,-.5){$C_1$}
\put(-3.2,.5){$C_2$}
\put(-3.2,1.5){$C_3$}

\put(-5.7,2.7){$-1$}
\put(-6.7,-.3){$-3$}
\put(-6.7,.7){$-3$}
\put(-6.7,1.7){$-3$}

\put(8,2.7){\framebox{$r'=6$}}

\put(6,0){\line(1,0){3}}
\put(6,1){\line(1,0){3}}
\put(6,2){\line(1,0){3}}
\put(7,-1){\line(0,1){4}}

\put(7.15,-1){$\tilde F$}
\put(8.8,-.5){$C_1$}
\put(8.8,.5){$C_2$}
\put(8.8,1.5){$C_3$}

\put(6.3,2.7){$-1$}
\put(5.3,-.3){$-2$}
\put(5.3,.7){$-3$}
\put(5.3,1.7){$-6$}

\end{picture}
\caption{Strict transform $\tilde F$ and exceptional curves $C_1, C_2, C_3$  
%on the minimal resolution $Z$, self-intersection numbers depending on $r$
}
\label{fig:exc}
\end{figure}

%Note that case $r=6$ exactly reverts the blow-ups of a cuspidal cubic
%to a simple normal crossings divisor explained in \cite{Ha}.

This set-up imposes strong restrictions on $\Aut_\QQ(S)$.
We first record the following properties of numerically trivial automorphisms 
$\Psi_Y$ and $\Psi_S$ induced by $\Psi\in N_{\Delta_G}$:

\begin{lemma}
\label{lem:Psi_Y}
If $\Psi_S\in\Aut_\QQ(S)$, then 

(1) $\Psi: = (\psi_1, \psi_2)$ must act trivially on $ H^*(C \times E, \QQ)^G$.

(2) $\Psi_Y$ fixes each singular point of $Y$
unless possibly when the corresponding fibre on $S$ has Kodaira type $\II$.
%to a point of its   $\De_G$-orbit,
%so that, changing representative, we may assume that  $(x_0, z_0)$ is fixed for $\Psi$.

(3) For each singular point $y$ of $Y$ which is fixed, 
 $\Psi_Z$ must preserve the irreducible components of the exceptional divisor in $Z$ over $y$.

(4) Conversely, if (i), (ii), (iii) are satisfied, then $\Psi_S\in\Aut_\QQ(S)$.

%$r=6$ and there is a singularity which is not of type $A_{r'-1}$.
\end{lemma}

\begin{proof}

(1) follows immediately since $ H^*(C \times E, \QQ)^G \cong H^*(Y, \QQ) \subset  H^*(Z , \QQ) $.

(2): If $\Psi_S$ is numerically trivial, then the reducible  singular fibres are
 left fixed: by our description,  the only  irreducible fibres  are  those  of type $\II$. 
 
 Hence the fibres not of type $\II$ are left invariant.

If the corresponding fibre is starred, then the exceptional divisors have negative self intersection, hence 
cannot be permuted on $Z$ and on $S$, hence the singular points have to be fixed.

The same argument directly applies to the case $r'=3$,
since all 3 exceptional curves survive on $S$.

If $r'=4$, only two exceptional curves remain uncontracted on $S$
(namely $C_1, C_3$ in the notation of Figure \ref{fig:exc}),
so each of these is fixed by $\Psi_S$, and so are the corresponding singular points on $Y$.
But then the fibre is fixed as well, of course, as it is reducible,
and so is the remaining singular point on it.

All in all, this only leaves the possibility for a non-trivial action
in the case for $r'=6$ stated in the lemma.
(More precisely, there have to be at least two fibres of the rightmost configuration in Figure \ref{fig:exc}
as these are contracted to cuspidal cubics on $S$ which can thus be permuted by a numerically trivial automorphism, 
 but the singular points on a given fibre of $Y$ can never be permuted because the respective orders are different.)

(3): Observe that $Z$ is obtained from 
$S$ via a canonical blow up procedure, making the fibre a normal crossing divisor. Hence, if the fibre is preserved, 
then the numerically trivial automorphism preserves all the exceptional components of $Z \ra Y$.

To show (4) we  apply the theorem of Mayer--Vietoris  to the union $U$ of  tubular neighbourhoods of the exceptional divisors  and to the complement $V$ of the 
exceptional divisors
to see that $\Psi_Z$, hence also $\Psi_S$, is  numerically trivial if  these conditions are satisfied
(cf.\ \cite[Prop.\ 3.1]{FPR}).
\end{proof}

 When considering the second fibration
$S\to E/G\cong \PP^1$, we can draw the following consequence:

\begin{proposition}
\label{prop:E/G}
$\Aut_\QQ(S)$ induces a  trivial action on $E/G$.
\end{proposition}

\begin{proof}
Consider the fibration $ S \ra E/G = \PP^1$. The bona fide singular fibres 
(they can be also multiple  with smooth support) occur 
over the branch points of $ E \ra E/G$, which are 3 or 4 in number according to the following table
(in agreement with the fibre configurations in Figure \ref{fig:exc}):

$$
\begin{array}{c||c|c|c|c}
r & 2 & 3 & 4 & 6\\
\hline
\text{number of branch points} & 4 & 3 & 3 & 3\\
\text{branching indices} &
(2,2,2,2) & (3,3,3) & (4,4,2) & (6,3,2)\\
\end{array}
$$

%
%but these are $4$ for $r=2$, $3$ for $r=3$, $3$ for $r=4$, 
%but with  different branching indices $(4,4,2)$ , three for  $r=6$ with  branching indices are $(2,3,6)$. 

For $r=2$ the corresponding fibres of $S\to E/G$ correspond to different components of a fibre 
 of the elliptic fibration $S\to C/G$ of type $\I_0^*$,
 hence they cannot be permuted by a numerically trivial automorphism. 
 The analogous picture holds  for $r=3$ and $r=4$, since we have observed that
 the singular points of $Y$ on a fibre (with $r' = r$) of type different from $\II$ must be kept fixed.

 For $r=6$ we can simply appeal to the fact that the branching indices are different.
In conclusion, with three points of $E/G\cong \PP^1$ fixed in each case, 
$\Aut_{\QQ}(S)$ acts trivially  on $E/G$ as stated.
\end{proof}

The proposition has the following important consequence on $\Aut_{B,\QQ}(S)$ which 
crucially improves on the statement of \eqref{eq: normalizer vs AutS}.

\begin{cor}
\label{cor:N_GxG}
Any automorphism $\Psi_S\in\Aut_{B,\QQ}(S)$ lifts to an element $\Psi\in N_{G\times G}(\Delta_G)$,
the normalizer of $\Delta_G$ inside $G\times G$.
\end{cor}

\begin{proof}
By Remark \ref{rmk: isotrivial nt}, $\Psi_S\in\Aut_{B,\QQ}(S)$ lifts to $\Psi=(\id_C,\psi_2)$.
But then $\psi_2\in G$ by Proposition \ref{prop:E/G}.
That is, $\Psi\in G\times G$, and the claim follows.
\end{proof}

As an application we derive following fundamental restriction in terms of the centre $Z(G)$ of $G$
which will guide many subsequent considerations:

\begin{cor}
\label{cor:Z(G)}
$\Aut_\ZZ(S)\subseteq\Aut_{B,\QQ}(S) \hookrightarrow Z(G)$.
\end{cor}

\begin{proof}
The first inclusion follows from Corollaries \ref{cor: rational J(S) AutSB} and \ref{cor: p_g>0 isotrivial}.

The second follows from Corollary \ref{cor:N_GxG} applied to \eqref{eq: normalizer vs AutS}
since 
$$N_{G\times G}(\Delta_G) = \{(g_1, g_2)| g_1 g g_1^{-1}= g_2 g g_2^{-1}
%g^{g_1} = g^{g_2}
, \  \forall g \in G \} = \{(g_1, g_2)| g_2^{-1}  g_1 \in Z(G)
 \} $$
hence $N_{G\times G}(\Delta_G)/\Delta_G \cong Z(G)$.
\end{proof}

For the reader's convenience, we include the classification of the possible centres $Z(G)$ 
from \cite[Lemma 4.4]{CFGLS24}:

\begin{lemma}
\label{lem:3cases}
In the above setting, there are essentially 3 cases:
\begin{enumerate}
\item[(i)] 
$ Z (G)\subseteq T$,
the non-trivial possibilities being $(\ZZ /2)^2, \ZZ /3,  \ZZ /2$, or 

\item[(ii)] 
$G = Z(G)$, hence $G$ is abelian and a subgroup of one of
the  following groups $G_0$:
$$
\begin{array}{c||c|c|c|c}
r & 2 & 3 & 4 & 6\\
\hline
G_0 & E[2]\times \mu_2 & \ZZ/3\ZZ \times \mu_3 & \ZZ/2\ZZ \times \mu_4 & \mu_6
\end{array}
$$
\item[(iii)]  
$G = E[2] \rtimes \mu_4,  \ {\rm here} \ Z (G) =  (\ZZ /2 )\frac{1}{2} (1 + i) \times \mu_2.$
\end{enumerate}

\end{lemma}

\subsection{ $\Aut_\ZZ(S)$ acts trivially on the base $B$}

The following lemma,  dealing with property (1) of Lemma \ref{lem:Psi_Y},
will be  useful momentarily.

\begin{lemma}\label{Factorization}
Consider the factorization 
$$ C \ra D : = C/T \ra B = C/G = D / \mu_r,$$
and set $E' : = E / T$, $S' : = (C \times E)/T$.
Then:

\begin{enumerate}
\item[(i)]  
There is an equality
$$ H^*(C \times E, \QQ)^G = [H^*(C \times E, \QQ)^T]^{\mu_r} = H^*(S', \QQ)^{\mu_r} = [H^*(D, \QQ) \otimes H^*(E', \QQ)]^{\mu_r}.$$

\item[(ii)]
Condition (1): $\Psi: = (\psi_1, \psi_2)$  acts trivially on $ H^*(C \times E, \QQ)^G$
  is automatically
verified if $D$ has genus $0$.

\item[(iii)]
If $D$ has genus $\geq 2$, 
or  $D$ has genus $=1$ (and $B$ has genus $0$)
no automorphism
with $\psi_1=1$, and with  $ \psi_2$ not a translation,  is numerically trivial,
  except for the case where $D$ has genus $\geq 2$, $r=6$  and $H^1(D, \CC)$
contains all nontrivial character spaces except those which correspond to a primitive character.

\item[(iv)]

 The case where $D,B$ have both genus $1$ cannot occur. 
\item[(v)] 
$D$ has genus $0$ if and only if $D \ra B = \PP^1$ is branched only in two points with
local monodromies $\e, \e^{-1} \in \mu_r$.
\end{enumerate}
\end{lemma} 

\begin{proof}
The first statement (i)  is obvious since $T$ acts trivially on  $H^*(E, \QQ) =  H^*(E', \QQ)$,
and $H^*(D, \QQ) = H^*(C, \QQ)^T$.

 For the other statements  observe preliminarly that any product action is the identity 
 on $H^i(D, \QQ) \otimes H^{2-i}(E', \QQ)$
for $i=0,2$. 

Then it suffices to consider $$ \sH : = [H^1(D, \QQ) \otimes H^1(E', \QQ)]^{\mu_r},$$ and the second factor
$H^1(E', \CC)$ splits  by Hodge Theory as $V_1 \oplus V_{-1}$ according to the two characters $1,-1$ of $\mu_r$, 
 considered as elements of  $\ZZ/r\ZZ$.

 (ii): if $D= \PP^1$, then $ \sH =0$ and there is nothing to verify.

(iii): if $D$ has genus $\geq 2$, then since $\mu_r$ (by Lefschetz) acts faithfully on $H^0(\Omega^1_D)$, this representation contains,
if $r=2,3,4$
either $V_1$ or $V_{-1}$, hence $H^1(D, \CC)$ contains $V_1 \oplus V_{-1}$.  If $\psi_1=1$, then
the action on $V_1 \otimes V_{-1}$ is nontrivial if $\psi_2$ acts nontrivially on $H^0(\Omega^1_E)$:
hence $\psi_2$ must be a translation.

If instead $r=6$, the same argument applies if $H^0(\Omega^1_D)$ contains a primitive character. There remains the case
where $H^1(D, \CC)$ contains both $V_2 \oplus V_{-2}$ and $V_3 $ but not $V_1 \oplus V_{-1}$.

A similar argument applies if $D$ has genus $1$ and $B = \PP^1$. Since then $\mu_r$ acts on $D$ by a non translation,
faithfully, hence in suitable coordinates $ z \mapsto  \la z$ with $\la$ a primitive r-th root of $1$;  then we observe that, for $ r =2,3,4, 6$,
$\la $ can only be $\e, \e^{-1}$. 

(iv): if $D,B$ have genus $1$, then $\mu_r$ acts on $D$ via a translation, hence the action of $\mu_r$ on $D \times E'$ is
free, hence $ C \times E \ra S$ is unramified, a contradiction since $\chi(S) >0$.

(v): finally, if $D= \PP^1$,  an automorphism of order $r$ has necessarily the form $(x_0 , x_1) \mapsto (x_0, \la x_1)$
in suitable coordinates, with $\la$ an  $r$-th root of $1$.
\end{proof}

The next result and its proof illustrate once again how isotrivial fibrations lend themselves
to direct geometric arguments.
Hence we can bypass the classifications of \cite{DM22}, \cite{MP86} and \cite{SS19}
which our general results such as Corollary \ref{cor: rational J(S) AutSB} and thus also Corollary \ref{cor:Z(G)}
are based on.

\begin{theorem}\label{trivialonB}
In the isotrivial case, $\Aut_\ZZ(S)$ acts trivially on the base $B$.
\end{theorem}

\begin{proof}
 By Proposition \ref{prop: trivial Aut_Q on B}, $\Aut_\QQ(S)$ acts trivially on the base $B$
 unless $J(S)$ is rational, that is, 
  $q(S)= p_g(S)=0 \Rightarrow e(S)=12$, and  $B$ is $\PP^1$.
  
  By Kodaira's canonical bundle formula \eqref{eq: canonical bundle formula}, 
$K_S$ is given as the pull back of a divisor of degree $-1$ on $\PP^1=B$
plus the contribution $\sum_j \frac{m_j-1}{m_j} F_j$ over the multiple fibres, hence in our case of
Kodaira dimension $Kod(S)=1$ there are at least two multiple fibres.

By the cited Lemma 2.1 of Part I local monodromies which are translations (hence correspond to smooth multiple fibres) yield fixed points for the action on $B$
of the  automorphism $\s$ induced by $\psi_1 $.

By Lemma \ref{lem:Psi_Y}  the same holds true  for other local monodromies which are not
of the form 
 \begin{equation}\label{eq: two star}
   z \mapsto
   \e z + t, \ \e = \exp\left(\frac{1}{3} \pi i\right),\, 
    t\in E,
  \end{equation}
  since only for $r'=6$ we have a fibre of type II.
  
  Hence there remains  only to consider only  the case $r=6$, 
   and  $r'=6$ 
  for all local monodromies which are not translations, and more precisely we may assume that
   the number of monodromies
which are translations equals exactly $2$ and the other ones are of type \eqref{eq: two star}.

By the Zeuthen--Segre formula, fibres of type II give a contribution $2$ to $e(S)$, smooth multiple
fibres give no contribution.\footnote{ It is amusing to observe that a local monodromy of the form
%another fibre with $r'=2$ (then the monodromy is of the form $\e^3 z + t$)
%contributes $6$,  with $r'=3$ (then the monodromy is of the form $\e^{\pm 2} z + t$) contributes
%$4$ or $8$, while type $\II^*$ (then the monodromy is of the form $\e^{-1} z + t$) contributes $10$.
 $\e^h z + t$, with $ 2 \leq h \leq 5$, yields a 
contribution to the Euler number which equals $2h$.}

Hence we have  exactly $6$ points with monodromy of type \eqref{eq: two star}, and $ 2$  points whose monodromies  are translations.

 The set $\sD$ of $6$ points  with monodromy \eqref{eq: two star} is permuted by $\s$, since two points can be exchanged only if
 they have the same monodromy.
 
Now  $\s$ fixes the two points with local monodromy equal to a translation, hence $\s$  
 belongs  to a cyclic group $\subset \CC^*$, permuting the set $\sD$, and moreover all the orbits
 different from the two fixed points of $\s$ have the same cardinality.
 
 Hence  $\s$ belongs to a subgroup of the cyclic group $\ZZ/6\ZZ$, and let us assume that $\s$
 generates the subgroup $H$  induced  from the first components $\psi_1$ of  cohomologically trivial automorphisms of $S$;  
   we denote by $d  $  the order of $\s$, hence $ d \in \{ 2,3,6\}$ if $\s$ is nontrivial.

  We can choose coordinates such that $H$ acts on $\PP^1$ via $ x \mapsto \zeta x$,
 with $\zeta$ the primitive $d$-th root of $1$ with smallest argument.
 
 Let us exclude first  the case where $d=6$. For $d=6$,  then all the monodromies of points of $\sD$ are equal 
 (and changing the origin in the elliptic curve $E$ we can make  them to be
 of the form $\e z$) hence the two monodromies which are translations are equal, and they must be of the form $t, -t$,
 where $ t = \frac{1-\e^2}{3}$.

 Writing $G \cong T \rtimes \mu_6$, consider the  factorization of Lemma \ref{Factorization}
 $$ C \ra D : = C/T \ra D/ \mu_6 = C/G =B = \PP^1,$$
 and recall that  $E' : = E/T$.

 Since all the $\mu_6$ valued monodromies are equal, $D$ is the familiar Fermat sextic,
 whose equation in affine coordinates equals
 $$ \{ (x,z) \mid F(x,z): = z^6 - x^6 +1=0 \}.$$
 
 The space of holomorphic of differentials $H^0(D, \Omega^1_D) \cong H^0(\hol_D(3))$ has as basis
 $$ \Res \left( \frac{1}{F} \ dx\wedge dz \right) \ ( z^j x^i), \ i +j \leq 3.$$
 
 There are four character spaces as a representation of $\mu_6$, $V_1, V_2, V_3, V_4$ and   a basis element
 as above belongs  to $V_{1 + j}$. While the character of the group $H$
 on the same basis element equals
 $1 + i $.
 
 We want to see whether condition (1) of Lemma \ref{lem:Psi_Y} is verified.
 
 By Lemma \ref{Factorization} (i), we have
 $$ H^*(C \times E, \QQ)^G = [H^*(C \times E, \QQ)^T]^{\mu_r}  = [H^*(D, \QQ) \otimes H^*(E', \QQ)]^{\mu_r}, $$
 and we shall see that the action is nontrivial on
 $$[(H^0(D, \Omega^1_D)\oplus  \overline{ H^0(D, \Omega^1_D)}) \otimes  H^1(E, \CC)]^{\mu_6}=
 [V_1 \otimes \overline{ H^0(E, \Omega^1_E)}] \oplus [ \overline{V_1} \otimes H^0(E, \Omega^1_E)],$$
 leading to a contradiction.
 
 Whatever the action of $\psi_2$ on $H^0(E, \Omega^1_E)$, the automorphism $\Psi$ shall never act trivially
 on this subspace of the $G$-invariant cohomology, since $V_1$ contains $4$ different character spaces for the action of $\s$.
 
 Hence this action is not numerically trivial and the case $d=6$ is excluded.
 
 We  proceed quite similarly with the cases $d=2, 3$: then again, since we have $6$ monodromies of type \eqref{eq: two star}, we have a similar 
 factorization, and we look again at the subspace of the cohomology of $C \times E$ given by the pull back of 
 the cohomology of $D\times E$. The character spaces as a representation of $\mu_6$ are the same as before
 (just changes the equation $F$), and we must see how $ H = < \s > $ acts. 
 
 For $d=3$,  on $V_1$ we have the characters $1,2,0$ of $H$, and again we have a contradiction; for $d=2$,
 we have the characters $1,0$ of $H$, hence this case is also excluded.
\end{proof}

\subsection{Analysis of $\Aut_\QQ(S)$ and $\Aut_{B,\QQ}(S)$ for $r=6$}

We can now prove the remaining  of Proposition \ref{prop: pg=0 bd J(S)} (3) (b)
by considering the case where the special surface $X_{22}$ from Table \ref{table:special}
appears as Jacobian.

This will be instrumental for the proof of Theorem \ref{thm:p_g=0} (ii).

\begin{lemma}[$r=6$]
\label{lem:r=6}
If $J(S)=X_{22}$, then $\Aut_{B,\QQ}\subsetneq \mu_6$.
\end{lemma}

\begin{proof}
By Proposition \ref{prop: pg=0 bd J(S)} (3), we have 
$\Aut_{B,\QQ}(S)\hookrightarrow \Aut_{B,\QQ}(J(S))=\mu_6$.
Consider the Galois base change $C\rightarrow B$ with Galois group $G$ as above.

By  \eqref{eq: normalizer vs AutS}, any automorphism $\Psi_S\in\Aut_{B,\QQ}(S)$ lifts to $\Psi=(\id_C,\psi_2)\in N_{\Delta_G}$,
since $\Psi_S$ acts trivially on $B=C/G$.
By Remark~\ref{rmk: isotrivial nt}, we have $\psi_2=\sigma\in G_E$, 
the image of the action homomorphism $G\rightarrow \Aut(E)$. 
That is, 

$\sigma(z) = \lambda z + c$ for some $\lambda\in \mu_6$ and $c\in E$. 
Since $\Psi\in N_{\Delta_G}$, that is, $\Psi\Delta_G\Psi^{-1} = \Delta_G$, 
we infer that $\sigma$ lies in the centre of $G_E$. 
Since $S$ has Kodaira dimension 1 and $J(S)$ is rational, there must be multiple fibres of $f$. 
This implies that the translation subgroup $T$ of $G_E$ is nontrivial. 
Since $\sigma$ commutes with the elements of $T$, which are of the form $z\mapsto z+ t $, it must hold $\lambda t  = t$ for any $t\in T$. 
It is then clear that $\lambda$ cannot generate $\mu_6$, since a generator of $\mu_6$ does not fix any non-neutral point of $E$. 
It follows that the image of $\Aut_{B,\QQ}(S)\hookrightarrow \Aut_{B,\QQ}(J(S))=\mu_6$ is a proper subgroup of $\mu_6$ as stated.
%, and hence has order at most $3$.
\end{proof}

For later reference, we record the following immediate consequences, using 
 the bound $|\Aut_\QQ(S)|_B|\leq P_2(S)+1$ from 
Proposition \ref{prop: pg=0 bd J(S)} (3) (a):

\begin{cor}
\label{cor:X_22}
If $J(S)=X_{22}$, then 
$$|\Aut_{B,\ZZ}(S)|\leq |\Aut_{B,\QQ}(S)|\leq 3 \;\;\;  \text{ and } \;\;\; |\Aut_\QQ(S)|\leq 3 (P_2(S)+1).
$$
\end{cor}

In particular, this verifies the second bound $ |\Aut_{B,\QQ}(S)|\leq 4$ of Theorem \ref{thm:p_g=0} (ii)
on isotrivial surfaces,
based on Proposition \ref{prop: pg=0 bd J(S)} (2), (3) (a) and (b).

To complete the proof of the overall bound $|\Aut_\QQ(S)|\leq 3 (P_2(S)+1)$ of Theorem \ref{thm:p_g=0} (ii), 
it remains to analyse the case $r=4$, i.e.\ the special surface $X_{33}$ from 
Table \ref{table:special}.

\subsection{Local monodromies for $\Aut_\QQ(S)$}

We now switch the focus from $\Aut_{B,\QQ}(S)$ to the whole of $\Aut_\QQ(S)$.
To this end, the local monodromies of the cover $C\to C/G=B$ will play a crucial role.

\begin{lemma}
\label{lem:monos}
\label{mon-restriction}
Assume that 
$\Psi_S\in\Aut_\QQ(S)$
is induced by $\Psi : = (\id_C, \psi_2)$, and that we have a local monodromy $\s\neq\id$ over the
point $G x\in B$. 

%, and that $\s (y) = y$.
%
%Then $(x,y)$ maps to a singular point $P$ of $Y$ and the fibre 
%of the projection $Y\to E/G$
%containing this singular point is in bijection with $ E / \s$.
%Hence  condition (2)  that $\Psi_Y$  fixes $P$ is equivalent to the condition that $\psi_2(y)=y$. 

If there is a point $y\in E$ with stabilizer generated by $\s$  under the $G$-action, 
then $\psi_2\in\langle\s\rangle$.
\end{lemma}

\begin{proof}
Under the given assumptions,
$(x,y)$ maps to a singular point $P$ of $Y$.

By Lemma \ref{lem:Psi_Y}, $\Psi_Y$ fixes $P$, unless we have a fibre of type $\II$: 
but in this case, since the action on the base $B$ is trivial, the same argument applies;
hence $\psi_2(y)=y$ and thus $\psi_2\in\langle\s\rangle$ as stated.
\end{proof}

\subsection{Interplay of $\Aut_\QQ(S)|_B$ and $\Aut_{B,\QQ}(S)$ for $r=4$}
\label{ss:r=4}

We are now in the position to analyse the remaining case needed to prove Theorem \ref{thm:p_g=0}.

\begin{theorem}
\label{prop:X_33}
If $J(S)=X_{33}$, then $|\Aut_\QQ(S)|\leq 2 (P_2(S)+1)$.

Moreover   there are  infinitely many  examples with $|\Aut_{\QQ}(S)|  = 2 s =  2 (P_2(S) +1)$, 
 for each value of $s \in 2 \NN, s \geq 4$,  and with $|\Aut_{B, \QQ}(S)|  =  s/2 =  (P_2(S) +1)/ 2$.
\end{theorem}

\begin{proof}

Assuming to the contrary that $|\Aut_\QQ(S)|> 2 (P_2(S)+1)$,
we infer from Proposition \ref{prop: pg=0 bd J(S)} (3) that 
\begin{enumerate}
\item
$\Aut_\QQ(S)|_B$ is cyclic of order $s=P_2(S)+1\geq 2$, equalling the number of multiple fibres,
\item
$\Aut_{B,\QQ}(S) \cong \Aut_{B,\QQ}(J(S)) \cong \mu_4$, and
\item
$J(S)=X_{33}$, i.e.\ $r=4$.
\end{enumerate}
Since  $\Aut_{B,\QQ}(S)\cong \mu_4$ centralizes $G$ by Corollary \ref{cor:Z(G)}, 
we deduce from Lemma \ref{lem:3cases} that $T \cong \ZZ/2\ZZ$,
generated by $t_0: =  \frac{1 + i}{2}$; in fact, $G = T \times \mu_4$.
More precisely, by the trivial action on the base, we can lift a generator $\Psi_S$ of $\Aut_{B,\QQ}(S)\cong \mu_4$
to $(\id_C, \psi_2)$.
Here $\psi_2\in G$ is of order $4$; 
after a change of coordinates
in the elliptic curve $E$, we can write $\psi_2= \s $ where $\s(z)  : = i z $.

It follows from Lemma \ref{Factorization} (iv) 
that $D =C/T \cong \PP^1$, and that $C$ is a double covering of  $D$ branched in $4s$ points (above the $s$ multiple fibres which are in fact double).

 The monodromy 
 of the covering $C\to C/G = B$
 with values in $G = T \times \mu_4$  must then be 
$(\s, 
 \s^{-1}, 
 \tau,  \hdots, \tau)$ by Lemma \ref{lem:monos};
here $\tau$ denotes translation by $t_0$,
 the first two branch points correspond to the reducible fibres, say at $0, \infty$,
and the other $s$ branch points correspond to the double fibres.

In particular, we infer that $s$ is even and confirm,  by looking at the two reducible fibres,
that $b_2(S)=10$, hence $q(S) = p_g(S)=0$.

With the cyclic order $s$ action of $\Aut_\QQ(S)|_B$, we may assume that  the other branch points   are
$ \{1, \zeta, \zeta^2, \dots, \zeta^{s-1}\}$, where $\zeta \in \mu_s$ is primitive.

Over the point $0$ there lie two points $x_1, x_2\in C$
which are exchanged by $\tau$ and  left fixed by $\s$; likewise over $\infty$, there
lie two points $y_1, y_2\in C$ which are exchanged by $\tau$ and  left fixed by $\s$.

Note that the 
two $A_3$ singularities, corresponding to $x_1, x_2$ on $C$, say, 
are interchanged by the automorphism induced by  $\tau\times\id_E$ and fixed by 
those induced by $\sigma\times\id_E$,
hence $\Psi_S$ preserves all components of the resulting $\III^*$ fibre on $S$,
and same for the components of the $\III$ fibre corresponding to $y_1, y_2$
(corresponding to $C_1, C_3$ in Figure \ref{fig:exc}).

We infer that
$\Psi_S\in\Aut_\QQ(S)$ by construction,
since $\Psi_S$ 
preserves the fibre components which generate $\Num(S)\otimes\QQ$ together with any multisection
(because $\MW(J(S))$ is finite).

 By construction, $\Aut_\QQ(S)|_B$ acts on $B=\PP^1$ by the automorphism $\phi:  \xi \mapsto \zeta \xi$  which preserves the branch set and  the monodromy.
Hence $\phi$ lifts to an automorphism $\phi'$ of $C$ which centralizes $G$ 
 (see for instance \cite{topmethods}, remark 2.8 of section 6.1 implying the existence of a lift  
 $\phi'$ to $C$: then the condition that the monodromy is preserved amounts to an
 exact sequence $ 1 \ra G \ra \pi_1^{orb} / H \ra \langle \phi' \rangle \ra 1$, where $\phi'$ 
 conjugates $G$ trivially).
We will see that 
 this lift has order $s'=4s$.

For this purpose, consider the abelian group $G'=\langle \phi', G\rangle$.
Since this contains $G$ as a normal subgroup, we have a composition of covers
\[
C \to C/G \cong \PP^1 \to \PP^1/\langle\phi\rangle \cong C/G'
\]
which as a whole is branched in three points, say $\{0, \infty, 1\}$.
By construction, there is ramification of order $4s$ at the former two branch points.
This implies not only that 
$$
G'\cong \ZZ/4s\ZZ \times\ZZ/2\ZZ = \langle\phi'\rangle\times\langle\tau\rangle,
$$
but also gives the local monodromies $\{\phi',\phi'^{-1}\tau, \tau\}$.
In particular, we infer as in the above argument for $\Psi_S\in\Aut_\QQ(S)$ that 
neither the induced automorphism $\Phi'\in\Aut(S)$ nor $\tau_S$
fixes the fibre components  at $\infty$, so $\Phi', \tau_S\not\in\Aut_\QQ(S)$,
and same with $\Phi'\circ\tau_S$ at $0$.
Hence only a subgroup of $G'$ of size at most $2s$, namely $\langle\Phi'^2\rangle$, 
acts numerically trivially,
giving the required contradiction.

Considering  now, for even $s$,  the  Galois covering 
 $ C \ra \PP^1$ branched on three points, namely
 $\{0, \infty, 1\}$, with Galois group $\mu_{4s} \times \ZZ/2$,
 and local monodromies  $\{\phi',\phi'^{-1}\tau, \tau\}$, 
 and  the  Galois covering $ C \ra C/G$ corresponding 
 to the normal subgroup $G= \langle (\phi')^s, \tau \rangle$,
 we obtain the desired examples $S \sim_{bir} (C \times E)/G$ showing
 that the cases 
with $|\Aut_{\QQ}(S)|  = 2 s =  2 (P_2(S) +1)$, 
 for each value of $s \in 2 \NN, s \geq 4,$ are realized.

\end{proof}

\subsection{Examples with $|\Aut_\QQ(S)|=2(P_2(S)+1)$}
\label{rem:X_33}

In line with Theorem \ref{thm:X_33},
it is instructive to note that there are both examples with $|\Aut_{B,\QQ}(S)|=4$
and with $|\Aut_\QQ(S)|_B|=s$, but equality cannot be attained simultaneously. To see this in explicit examples, we follow \cite{KS23}
to consider the affine surface 
\begin{eqnarray}
\label{eq:Y-d}
S_0: \;\; 
y^2 = g x^4 + g^3 t
\end{eqnarray}
where $g\in\C[t]$ generally can be any polynomial without multiple roots.
If $t\nmid g$, then the minimal model $S$ has an elliptic fibration
with singular fibres of type $\III$ at $t=0$, $\III^*$ at $\infty$,
and double fibres with smooth support at the roots of $g$.
(These arise from the resolution of the elliptic singularities at $(0,0)$
in the fibres above the roots of $g$ as in \cite[\S9]{KS23}.)
Note that $J(S)=X_{33}$ by construction.

To continue, we distinguish by the parity of $\deg(g)$.

\begin{example}[even degree]
\label{ex:here-even}
If $\deg(g)=2m$, then the model \eqref{eq:Y-d} compactifies to a singular  hypersurface of degree $6m+4$ in weighted projective space 
$\PP[1,1,m+1,3m+2]$.
At $\infty$, we thus get a local equation
\begin{eqnarray}
\label{eq:Y-d'}
S_0: \;\; 
y'^2 = g' x'^4 + g'^3 t'^3
\end{eqnarray}
which reveals an $E_6$ singularity whose exceptional curves connect with
the strict transforms of the two components given by $t'=y'\pm
 \sqrt{g'(0)}
 x'^2=0$ 
to a fibre of Kodaira type $\III^*$.
We  infer that the order $4$  automorphism $\Psi: (x,y,t)\mapsto (\sqrt{-1}x,-y,t)$ 
preserves each fibre component and is thus numerically trivial, i.e.\ $\Psi\in\Aut_{B,\QQ}(S)$
as in the proof of Theorem  \ref{prop:X_33}.
\end{example}

\begin{example}[odd degree]
\label{ex:here-odd}
If $\deg(g)=2m+1$, then the model \eqref{eq:Y-d} compactifies again to a singular  hypersurfaceof degree $6m+4$ 
in weighted projective space 
$\PP[1,1,m,3m+2]$.
However, at $\infty$, we  get a different kind of local equation
\begin{eqnarray}
\label{eq:Y-d''}
S_0: \;\; 
y'^2 = t'^3g' x'^4 + g'^3
\end{eqnarray}
with an $E_6$ singularity not visible in this chart,
but again, its exceptional curves connect with
the strict transforms of the two components given by $t'=y'\pm \sqrt{g'(0) ^3
}=0$ 
to a fibre of Kodaira type $\III^*$.
We thus infer that the order $4$  automorphism $\Psi: (x,y,t)\mapsto (\sqrt{-1}x,-y,t)$ 
does not preserve the fibre components at $\infty$ (but $\Psi^2$ does, of course). %, as in Subsection \ref{sss:odd_s}).
\end{example}

\begin{lemma}
\label{lem:2s}
Setting $g=t^s-1$ in the above examples, we obtain $|\Aut_\QQ(S)|=2s$.
\end{lemma}

\begin{proof}
This follows from the order $s$ automorphism on the base.

In detail, if $s$ is even, this extends to an order $4s$ automorphism $\Phi'$ on $S$ 
which can be given by
$$
\Phi': \;\; (t,x,y) \mapsto (\zeta^4 t,\zeta x, \zeta^2 y) \;\;\; (\zeta\in\mu_{4s}),
$$ 
but this is not numerically trivial, as predicted by Theorem \ref{prop:X_33}. 
Hence only $\Phi'^2\in\Aut_\QQ(S)$, with induced action of order $m=s/2$ on the base.

If $s$ is odd, the  action on the base, of order $s$, extends to an automorphism 
$$
\Aut_\QQ(S) \ni \Phi: \;\; (t,x,y) \mapsto (\xi^4 t,\xi x, \xi^2 y) \;\;\; (\xi \in \mu_s)
$$ of the same order $s$, 
but then we have  $\Psi^2\circ\Phi\in\Aut_\QQ(S)$ which is of order $2s$ as stated
(but  $\Psi\circ\Phi\not\in\Aut_\QQ(S)$).
%
%
%In comparison, in the even degree case, setting $g=t^{2m}-1$ allows
%for a $\mu_{2m}$-action on the base which extends to an order $8m$ automorphism $\Phi'$ on $S$,
%but this is not numerically trivial. Hence only $\Phi'^2\in\Aut_\QQ(S)$, with induced action of order $m=s/2$ on the base.
\end {proof}

Note that the above  lemma provides the examples for any $s>1$ whose existence was stated in Theorem \ref{thm:p_g=0}.

\subsection{Completion of proof of Theorem \ref{thm:p_g=0}}

We complete the proof of Theorem \ref{thm:p_g=0} (ii) by exhibiting the examples needed for the remaining existence statement:

\begin{prop}\label{order3}
For any $s\in 3\NN$, there is a properly elliptic surface $S$
with $\chi(S)>0$, with $s$ triple fibres and with $|\Aut_\QQ(S)|=3s = 3 (P_2(S)+1)$.
% $S$ with group $G = \ZZ/3\ZZ \rtimes \mu_6$
% with $\Aut_{\ZZ}(S)$ of order $3$.
\end{prop}

\begin{proof}

We construct $S$ as quotient of a product $C\times E$ by 
$$G = \ZZ/3\ZZ \rtimes \mu_6 = \langle\eta\rangle\rtimes\langle\e\rangle,
\;\; \eta = \left(\frac{1 + \e}{3} \right).
$$
Here $E$ is the Fermat elliptic curve and we let $C \ra \PP^1$ be the Galois $G$-covering with $s+2$ branch points and monodromies 
\begin{eqnarray}
\label{eq:monos}
(\e, \e^{-1}, \eta, \dots, \eta)
\end{eqnarray}
%, \ \s (z) : = \e z, \ \eta_j \in T : = \ZZ/3\ZZ \left(\frac{1 + \e}{3} \right), \ \sum_j \eta_j = 0.
compatible with Lemma \ref{lem:monos}.

As before, $S$ is birational to $Y : = (C \times E)/G$, and the pluricanonical fibration $f\colon S \ra C/G\cong\PP^1$
has the first  singular fibre of type $\II$ (a cuspidal cubic) and the second one of type $\II^*$. %, that is, of type $\tilde{E_8}$.
The other  singular fibres are $ 3 F'_1, \dots, 3 F'_s$, where each $F'_j$ is again isomorphic to
the Fermat   elliptic curve.

Clearly $\e^2\times\id_E$ descends to an automorphism $\Psi$ of $S$ because $\e^2$ centralizes $G$.
Note that it is automatically numerically trivial because the 
$\II^*$ fibre does not admit any symmetries.
Hence $\Psi$ preserves each fibre component,
and same for some multisection,
but together they form a $\QQ$-basis of $\Num(S)$ because $\MW(J(S))$ is finite.
(We will discuss the case over $\ZZ$ briefly in Section \ref{ss:pf_iso}.)

Now assume that the branch points are (in the order of the monodromies in \eqref{eq:monos})
$0, \infty, 1, \zeta, \zeta^2,\hdots, \zeta^{s-1}$ for some primitive $s$-th root of unity.
Then multiplication by $\zeta$ preserves the branch locus and the monodromies,
hence lifts to an automorphism of $C$ and descends to an automorphism $\Phi$ of $S$.
We record the following two properties:
\begin{itemize}
\item
$\Phi\in\Aut_\QQ(S)$ by the same argument as for $\Psi$ above;
\item
$\Phi|_B$, and thus 
$\Phi$, has order a multiple of $s$
by construction.
\end{itemize}
%Since $\Phi$ commutes with $\Psi$, 
From their actions on the base  we infer that
$|\langle\Psi, \Phi\rangle|\geq 3s$. But then Corollary \ref{cor:X_22} implies equality.
\end{proof}

\begin{remark}
\label{rem:r=2,3}
One can construct similar examples with large order numerically trivial action on the base
for $r=2, 3$ either directly as a suitable Galois quotient
or by applying Construction \ref{constr} to $X_{11}(\lambda)$ and $X_{44}$.
(For $X_{33}$, this has effectively been carried out in Examples \ref{ex:here-even}, \ref{ex:here-odd}.)
\end{remark}

\section{Cohomologically  trivial automorphisms in the non isotrivial case -- Proof of the order 2-statement of Theorem \ref{thm:ct}}
\label{s:2}

The general statements
 (i) and (ii)
of Theorem \ref{thm:ct} follow from Corollaries \ref{cor: rational J(S) AutSB} and \ref{cor: pg=0 Aut_Z}.
In this section, we construct properly elliptic surfaces $S$ with $\Aut_\ZZ(S)\cong\ZZ/2\ZZ$ 
in abundance,  confirming this part of Theorem \ref{thm:ct} (iii).
That is, for any $s\in\NN$, we construct a $2s$-dimensional family of non-isotrivial elliptic surfaces with $p_g=q=0$ admitting a cohomologically trivial involution.

Let $h\colon X\rightarrow B=\PP^1$ be a rational elliptic surface with a 2-torsion section $P$ and a single type $\I_8$ reducible fibre $h^*\infty = \sum_{i=0}^7\Theta_k$ at $\infty$ such that $P$ intersects $\Theta_4$. Such $X$ depends on one parameter up to the standard action by $\C^*\times\GL(2,\C)$;
to see this, just use the Weierstrass form 
\begin{equation}
\label{eq: 8211}
S\colon\quad y^2 = x(x^2+2a_2x+1)
\end{equation}
where $a_2\in\C[t]$ has degree two,
 the $\I_8$ fibre has been located at $\infty$, and $P=(0,0)$.

Let $\tB\rightarrow B$ be a double cover branched at $\infty$ and above $2s-1$ points $b'_2,\hdots,b'_{2s}$
outside the zeros of the discriminant locus given by $a_2^2-1$. 

Applying Construction~\ref{const: log transform}, we obtain a $2s$-dimensional family
of non-isotrivial elliptic surfaces $f\colon S\rightarrow B$ with an involution $\varphi = t_S(P)$ induced by $P\in \MW(X)$. 
 By 
 %Lemma~\ref{lem: G nt}
 Proposition \ref{prop:G=triv}, 
 $\varphi$ is numerically trivial.

There are $2s$ double fibres of $f$:
\begin{itemize}
\item
at $\infty$, there 
is the fibre $F_1=2F_1'$ whose support is of type $F_1'=\I_8$;
\item
at $b_2',\hdots,b'_{2s}$, the fibre $F_i=2F_i'$ has smooth support ($F_i'=\I_0$)
\end{itemize}
%say $2\I_8 = 2F' = 2F'_1, 2F'_2,\hdots, 2F'_{2s}$
%(each of the latter having smooth support, i.e.\ of type $2\I_0$)
%and 
By the canonical bundle formula,
$$
K_S = -F + F'_1+\hdots+F'_{2s}.
$$
where $F$ is any fibre of $f$. The invariants of $S$ can be readily computed:
\[
p_g(S)=q(S)=0, \quad h^0(2K_S)=2s-1
\]
If $s>1$, $S$ is a properly elliptic surface; if $s=1$, $S$ is the Enriques surface first studied by Barth--Peters \cite{BP83}
(as can be inferred from Figure \ref{fig}).

\begin{prop}
\label{prop: coh trivial inv 1}
\label{prop:2}
The involution $\varphi$  is cohomologically trivial.
\end{prop}

\begin{proof}
The construction endows us with 10 curves on $S$ with square $(-2)$:
\begin{itemize}
\item
the eight components $\Theta_0,\hdots,\Theta_7$ of the $2\I_8$ fibre,
numbered cyclically such that 
\item
the bisection $O_S$ which is the image of $O$ and $P$ meets $\Theta_0$, and
\item
another bisection $R_S$
which is the image of the missing generator $R$ of $\MW(X)$ %(and of $R+P$) 
and which will be shown to meet $\Theta_4$.
\end{itemize}
The section $R$ on $X$ has height $1/2$ by \cite[Table 8.2]{SS19}, and is
thus of infinite order. In the degenerate case with an additional reducible fibre, of type $\I_2$,
it becomes a torsion section of order $4$ such that $2R=P$.
Explicitly, $R=(u,v)$ is determined, up to sign and translation by $P$, 
by the $x$-coordinate $u\in \C$
such that the right-hand side of \eqref{eq: 8211} is a perfect square in $\C[t]$ (namely $v^2$).
In particular, the construction verifies that 
$R$ meets 
$\Theta_2$ on $X$ (and $\Theta_4$ on $\tX$, but there it's the fourth component out of 16).

Note also that $O,\, P,\, R$ and $P+R$ are all disjoint on $X$, and so are $O_S, R_S$ on $S$, 
but they are not smooth rational if $s>1$
(since they map $2:1$ to the base curve $\PP^1$ with $2s$ ramification points), 
so we better not call them $(-2)$-curves.
We depict these 10 curves in the following diagram:

\begin{figure}[ht!]
\setlength{\unitlength}{.6in}
\begin{picture}(4,2.5)(0,-.2)
\thicklines

%I_6

\multiput(1,0)(1,0){3}{\circle*{.1}}
\multiput(1,2)(1,0){3}{\circle*{.1}}
\multiput(0,1)(1,0){2}{\circle*{.1}}
\multiput(3,1)(1,0){2}{\circle*{.1}}

%\multiput(4,1)(2,0){2}{\circle*{.1}}

\put(1,0){\line(1,0){2}}
\put(1,2){\line(1,0){2}}
\put(0,1){\line(1,0){1}}
\put(3,1){\line(1,0){1}}
\put(1,2){\line(0,-1){2}}
\put(3,2){\line(0,-1){2}}

\put(-.2,.75){$O_S$}
\put(4,.75){$R_S$}
\put(3.05,.75){$\Theta_4$}
\put(1.05,.75){$\Theta_0$}
\put(1.05,-.25){$\Theta_7$}
\put(2.05,-.25){$\Theta_6$}
\put(3.05,-.25){$\Theta_5$}
\put(1,2.1){$\Theta_1$}
\put(2,2.1){$\Theta_2$}
\put(3,2.1){$\Theta_3$}

\end{picture}
\caption{Configurations of ten curves, forming a $\QQ$-basis of $\Num(S)$}
\label{Fig:R}
\label{fig}
\end{figure}

Together these 10 curves generate an index 4 sublattice $L$ of $\Num(S)$ 
(by inspection the determinant $-16$ of their Gram matrix
since $\Num(S)$ is unimodular).\footnote{ Alternatively, consider the auxiliary lattice $L_0$ generated by the bisection $O_S$ and the fibre components $\Theta_i$.
Then $L_0\cong U \oplus A_7$ has determinant $8$,
and we can mimic the approach from the theory of Mordell--Weil lattices by applying the orthogonal projection 
$\pi: L\otimes\QQ \to (L_0\otimes \QQ)^\perp$ to get
$\pi(R_S) = R_S-O_S-2E-a_4^\vee$
where $a_4^\vee$ denotes the dual vector of the middle vertex in the Dynkin diagram $A_7$.
Then $\pi(R_S)^2=-2$, yielding $\det(L) = \det(L_0)\cdot\pi(R_S)^2=-16$.}
For a $\ZZ$-basis of  $\Num(S)$, we would like to
complement them by fractions of two isotropic vectors 
which are motivated from the Enriques surface case $s=1$:
\begin{eqnarray*}
D_1 & = & 2O_S + 4 \Theta_0+3(\Theta_1+\Theta_7)+2(\Theta_2+\Theta_6)+\Theta_3+\Theta_5 \;\; (\text{type } \III^*),\\
D_2 & = & O_S + \Theta_1+ 2 (\Theta_0+\Theta_7+\Theta_6+\Theta_5+\Theta_4)+\Theta_3+R_S\;\; (\text{type } \I_4^*).
\end{eqnarray*}
More precisely, in the Enriques surface case,
each divisor is of given Kodaira type.
Hence it induces an alternative elliptic fibration $S\to\PP^1$ which, a priori, features the divisor as singular fibre or as half-pencil.
The latter alternative, however, is excluded since multiple fibres are only supported on types $\I_n\, (n\geq 0)$;
therefore, each $D_i$ is 2-divisible, and we obtain $\Num(S)$ as desired in the Enriques surface case.
We now verify the analogous statement for arbitrary $s$:

\begin{lemma}
\label{lem:D/2}
For any $s\geq 1$, $\Num(S)$ is obtained from $L$ by adjoining the dual vectors
$\frac 12 D_1 = \Theta_4^\vee, \frac 12 D_2=\Theta_2^\vee$.
\end{lemma}

The proof of the lemma will require a bit of work.
In fact, it will also pave the way to deduce that $\varphi$ is cohomologically trivial.

Since $\Num(S)$ is a unimodular overlattice of $L$,
it is encoded in an isotropic subgroup of the discriminant group $A_L = L^\vee/L$
of size $4$.
We start by exhibiting generators for $A_L$.
For this purposes, it suffices to note  the isotropic divisor
\[
D_3 = O_S + 2 \Theta_0+3\Theta_1+4\Theta_2+5\Theta_3+6\Theta_4+4\Theta_5+2\Theta_6+3R_S= 4\Theta_7^\vee.
\]
Note that the configuration is that of a divisor of Kodaira type $\II^*$,
so by the same argument as above, it would be 2-divisible on an Enriques surface, but certainly not 4-divisible.
Apparently, we get a vector $\frac 14 D_3$ of order 4 in $L^\vee/L$,
and another one by applying the horizontal  symmetry of Figure \ref{fig}:
\[
D_3' = O_S + 2 \Theta_0+3\Theta_7+4\Theta_6+5\Theta_5+6\Theta_4+4\Theta_3+2\Theta_2+3R_S = 4\Theta_1^\vee.
\]
By inspection of the coefficients of $\Theta_3$ and $\Theta_5$, the vectors $v=\frac 14 D_3, v'=\frac 14 D_3'$ 
generate $A_L\cong(\ZZ/4\ZZ)^2$
with intersection form $U(3/4)$, the hyperbolic plane scaled by $3/4$.
It is then an easy exercise to check that the only isotropic vectors mod $2\ZZ$
(for $\Num(S)$ is even by the adjunction formula) are $\pm v, \pm v'$ and the 2-torsion elements 
%$2v\equiv \frac 12 D_2, 2v', 2(v+v')\equiv \frac 12 D_1$.
$$2v\equiv\frac 12 D_3 \equiv \frac 12 D_2 \; (\mathrm{mod} \, L), \;\; 2v', \;\;
2(v+v')\equiv \frac 12 D_1 \; (\mathrm{mod} \, L).
$$
For symmetry reasons, it thus suffices to inspect the following two cases:
\begin{eqnarray}
\label{eq:2Num's}
\Num(S) =\left \langle L, \frac 12 D_1, \frac 12 D_2\right\rangle \;\;\; \text{ or } \;\;\;  \left\langle L, \frac 14 D_3\right\rangle.
\end{eqnarray}
We now focus on the isotropic vectors adjoined to $L$,
to  rule out the second alternative
and eventually verify that  $\varphi$ is cohomologically trivial.
For this purpose, we let
\[
D = \frac 12 D_1, \frac 12 D_2 \;\; (\text{in the first case) \;  or }\;\;  \frac 14 D_3  \;\; (\text{in the second case}).
\]
We start by modifying $D$ to become effective.
To this end, 
let $N\in \N_0$,
write $F'=F_1'$ for ease of notation (any other $F_i'$ would work equally well), and consider
\[
\hat D = \frac 12 D_1 + NF'.
\]
Since $\hat D.F'=D.F' = 1$, 
it follows from Riemann--Roch that 
$$\chi(\hat D) = 1 + (N-(s-1))/2.
$$
Hence $\hat D\geq 0$ for all $N\geq s-1$ 
(otherwise $-\hat D\geq 0$, but then $(-\hat D.F')=-1$ contradicting that $F'$ is nef).
Fixing the minimal such $N$, this satisfies 
\begin{eqnarray}
\label{eq:s-1}
N\leq s-1.
\end{eqnarray}

\begin{claim}
\label{claim:C}
$\hat D = C$ for some irreducible curve $C\subset S$.
\end{claim}

\begin{proof}[Proof of Claim \ref{claim:C}]
Writing $\hat D = C_1+\hdots+C_r$ for irreducible curves $C_i\subset S$,
we have to show that $r=1$. 
Since $\hat D.F'=1$, we may assume without loss of generality that $C_1$ is a bisection
while the other $C_i$ are fibre components.
Since $C_i=F'$ or $F$ are excluded by the minimality assumption for $N$,
we can only have $C_2=\Theta_j$ for some component of the $\I_8$ fibre.
Regardless of $j$,
$\hat D-\Theta_j\geq 0$ implies that $(\Theta_{j-1}.(\hat D-\Theta_j))<0$ or $(\Theta_{j+1}.(\hat D-\Theta_j))<0$,
since $D$ was defined as dual vector of some $\Theta_l$.
Therefore $\Theta_{j-1}$ or $\Theta_{j+1}$ is contained in the support of $\hat D$.
But
this shows successively that $F'\subset$ supp$(\hat D)$ -- which we excluded before.
Hence $r=1$ as claimed.
\end{proof}

\begin{claim}
\label{claim:phi(C)}
$\varphi(C)=C$.
\end{claim}

\begin{proof}[Proof of Claim \ref{claim:phi(C)}]
Note that $C^2=2N$ by definition.
Hence 
\[
C.\varphi(C)=2N,
\]
since $\varphi$ is numerically trivial.
Assuming that $\varphi(C)=C'\neq C$, we can compare this intersection number against
the number of intersection points forced upon $C$ and $C'$ by the geometry of $S$.

To this end, consider the precise action of the involution on the fibre components.
Indeed, the $2s-1$ ramified smooth fibres are fixed pointwise.
It follows that $C$ meets $S^\varphi$ in at least $2s-1$ points,
so
\[
C.C'\geq 2s-1>2N = C.C'.
\]
This gives the required contradiction to \eqref{eq:s-1} and thus completes the proof that $\varphi(C)=C$.
\end{proof}

\subsection{Completion of proof of Lemma \ref{lem:D/2}}
\label{ss:pf-lem}

Having analysed the action of $\varphi$, we come back to the proof of Lemma \ref{lem:D/2}.
For this purpose, we use the action of $\varphi$ on the ramified $\I_{8}$ fibre.
One checks that the action alternates between 
\begin{itemize}
\item
fixing a component pointwise, starting from the component $\Theta_0$ met by the bisection $O_S$, 
and continued with $\Theta_2, \Theta_4, \Theta_6$, and
\item
acting as on involution on $\Theta_1, \Theta_3, \Theta_5, \Theta_7$, only fixing 
the two nodes where the adjacent fibre components are intersected.
\end{itemize}
Explicitly, this can be seen by considering $\iota$ on $\tX$, since both $\iota$ and $t_P$ induce the same involution on $S$.
Now, for $D=\frac 12 D_1$, the curve $C$ meets $\Theta_4$, and for  $D=\frac 12 D_2$, the curve $C$ meets $\Theta_2$;
both components are fixed pointwise by $\varphi$
which is compatible with the property that $C=\varphi(C)$.
For  $D=\frac 14 D_3$, however, the curve $C$ meets $\Theta_7$ which is not fixed pointwise by $\varphi$.
Since the bisection $C$ cannot meet the nodes, its intersection point with $\Theta_7$ is not fixed,
contradicting the conclusion that $\varphi(C)=C$.
Hence $\frac 14 D_3\not\in\Num(S)$.
By \eqref{eq:2Num's}, this completes the proof of Lemma \ref{lem:D/2}.
\qed

\subsection{Completion of proof of Proposition \ref{prop:2}}

We have deduced from Proposition \ref{prop:G=triv} that $\varphi\in\Aut_\QQ(S)$.
As it also acts trivially on the base of the elliptic fibration by construction,
it remains to verify that $\varphi$ acts trivially on the representatives of a basis of $\Num(S)$ in $\Pic(S)$.
For  a basis of the sublattice $L\subset\Num(S)$, this follows by construction, as verified in Proposition \ref{prop:G=triv}.
By Lemma \ref{lem:D/2}, there are two  remaining generators $D_1/2, D_2/2$.
On these, the trivial action has been established in Claim \ref{claim:phi(C)}.
%In summary, $\varphi$ fixes a set of generators of $\Pic(S)$, so it is cohomologically trivial.
This completes the proof of Proposition \ref{prop: coh trivial inv 1}.
\end{proof}

%\begin{itemize}
%\item
%every second component of the ramified $\I_{8}$ fibre,
%starting from the component $\Theta_0$ met by the bisection $B$, and continued with $\Theta_2, \Theta_4, \Theta_6$
%(this can be seen by considering $\iota$ on $S$, since both $\iota$ and $t_P$ induce the same involution on $S$);
%\item
%the $2s-1$ ramified smooth fibres.
%\end{itemize}

\begin{remark}
For an alternative infinite series of families of non-isotrivial properly elliptic surfaces $S$
with $|\Aut_\ZZ(S)|=2$, 
one could realize the isotrivial families for $r=4$ from Theorem \ref{thm:X_33}
by way of Construction \ref{constr} applied to $X_{33}$,
amounting to the examples from Subsection \ref{ss:r=4},
and then
deform (preserving $\Psi^2$)  $X_{33}$ to $X_{321}$.
%(as hinted at in Remark \ref{rmk: Enriques1}).
%confer Remark \ref{rem:alt}.
\end{remark}

%{\color{red}
%\begin{rmk}\label{rmk: coh trivial inv 1}
%Along similar lines,
%mimicking the other two constructions of Enriques surfaces in Example~\ref{ex: Enriques2} gives properly elliptic surfaces with $p_g=q=0$ admitting numerically trivial automorphisms
%which are not cohomologically trivial.
%\end{rmk}
%}

\section{ Cohomologically  trivial automorphisms in the non isotrivial case -- Proof of the order 3-statement of Theorem \ref{thm:ct}}
\label{ex: nonisotrivial AutZ=3}
\label{s:3}

In this section, we construct a 1-dimensional family of non-isotrivial properly elliptic surfaces with $q=p_g=0$ 
admitting a cohomologically trivial automorphism of order 3,
 proving the remaining part of Theorem \ref{thm:ct} (iii).

Let $X$ be the extremal rational elliptic surface with singular fibres of types  $\I_9$ and three times $\I_1$
(No.\ 63 in \cite{OS90} or \cite[Table 8.3]{SS19}).
This has $\MW(X)=\langle P\rangle \cong \ZZ/3\ZZ$ and can be given in Weierstrass form
\[
X: \;\; y^2 + txy + y = x^3.
\]

Upon numbering the components of the $\I_9$ fibre cyclically $\Theta_0,\hdots, \Theta_8$
such that the zero section $O$ meets $\Theta_0$, we may assume that $P$ meets $\Theta_3$.
This follows either from the abstract description 
$$\MW(X) \cong \Pic(X)/\Triv(X) = \Triv(S)'/\Triv(S)
$$
from Subsection \ref{ss:MW}
where the trivial lattice Triv$(X)$ is generated by fibre components and zero section
(in parallel to the argument around \eqref{eq:3-dual}),
or from considering explicit equations as in \cite{Bea82}, \cite{MP86} or \cite[Table 8.3]{SS19}.

Let $\bar\pi\colon \tB=\PP^1\rightarrow B=\PP^1$ be a cyclic cubic base change branched at $b_1$ and $b_2$ such that $h^*b_1$ is the $\I_9$ fibre and  $h^*b_2$ is  a smooth fibre $F$. 

Applying Construction~\ref{const: log transform} to the triple $(X,\, \langle P\rangle,\, \bar\pi)$, we obtain a properly elliptic surface $f\colon S\rightarrow B$ with $p_g=q=0$, together with an automorphism $\varphi\in \Aut_B(S)$ of order $3$. 
 By 
 %Lemma~\ref{lem: G nt}
 Proposition \ref{prop:G=triv}, $\varphi=t_S(P)$ is numerically trivial.

\begin{prop}\label{prop: AutZ=3}
The automorphism $\varphi$ %of  Example~\ref{ex: nonisotrivial AutZ=3} 
is cohomologically trivial.
\end{prop}
\begin{proof}
First note that the elliptic fibration $f\colon S\rightarrow B$ has two triple fibres, of types $3\I_0$ and $3\I_9$ respectively, and a smooth rational trisection $B$ which is the image of $O$ (or $P$, or $-P$).  By construction, we have $B^2=-3$.

The fixed locus $S^\varphi$ consists of
\begin{itemize}
\item
the smooth genus one curve $E$ supporting the $3\I_0$ fibre,
\item
the components $\Theta_0, \Theta_3, \Theta_6$ of $\I_9$, numbered cyclically, as before, such that $B$ meets $\Theta_0$,
\item
the 3 intersection points $\Theta_i\cap\Theta_{i+1}$ for $i=1,4,7$,
\item
the 3 nodes of the (simple) $\I_1$ fibres.
\end{itemize}

From %the exponential sequence, 
\cite[Lemma 2.5]{CFGLS24},
we get that 
$$\ZZ/3\ZZ \cong H_1(S,\ZZ) \cong H^2(S,\ZZ)_{tor} = \Pic(S)_{tor}$$
is generated by $E-D'$ 
where $D'=\Theta_0+\hdots+\Theta_8$ is the sum of the fibre components of the reduced $\I_9$ fibre.
Note that $K_S = -F + 2 E + 2 D'$ where $F$ denotes any fibre.

To see that $\varphi$ acts cohomologically trivially, 
since it acts trivially on the first homology,
we need generators of $\Pic(S)$.
To this end, we switch to the Severi  lattice $\Num(S)$
%(since we can control the torsion in $\Pic(S)$ anyway)
and consider the standard index 3 sublattice $L$,
\begin{eqnarray}
\label{eq:Num}
\Num(S) \supset \langle E, B\rangle \oplus \langle \Theta_1,\hdots\Theta_8\rangle \cong \begin{pmatrix}
0 & 1\\1 & -3\end{pmatrix} \oplus A_8 = L
\end{eqnarray}
 which parallels the trivial lattice of a jacobian elliptic surface reviewed in Subsection \ref{ss:MW}.
Note that integral overlattices of $L$ are in bijective correspondence with isotropic subgroups of the discriminant group $L^\vee/L$.
Presently, this group is isomorphic to $A_8^\vee/A_8 = \langle\Theta_1^\vee\rangle\cong\ZZ/9\ZZ$,
with generator the dual vector $\Theta_1^\vee$ of square $-8/9$.
Since $\Num(S)$ is unimodular, we infer that it is obtained from $L$ by adjoining the class of 
the dual vector 
\begin{equation}
\label{eq:3-dual}
D_0=\Theta_3^\vee
= -\frac 13(2\Theta_1+4\Theta_2+6\Theta_3+5\Theta_4+4\Theta_5+3\Theta_6+2\Theta_7+\Theta_8);
\end{equation}
indeed,  $3D_0\in L$ and $D_0^2=-2$, confirming the claim.
% confirms that $L$ and $D_0$ generate an integral overlattice of $L$
%where $L$ has index .
% presumably non-effective (but I did not care to check).
Note that $D_0$ has three lifts in $\Pic(S)$ -- which we can treat all alike (!),
so we shall just write $D_0$ for some lift in $\Pic(S)$ or for the class in $\Num(S)$
without differentiating. Indeed:

\begin{claim}
Independent of the chosen lift, $D=D_0+B+3E$ is effective.
\end{claim}

\begin{proof}
Since $D^2=D.K_S = D.E = 1$,
Riemann--Roch gives $\chi(D)=1$, so $h^0(D)+h^0(K_S-D)\geq 1$.
But then $E.(K_S-D)=-1$, so $K_S-D$ cannot be effective,
since $E$ is nef. Hence $D>0$ as claimed.
\end{proof}

As a consequence,
we can write
\[
D = C_1+\hdots+C_r
\]
for irreducible curves $C_i\subset S$.
Since $D.E=1$, we may assume that $C_1.E=1$ (so $C_1$ is another trisection of the elliptic fibration),
while the $C_i$ are fibre components.

\begin{claim}
$r=1$, i.e.\ $D=C_1$.
\end{claim}

\begin{proof}
If $D-E\geq 0$, then we use $(D-E).B=-1$
to infer that $B\subset$ supp$(D)$,
but this is only possible if $C_1=B$.
It follows that $D$, and thus also $D_0$, is supported on fibre components and on $B$,
i.e.\ on the index 3 sublattice of $\Num(S)$ from \eqref{eq:Num}.
But this contradicts the choice of $D_0$; hence $E\not\subset$ supp$(D)$.

Assume that $\Theta_i\subset$ supp$(D)$ for some $i\in\{0,\hdots,8\}$.
Then $\Theta_{i-1}.(D-\Theta_i)<0$ or $\Theta_{i+1}.(D-\Theta_i)<0$,
and we can infer successively that each $\Theta_j$ is contained in supp$(D)$.
But this gives $D-D'\geq 0$, with the same contradiction as before.

Obviously this argument also covers the case $F\subset$ supp$(D)$ for any simple fibre $F$,
so the claim follows.
\end{proof}

To complete the proof of Proposition~\ref{prop: AutZ=3},
it remains to verify that $\varphi(C_1)=C_1$,
for then there is a set of generators of $\Pic(S)$ each of which is fixed by $\varphi$.

Assume to the contrary that
$\varphi(C_1) = C'\neq C_1$.
Then $C_1.C'=1$ by the invariance of intersection numbers.
On the other hand, 
\[
C_1\cap \varphi(C_1) \; \supset  \; C_1 \cap \mbox{Fix}(\varphi).
\]
By construction, $C_1$ not only intersects $E\subset \mbox{Fix}(\varphi)$ non-trivially,
but also $\Theta_3\subset \mbox{Fix}(\varphi)$. Hence $\# (C_1\cap \varphi(C_1))\geq 2$,
contradicting the fact that $C_1.C'=1$.
It follows that $\varphi(C_1) = C_1$ as required.
This completes the proof of Proposition~\ref{prop: AutZ=3}.
\end{proof}

\begin{summary}
With Propositions \ref{prop:2}, \ref{prop: AutZ=3} in place, the proof of Theorem \ref{thm:ct} is complete.
\qed
\end{summary}

\begin{rmk}
With more triple fibres with smooth support,
$\varphi$ remains to be numerically trivial,
but it is not obvious to decide whether it is cohomologically trivial: since multiple fibres $3E$ appear as $2E$ in the canonical divisor,
the intersection numbers go up more quickly, and the argument from the proof of Proposition~\ref{prop: AutZ=3} using $C\cap S^\varphi$ 
is not decisive.
\end{rmk}

\section{Cohomologically trivial automorphisms in the isotrivial case -- Proof of Theorem \ref{thm:iso}}
\label{s:iso2}

Throughout this section up to Corollary \ref{cor:non-ss},  we assume that 
the properly elliptic surface $S\to B$ with $\chi(S)>0$  is isotrivial.
If $p_g(S)>0$, then Corollary \ref{cor: p_g>0 isotrivial}
implies that $\Aut_\ZZ(S)\subset\Aut_\QQ(S)$ is trivial.

To prove the bound of Theorem \ref{thm:iso}, it thus remains to discuss the case $p_g(S)=0$,
i.e. the case where \ $J(S)$ is rational.
Then Corollary \ref{cor: pg=0 Aut_Z} shows that 
$$\Aut_\ZZ(S) = \Aut_{B,\ZZ}(S)\subset\Aut_{B,\QQ}(S).
$$
The bound $|\Aut_\ZZ(S)|\leq 3$  thus follows from Proposition \ref{prop: pg=0 bd J(S)} (2) and (3) (b)
outside the  special case where $J(S)= X_{33} \ (r=4)$. 

This is covered by the following result:

\begin{theorem}
\label{thm:X_33}
Assume that $J(S)=X_{33}$ and that $\Aut_{B,\QQ}(S)\cong\mu_4$.
Then $\Aut_\ZZ(S)\cong\mu_2$.
\end{theorem}

\begin{proof}

Since $\Aut_{B,\QQ}(S)\cong\mu_4$, we are necessarily in the overall setting of the proof of 
Theorem 
\ref{prop:X_33}, but where  the action on the base is trivial.
Indeed, as noted in the proof of Theorem 
\ref{prop:X_33},
we have an automorphism  $\Psi\in\Aut_{B,\QQ}(S)$ of order $4$ if and only if there is an even number $s=2m$ of double fibres.

To decide whether $\Psi\in\Aut_\ZZ(S)$, we have to take a closer look at the configuration 
of curves on $S$ in order to determine (effective) generators of $\Pic(S)$.

To this end, we supplement the two reducible fibres of types $\III, \III^*$ and the $s$ 
double fibres with two natural bisections $O_S, R_S$ induced by the 2-torsion of $E$.
On the singular model $Y$, the corresponding curve $R_Y$ meets exactly the singular points of type $A_1$
(induced by the 2-torsion points $\frac 12, \frac i2$), and $O_Y$ meets the other two singularities on each fibre, induced by $0$ and $t_0$.
On $S$, it follows that $O_S$ meets the $\III^*$ fibre transversally in the simple terminal components $\Theta_0, \Theta_6$
(cf.\ Figure \ref{Fig:add}),
and $\III$ in one point (different from the node) on each component $C_1, C_3$,
while $R_S$ intersects $\III$ in the node (which $C_2$ was contracted to, cf.\ Figure \ref{fig:exc}) 
and $\III^*$ in the central double component $\Theta_7$, as depicted in the following dual graph:

\begin{figure}[ht!]
\setlength{\unitlength}{.45in}
\begin{picture}(9,3.7)(0,-.7)
\thicklines

%I_6

\multiput(8,1)(1,0){2}{\circle*{.1}}
\multiput(3,2)(0,1){2}{\circle*{.1}}
\multiput(0,1)(1,0){7}{\circle*{.1}}

\multiput(3,1)(1,0){2}{\circle*{.1}}
\put(3,-.5){\circle*{.1}}

\multiput(8,.95)(0,.1){2}{\line(1,0){1}}

\put(3,1){\line(0,1){2}}
%\put(1,2){\line(1,0){2}}
\put(0,1){\line(1,0){6}}
\put(3,-.5){\line(2,1){3}}
\put(3,-.5){\line(-2,1){3}}
\put(3,-.5){\line(4,1){6}}
\put(3,-.5){\line(10,3){5}}

\put(3,3){\line(5,-2){5}}
\put(3,3){\line(3,-1){6}}

%\put(1,2){\line(0,-1){2}}
%\put(3,2){\line(0,-1){2}}
%
\put(0,1.15){$\Theta_0$}
\put(1,1.15){$\Theta_1$}
\put(2,1.15){$\Theta_2$}
\put(3.1,1.15){$\Theta_3$}
\put(4,1.15){$\Theta_4$}
\put(5,1.15){$\Theta_5$}
\put(5.75,1.15){$\Theta_6$}

\put(2.45,3){$R_S$}
\put(2.6,-.75){$O_S$}
\put(.1,1.8){$\III^*$}
\put(8.8,1.5){$\III$}
\put(7.85,.25){$C_1$}
\put(8.85,.25){$C_3$}

%\put(1.05,.75){$\Theta_0$}
%\put(1.05,-.25){$\Theta_7$}
%\put(2.05,-.25){$\Theta_6$}
%\put(3.05,-.25){$\Theta_5$}
%\put(1,2.1){$\Theta_1$}
%\put(2,2.1){$\Theta_2$}
\put(3.1,1.85){$\Theta_7$}

\thinlines
\put(-.2,.8){\framebox(6.4,1.4){}}
\put(7.8,.6){\framebox(1.4,.8){}}

\end{picture}
%\caption{Curves on $S$, giving a $\QQ$-basis of $\Num(S)$}
%\label{fig2}
\end{figure}

By construction, $\Psi$ preserves each multiple fibre
and acts trivially on the first homology group by Lemma \ref{lem2.1}
since $q(S)=0$.

%There are $m$ multiple fibres with multiplicities 2.

The submultiple fibres  of both fibrations, and  the other
 fibre components  (exceptional over $Y$), on which $\Psi$ acts trivially,  do not  induce a system of generators  of $\Num(S) = H^2(S, \ZZ) / \Tors$,
 but generate an index two subgroup.

 To see this, 
 just take one of the submultiple fibres $F'_1,\hdots,F_s'$ of the elliptic fibration
 and  consider the unimodular lattice 
 $$U=\langle F'_1, O_S\rangle.
 $$
The root lattices $A_1\oplus E_7$ supported naturally on the singular fibres
 by omitting $\Theta_0, C_3$ as in \eqref{eq:Triv},
 embed into $U^\perp$ once we subtract $F'_1$ from the components $\Theta_6, C_1$ met by $O_S$:
 \[
 \Num(S) \supset U \oplus A_1 \oplus E_7.
 \]
One can  complement this sublattice to $\Num(S)$ uniquely by adding a vector $u$ which
is the sum of the generators of the discriminant groups $A_1^\vee/A_1$ and $E_7^\vee/E_7$.
Note that we may assume that $u^2=-2$ 
 by representing $u$ by the sum of dual vectors 
 %$\frac 12 (C_1-F_1') + \Theta_6^{\vee}$, that is,
\begin{eqnarray*}
\label{eq:u}
u  & = & \;\;  (C_1-F_1')^\vee \; + \;\;\; (\Theta_6-F_1')^\vee\\
& = & -  \frac 12 (C_1-F_1') - \frac 12 (2\Theta_1+4\Theta_2+6\Theta_3+5\Theta_4+4\Theta_5+3(\Theta_6-F_1')+3\Theta_7)
\end{eqnarray*}

We are now in the position to prove the two statements of  Theorem \ref{thm:X_33}.

\begin{claim}
\label{claim:psi^2}
$\Psi^2$ is cohomologically trivial.
\end{claim}

\begin{proof}
Since $\Psi$ is numerically trivial, we have that $ \Psi (u) = u + \eta$ where $\eta\in\Pic(S)$ is a torsion element 
(of order $1$ or $2$).
 It follows then that $\Psi^2$ is cohomologically trivial, since $\Psi^2(u) = \Psi (u + \eta) = u + 2 \eta = u$. 
\end{proof}

\begin{claim}
\label{claim:DnotD}
$\Psi$ is not cohomologically trivial.
\end{claim}

\begin{proof}
We have to prove that $\Psi(u)\neq u$ in $\Pic(S)$.
To this end, we replace $u$ by an
auxiliary divisor supplied by the  following lemma:

\begin{lemma}
\label{lem:D=D_N}
There is an $N\in\ZZ$ with $N\leq m=s/2$ such that  
\[
D_N = u + O_S + \Theta_6 + NF_1'
\]
 satisfies $D_N\geq 0$ and $|D_N|=\{D_N\}$.
\end{lemma}

\begin{proof}
The divisor $D_N$  is effective for $N\geq m$ by Riemann--Roch,
using that $g(O_S)=m-1$ by Riemann--Hurwitz, so $O_S^2=-2$ by adjunction and $D_N^2=2(N-1)$,
while $D_N.K_S=O_S.K_S=2(m-1)$.
We claim that the  minimal $N\leq m$ such that $D_N\geq 0$ gives $|D|=\{D\}$.

To see this, consider  the following short exact sequence:
  $$ 0 \ra \hol_S (D_{k-1})\ra  \hol_S(D_{k})\ra  \hol_{F'_1} (D_{k})\ra 0.$$
 Since the cokernel is a line bundle of degree $1$ on the elliptic curve $F'_1$, 
  Riemann--Roch gives
  $h^0( \hol_{F'_1} (D_{k}))=1, $ 
  hence $$h^0(  \hol_S(D_{k})) - h^0(  \hol_S(D_{k-1})) \leq 1.$$
  Therefore there exists a minimal integer $ N \in \ZZ $ such that 
  $$h^0(  \hol_S(D_{N})) = 1, \  h^0(  \hol_S(D_{N-1}))=0.$$
as claimed, and this satisfies $N\leq m$ by the first argument.
%Picking the minimal $N\leq s$ such that $D_N\geq 0$, and setting $D=D_N$,
%we can ensure that $|D|=\{D\}$
%as in Lemma \ref{lem:D=D}.
\end{proof}

Coming back to the proof of Claim \ref{claim:DnotD}, 
we write $D=D_N$ as in Lemma \ref{lem:D=D_N}.
Then  $\Psi(u)= u$ in $\Pic(S)$ is equivalent to $\Psi(D)=D$ as divisors.

Since $D.F_1'=1$, the support of $D$ contains an irreducible bisection $C\neq O_S, R_S$, 
but it may still include some fibre components (possibly multiple times), i.e.\
\[
C = D - \sum_i \Theta_i.
\]
Regardless of the precise details, this implies that
\begin{eqnarray}
\label{eq:C<D}
C^2 \leq D^2, \;\;\; C.O_S \leq D.O_S = N-1 \;\;\; \text{ and } \;\;\; C.R_S\leq D.R_S.
\end{eqnarray}
We shall now compare $C$ against the fixed loci of $\Psi$ and of $\Psi^2$.
Starting with the latter, we have
\[
\mbox{Fix}({\Psi^2}) = \Theta_1 \cup \Theta_3 \cup \Theta_5 \cup O_S \cup R_S \cup U 
\]
where 
$U$ comprises two isolated points on each elliptic curve $F_i' \, (i=1,\hdots,s)$.
%and the curve $R_S$ is the image of the bisection on $X$ and on $\tX$ describing the pair 
%of conjugate 2-torsion points on the generic fibre of $X$ (disjoint from $O$ and $P$,
%thus meeting $\Theta_7$ of the $\III^*$ fibre and the $\III$ fibre at  $\Theta_8\cap\Theta_8'$ as indicated in Figure \ref{fig2}).
%we recall that $R_S$ is the image of the bisection on $X$ which describes the pair 
%of conjugate 2-torsion points on the generic fibre of $X$,
Since $C=\Psi^2(C)$ by Claim \ref{claim:psi^2},  $C$ intersects each $F_i'$ in one of these two points 
or in the intersection point with $O_S$ or $R_S$.

Switching to $\Psi$, the fixed locus becomes
\[
\mbox{Fix}(\Psi) = \Theta_3 \cup O_S \cup U_1 
\cup U_2
\]
where $U_1$ comprises the five remaining intersection points of  components of the reducible fibres
and $U_2 = R_S \cap (\Theta_7\cup F_1'\cup\hdots\cup F_{s}')$ consists of another $s+1$ points.

Assume that $\Psi(C)=C$. 
Since $\Psi$ interchanges the points in $U$,
the bisection $C$ can intersect each elliptic curve $F_i'$ only in  the intersection point with $O_S$ or $R_S$.
It follows that
\begin{eqnarray}
\label{eq:CBB}
C.(O_S+R_S)\geq s\geq2N.
\end{eqnarray}
 On the other hand, 
$R_S.u=0$ by the choice of $u$, so  $D.R_S=N$.

As above, it follows that $C.R_S\leq N$.
Together with $C.O_S\leq N-1$ from \eqref{eq:C<D}, this gives the required contradiction to \eqref{eq:CBB}.
Hence $\Psi(C)\neq C$, and equivalently $\Psi(u)\neq u$, as claimed.
\end{proof}

\subsection{Proof of Theorem \ref{thm:X_33}}

Note that we have just completed the proof of Theorem \ref{thm:X_33}:
$\Psi\in \Aut_\QQ(S)\setminus\Aut_\ZZ(S)$ while $\Psi^2\in\Aut_\ZZ(S)$.
\end{proof}

 \begin{remark}
 \label{rem:Enriques}
The above arguments remain valid for $s=2$ to give the one-dimensional family of Enriques surfaces 
with $\Aut_\QQ(S)\cong \ZZ/4\ZZ$ and $\Aut_\ZZ(S)\cong\ZZ/2\ZZ$
from \cite[Ex.\ 3]{MN84}.
\end{remark}

 \subsection{Proof of Theorem \ref{thm:iso}}
 \label{ss:pf_iso}

Theorem \ref{thm:X_33} implies that $|\Aut_\ZZ(S)|\leq 3$ as stated in Theorem \ref{thm:iso}.

Moreover, by varying the branch points with monodromy $\tau$ in the above construction,
we obtain, for $s>2$ even, $(s-1)$-dimensional families of properly elliptic surfaces with $s>2$ double fibres and 
$\Aut_\ZZ(S)\cong\ZZ/2\ZZ$, verifying another claim of Theorem \ref{thm:iso} (cf.\ Example \ref{ex:here-even}).

For the remaining claim, namely the existence of surfaces with $\Aut_\ZZ(S)\cong\ZZ/3\ZZ$,
we turn to the fibrations constructed in the proof of Proposition \ref{order3},
but in greater generality as evidenced by the following theorem
which completes  the proof of Theorem \ref{thm:iso}.
\qed

\begin{theorem}\label{order3Z}
There are infinitely many families of  isotrivial surfaces $S$ with group $G = \ZZ/3\ZZ \rtimes \mu_6$
and with $\Aut_{\ZZ}(S)$ of order $3$.
\end{theorem}

\begin{proof}
We let $C \ra \PP^1$ the Galois $G$-covering with monodromies 
$$(\s, \s^{-1}, \eta_1, \dots \eta_n), \ \s (z) : = \e z, \ \eta_j \in T : = \ZZ/3\ZZ \left(\frac{1 + \e}{3} \right), \ \sum_j \eta_j = 0.$$

$S$ is birational to $Y : = (C \times E)/G$, and the pluricanonical fibration $f\colon S \ra C/G$
has the first  singular fibre of type $\II$ (a cuspidal cubic) and the second one of type $\II^*$, that is, of type $\tilde{E_8}$.

The other  singular fibres are $ 3 F'_1, \dots, 3 F'_n$, where each $F'_j$ is again isomorphic to
the Fermat   elliptic curve.

The only singular points of $Y$ lie in the first two fibres of the fibration over $C/G$
(with monodromies $\sigma, \sigma^{-1}$) and in the
first and third fibres of the projection onto $E/G $, i.e.\ those with branching indices $6$ and $2$,
while there are no singular points over the point of $E/G $ with branching index $3$.

The first fibre of the rational fibration of $S \dasharrow E/G$ consists of the proper transform $\Phi_0$ of
 the first  fibre of 
the  fibration $ Y \ra E/G$, counted with multplicity $6$, and the last curves of the configurations of type $A_5$, 
respectively $A_2$, coming from the resolution of the singular points of $Y$ lying in the second fibre 
of $S \ra C/G$ and in the first fibre of $S \dasharrow E/G$. 

We remove   in the configuration of type $A_5$  the $(-2)$ curve  which  intersects $\Phi_0$, 
and we obtain from the $\II^*$ fibre a configuration of type $E_8$,
intersecting $\Phi_0$ only on $\Theta$, the last component of $A_2$ configuration
(appearing with multiplicity $2$ on the fibre).

Now, $S$ has $ p_g(S) = q(S)=0$, hence the second Betti number $b_2(S) = 10$, and $\Num(S)$
contains the  direct sum $ E_8 \oplus U'$, 
where $U'$ is the lattice generated by $ F'_1$ and $\Phi_0$, which is unimodular because 
$$ F'_1 \cdot \Phi_0= 1$$
by inspection of their multiplicities in $Y$.
Since $E_8$ is unimodular, and $F'_1$ is orthogonal to $E_8$, we can subtract from $\Phi_0$ 
the dual vector
$\Theta^\vee \in E_8^\vee =E_8$ so that $U'' = \langle F'_1, \Phi_0  -\Theta^\vee \rangle$ is still unimodular and now 
orthogonal to $E_8$.
Hence we get an orthogonal direct sum:
$$ \Num(S) =  E_8 \oplus U''.$$
Since
$\Psi$ acts trivially on $H_1(S,\ZZ)$ and also on $F'_1, \Phi_0$ and on $E_8$,
we infer that  $\Psi$ is cohomologically trivial.
\end{proof}

\subsection{Case $r=3$}

The reader may wonder whether it may not be easier to use the special isotrivial families with $r=3$
to produce examples with $\Aut_\ZZ(S)\cong\ZZ/3\ZZ$. 
For completeness, we discuss this case in full as well:

 \begin{theorem}\label{r=3isotrivial}
 \label{thm:X_44}
 For each  isotrivially fibred properly elliptic surface $S$ with $q(S)=p_g(S)=0$
and  $r=3$, the group $\Aut_{\ZZ}(S)$ is trivial.
 \end{theorem}

\begin{proof}

By Theorem \ref{thm:iso}, we have   $|\Aut_{\ZZ}(S)|\leq 3$. Presently with $r=3$, 
if equality holds, 
then $\mu_3$ centralizes $G$ 
by Corollary \ref{cor:Z(G)}, 
hence 
$G = T \times \mu_3$ 
 by Lemma \ref{lem:3cases} 
where $T=\ZZ/3\ZZ$.

We take generators $\s \in \mu_3$ such that $\s (z) = \e z$, and $\tau$ such that $\tau (z) = z + \eta$, where $ \eta =  \frac{1-\e}{3}$.

 If $\Aut_{\ZZ}(S)$ has order three,
 then $\Aut_\ZZ(S)=\Aut_{B,\ZZ}(S)$ (by Corollary \ref{cor: rational J(S) AutSB}) shows that
we can assume that it is generated  by $\Psi$,  the automorphism of $S$  induced by $\id \times \s$ on $C \times E$.

By  Lemma \ref{Factorization} (iv), (v) and 
 Lemma \ref{mon-restriction}, we may take
   $C$ to be  the  $G$-covering of  $\PP^1$ branched in $2 + n_1 + n_2 = 2 + n  $ points, $x_1, x_2, \xi_1, \dots, \xi_{n_1}, 
   \xi_{n_1+1 }, \dots, \xi_n$,  
  with local monodromies
$$ (\s, \s^{-1}, \tau, \dots, \tau, \tau^{-1}, \dots, \tau^{-1})$$
where $3\mid(n_1-n_2)$.

On $E$, the elements with fixed points are $\s, \s' : = \s \tau, \s'' : = \s \tau^2$ (the last is the inverse of  $\s^{-1}  \tau$)
and their inverses:
each has a triple of fixed points, mapping to
three points $y_1, y_2, y_3$ of $E/G = \PP^1$, which are the three branch points, with local monodromies $(\s, \s', \s'')$.

The quotient
$Y = (C \times E)/G$ has three singular points of type $\frac{1}{3} (1,1)$ lying over $(x_1, y_1)$
and  three singular points of type $A_2$ lying over $(x_2, y_1)$.

As in Subsection \ref{ss:geom}, let $Z$ be the minimal resolution of singularities of $Y$. 
Then
the fibre of $Z$ over $x_2$ contains three $A_2$ configurations, hence it is a fibre of type $\IV^*$.
Meanwhile the fibre over $x_1$ is not minimal, since the proper transform $\tilde F$ of the fibre of $Y$ over $x_1$ is a $(-1)$-curve as displayed in Figure \ref{fig:exc}.;
contracting it, we get a fibre of type $\IV$, with three components $\De_0, \De_1, \De_2$ meeting in a point $P \in S$.

Through this point $P$ pass all the fibres of the second projection 
(the rational map $\pi: S \dasharrow E/G$).
We denote by $F'_1, \dots, F'_n$ the divisors such that $ 3 F'_j$ is the fibre of $f$ over $\xi_j$, while 
$\Phi_2$ is such that $3 \Phi_2$ is  the fibre of $\pi$ over $y_2$.

Since $|G|=9$, it follows that $F' : = F'_1, \Phi : = \Phi_2$ generate a unimodular odd sublattice $U'\subset\Num(S)$;
moreover $e(S)=12 \Rightarrow b_2(S) = 10$, and we choose the divisors  $\De'_1 : = \De_1 - F', \De'_2 : = \De_2 - F'$
 to generate an
$A_2$ lattice, and  since the fibre of type $\IV^*$ consists of a central curve $\sX$ meeting $A'_1, A'_2, A'_3$,
with tails $A''_1, A''_2, A''_3$ we omit $A''_2$ in order to obtain an $E_6$ lattice.
 To make it orthogonal to $U'$ we replace $\sX$ by $\sX' : = \sX - F'$ as in the proof of Theorem \ref{thm:X_33}.

 Both $A_2$ and $E_6$ have  discriminant group $\ZZ/3\ZZ$, hence the orthogonal direct sum 
 
 $$ \Lambda' = U' \oplus A_2 \oplus E_6 \subset \Lambda = \Num (S) = H^2(S, \ZZ)/\Tors $$
 has index three.
 
 The choice of $\Lambda$ amounts to the choice of an isotropic subgroup $\cong \ZZ/3\ZZ$ in the discriminant group
 of the non unimodular lattice $U' \oplus A_2 \oplus E_6$,  which is the orthogonal direct sum $\ZZ/3\ZZ \oplus \ZZ/3\ZZ$.
 
 We take as generator of the discriminant group $(A_2)^{\vee} / A_2$ the linear form $v_1$ dual to $\De'_2$,
 which is represented by the divisor $- \frac{1}{3} ( \De'_1 + 2 \De'_2)$.
 
 Whereas we take as generator of the discriminant group $(E_6)^{\vee} / E_6$ the linear form $v_2$ dual to $A''_1$,
 which is represented by the divisor 
 $$- \frac{1}{3} ( 4 A''_1 + 5 A'_1 + 6 \sX' + 3 A'_2 + 4 A'_3 + 2 A''_3).$$
 
 An easy calculation shows that $v_1^2 = - \frac{2}{3}, v_2^2 = - \frac{4}{3}$,
 hence all such isotropic subspaces are given by the  multiples of $(v_1, v_2)$ or of  $(-v_1, v_2)$.
 
 One sees however that the second choice amounts to exchanging $\De_1$ with $\De_2$.
 Hence we go for the first choice, and let $v$ denote the cohomology class in $\Lambda$ corresponding to the pair 
 $(v_1, v_2)$ (hence $v^2 = -2$).
 
 Note that $\Tors(S) \cong(\ZZ/3\ZZ)^{n-1}$, generated by  $F'_1 - F'_2, F'_1 - F'_3, \dots, F'_1 - F'_n$.  

 \begin{lemma}
 \label{lem:D=D}
 There is an $N\in\ZZ$ with $N\leq n-1$ such that  
 $$ D : = D_N : =  v + \Phi +  N F'$$
 satisfies $D\geq 0$ and $|D|=\{D\}$.
 \end{lemma}

 \begin{proof}
 We first prove that  for $N=n-1$ we get $D_{n-1} \geq 0$. To see this, we observe some easy formulae, the first following from Kodaira's canonical bundle formula \eqref{eq: canonical bundle formula}, 
 $$ K_S = 2 (F'_1 + \dots +  F'_n) - F \sim (2n-3) F', \;
 \ v^2 = -2, \; v_1.\Phi =  v_2.\Phi =  v.\Phi =
 F'.v = 0,$$
 hence
 $$ D_{n-1}^2 = (2n-3)  = K_S.D_{n-1}.$$
 
 By Riemann-Roch (since $\chi(S)=1$) we obtain $ h^0(D_{n-1}) + h^0(K_S-D_{n-1}) \geq 1$.
 
 Since $ (K_S-D_{n-1}) F = -3$ and $F$ is nef, we conclude that $h^0(K_S-D_{n-1})=0$, 
 hence $D_{n-1}$ is  
  an effective divisor.

 To arrange for the second claim, it suffices to find some $N\leq n-1$ such that $h^0(D_N)=1$.
 This is ensured as in Lemma \ref{lem:D=D_N}.
% Picking the minimal $N\leq s$ such that $D_N\geq 0$, and setting $D=D_N$,
%we can ensure that $|D|=\{D\}$
%
%
% 
%For this purpose, consider now the following short exact sequence:
%
%  $$ 0 \ra \hol_S (D_{k-1})\ra  \hol_S(D_{k})\ra  \hol_{F'} (D_{k})\ra 0.$$
% Since the cokernel is a line bundle of degree $1$ on the elliptic curve $F'$, 
%  Riemann--Roch gives
%  $h^0( \hol_{F'} (D_{k}))=1, $ 
%  hence $$h^0(  \hol_S(D_{k})) - h^0(  \hol_S(D_{k-1})) \leq 1.$$
%  Therefore there exists a minimal integer $ N \in \ZZ $ such that 
%  $$h^0(  \hol_S(D_{N})) = 1, \  h^0(  \hol_S(D_{N-1}))=0.$$
%as claimed.
 \end{proof}

% \begin{remark}
%We emphasize that $N$ is the minimal integer such that $D_N \geq 0$,
%as we shall use again in Subsection \ref{ss:r=4} and then in Sections \ref{s:2}, \ref{s:3}.
% \end{remark}
%

 Note that $D$ is not irreducible, since we have
 \begin{eqnarray}
 \label{eq:DD} 
 D .\De_j = j, j = 0,1,2, \ D .\sX = 1,  D. A''_1 =1, D. A''_2 = -1, 
 \end{eqnarray}
 while for the other components of the fibre of type $\IV^*$ the intersection number is $0$.
 
 Hence $D $ contains $A''_2  + A'_2 $ as subdivisor, and $D'  : = D - A''_2  - A'_2 $ 
 has intersection numbers all equal to zero with the components of the fibre of type $\IV^*$,
 except for $D'  .A'_2  = D'   .A''_1 =1.$ 
 
 By our choice of the minimality of $N$ the vertical  part of $D' $ does not contain $F'$, and a fortiori 
 no full fibre.
 
 \begin{lemma}\label{ct}
 Either $D' $ is irreducible, 
 or it is a sum of effective divisors $D'   =   D_0 +  \sE + \De +  F''$,
 where $D_0$ is irreducible, $F''$ is a sum of submultiple fibres $F'_j \, (j>1)$, $\sE$ is supported on the fibre of type $\IV^*$, 
 but is not equal to a multiple of this fibre,
 and similarly $\De$  is supported on the fibre of type $\IV$, but is not equal to a multiple of this other fibre.

 \end{lemma} 
 
\begin{proof}[Proof of the Lemma]
 
 Since the surface $S$ has multiple fibres, hence does not have a section, the horizontal component $D_0$ of $D' $ (which is a trisection  different from $\Phi$ (otherwise $v\in\Lambda'$))
  is irreducible. Each other component is vertical, hence it is either a submultiple fibre different from $F'_1$,
  or  is contained in  the two fibres with singular support.
 Here the components (with multiplicities) never sum up to a full fibre by the minimal choice of $N$ in defining $D$.
\end{proof}

\begin{claim}
\label{claim:Psi(D)}
 $\Psi (D_0) \neq D_0$.
 \end{claim}
 
 \subsection*{Conclusion of proof of  Theorem \ref{r=3isotrivial}}
  Given this, since $|D_0| = \{ D_0\}$,
 we conclude that $\Psi (D_0)$ is not linearly equivalent to $D_0$: 
 since $q(S)=0$,
 we can conclude that  $\Psi$ is not cohomologically trivial, thus proving Theorem \ref{r=3isotrivial}.
 \end{proof}

\begin{proof}[Proof of Claim \ref{claim:Psi(D)}]

Assume in fact that  $\Psi (D_0) = D_0$: then the sets $D_0\cap \Delta_j,  D_0 \cap A'_j$ and $ D_0 \cap A''_j$ are  left invariant by $\Psi$. 
It will be sufficient to argue at the fibre of type $\IV$,
but a similar argument works at the $\IV^*$ fibre.

By Lemma \ref{ct}, there are 3 alternatives for the support of $D'$ on the $\IV$ fibre:
\[
\Delta = 0, \;\;\; \Delta = \Delta_0, \;\;\; \Delta=\Delta_0+\Delta_1.
\]

Because, in the cases $\De_1, \De_2, \De_1 + \De_2$ the intersection number of $D_0$ with $\De_0$ would be negative, a contradiction. Similarly we can exclude the case where the support equals $\De_0 + \De_2$, since then the intersection
number of $D_0$ and $\De_1$ would be negative.

In each case, we read off that one component is disjoint from $D_0$ (which therefore  does not meet the node of the fibre)
and the other two fibre components are met with multiplicity one resp.\ two.
Up to symmetry, it thus suffices to study the case from \eqref{eq:DD} with $D'$ and  $D_0$ having the same intersection numbers
$ D' .\De_j = D_0 .\De_j = j$.

It remains to use that $\Psi$ acts on each component $\Delta_j$ with exactly two fixed points,
namely the node and one other point $P_j$, say.
To see this, note that the curve $C_0$ contracted to a point in $S$ is fixed pointwise under $\Psi$
(since it meets the strict transforms of the $\Delta_j$ in three distinct points),
hence the $\Delta_j$ cannot be fixed pointwise (this follows also from an easy local calculation).

From $\Psi (D_0) = D_0$, we thus infer that $D_0\cap\Delta_j = \{P_j\}$ for $j=1,2$,
with tangential intersection at $P_2$ which we shall now use to establish a contradiction.

Take local coordinates $(x, t)$   such that $t$ is the lift of a local coordinate on the base curve $B$, where $\Psi$
acts as the identity.

Hence $\Psi^* (t) = t$, while $\Psi^* (x) = \e x$, for $\e$ a generator of $\mu_3$.

If $D_0$ is smooth and has the same tangent as $\Delta_2 = \{ t=0\}$.
then $D_0 = \{ (x,t) | t = F(x)\}$. By $\Psi$-invariance, $F(x)$ contains only powers of $x$ with exponent
divisible by $3$, hence $(D_0 . \Delta_2 )_{P_2}  \geq 3$, a contradiction.

If instead $D_0$ has a double point, its quadratic part is an eigenvector, hence either $t^2$, or $tx$, or $x^2$.
In the first case the equation of $D_0$, modulo the ideal generated by $t$, contains only powers of $x$ with exponent
divisible by $3$, hence $D_0 . \Delta_2  \geq 3$, the same  contradiction.

In the second case the other terms, modulo $t$, contain $x$ with exponent congruent to $1$ modulo $3$,
and at least $3$, hence $D_0 . \Delta_2  \geq 4$,  contradiction.

In the third case the equation of $D_0$ contains only monomials where $x$ has exponent congruent to $2$ 
modulo $3$, hence the equation is divisible by $x^2$ and $D_0$ is not irreducible and reduced.

Obviously $D_0$ cannot have multiplicity $\geq 3$ at $P_2$, hence,
all in all, we have derived the required contradiction to the assumption $\Psi(D_0)=D_0$; 
this completes the proof of Claim \ref{claim:Psi(D)}.
\end{proof}

We conclude this section with an easy consequence beyond the isotrivial case,
(thus improving on Theorem \ref{thm5} (i) in a special setting):

\begin{cor}
\label{cor:non-ss}
If $S\to B$ is a properly elliptic surface with an additive fibre, then $|\Aut_\ZZ(S)|\leq 3$.
Moreover, equality can only hold if $J(S)=X_{22}$ (and thus $p_g(S)=0$).
\end{cor}

\begin{proof}
If $J(S)$ is not rational, then $\Aut_\ZZ(S)=\Aut_\QQ(S)=\{\id_S\}$
by Corollary \ref{cor: pg>0 add fib 2}.
If $J(S)$ is rational, then $\Aut_\ZZ(S)=\Aut_{B,\ZZ}(S)$ by Corollary \ref{cor: pg=0 Aut_Z}.
Hence the bound $|\Aut_{B,\QQ}(S)|\leq 2$ from Proposition \ref{prop: pg=0 bd J(S)} (2) leaves us with
 the four special isotrivial cases with $J(S)=X_{22}, X_{33}, X_{44}$ or $X_{11}(\lambda)$
for which the upper bound $|\Aut_\ZZ(S)|\leq 3$ of  Theorem \ref{thm:iso} applies.

Then, by Proposition \ref{prop: pg=0 bd J(S)} (3) (b) there remain only the cases $r=6, r=3$. 
But the case $r=3$ is excluded by Theorem \ref{thm:X_44}, hence $r=6$ and we have only the case $J(S)=X_{22}$ as stated.
\end{proof}

\section{Proof of Theorem \ref{thm:abelian}}
\label{s:ab}

If $S$  has $\chi(\mathcal O_S)>0$, 
then the statement of Theorem \ref{thm:abelian} that $\Aut_\QQ(S)$ is abelian and 2-generated
(hence also $\Aut_\ZZ(S)$) follows from the results of this paper:
\begin{itemize}
\item
if $p_g>0$, then the statement follows from Theorem \ref{thm} (i);
\item
if $p_g=0$, this follows from Proposition \ref{prop: pg=0 bd J(S)} (1), (2)
except for the four special isotrivial  surfaces from Table \ref{table:special};
\item
if the fibration is isotrivial, then there is an additive fibre and,  as stated in the proof of 
(2) of Proposition~\ref{prop: pg=0 bd J(S)}, 
 %Lemma~\ref{lem: add fib},
  the homomorphism $\Aut_{\QQ}(S)\rightarrow \Aut_{\QQ}(J(S))$ of Lemma~\ref{lem: Jacobian} is injective.

For the four special surfaces,  $\Aut_{\QQ}(J(S)) \cong \CC^*$ by \cite[Table3]{DM22}, so we are done.
\end{itemize}

If $S$  has $\chi(\mathcal O_S)=0$, we write $S$ as a free quotient of a product,
$S=(C\times E)/G$. 
Then \cite[Thm 1.2]{CFGLS24} states explicitly that 
$\Aut_\ZZ(S)$ is abelian except for a special case where 
$G\cong\ZZ/2m\ZZ$ for odd $m$ acts by translations on $E$ and $\Aut_\ZZ(S)/\Aut^0(S)\cong\ZZ/2\ZZ=\langle\imath\rangle$.
But here, the action of $\Aut^0(S) = E$ is induced by translations on the second factor of $C\times E$
while $\imath$ can only be induced by an involution on the first factor, possibly composed with translation by a 2-torsion point on $E$.
In either case, these automorphisms commute on $C\times E$ and thus also on $S$.
\qed

\begin{remark}
For  $\Aut_\QQ(S)$,
 there exist non-abelian examples in the case $\chi=0$, as one
 can show that any finite subgroup $H$ of $\Aut(\PP^1)$ can be realized as $\Aut_\QQ(S)$.

It suffices to define $C$ as  the fibre product of $B:= \PP^1 \ra B' : = \PP^1 /H$ and of a general ramified double cover 
$ C' \ra B' = (\PP^1 /H)$;
letting $\iota$ be the involution of $C$  induced by the identity  of $B$ and the covering  involution of $C'$,
$H \times \langle \iota \rangle $ acts on $C$. We let $E = \CC/\ZZ \oplus \tau \ZZ$ be an elliptic curve, and $S : = (C \times E)/(\ZZ/2\ZZ)$,
where $\ZZ/2\ZZ$ acts freely via $(x,z) \mapsto (\iota(x), z + 1/2)$. Then $H\times  \{\id_E\}$ induces an inclusion $H < \Aut_\QQ(S)$  by (1) of Lemma \ref{lem:Psi_Y}.
\end{remark}

\subsection*{Acknowledgements} 

We thank John Voight for start-up aid regarding Magma. The second author would like to thank Professor Jin-Xing Cai for many helpful discussions on the topic of numerically trivial automorphisms.
We are grateful to the referee for numerous comments and corrections which helped us improve the paper.

\end{document}